\documentclass[11pt]{article}
\usepackage{amsmath,amssymb,amscd,amsthm}
\usepackage{xcolor}
\usepackage{multirow}
\newdimen\unit\newdimen\psep\newcount\nd\newcount\ndx\newbox\dotb\newbox\ptbox
\newdimen\dx\newdimen\dy\newdimen\dxx\newdimen\dyy\newdimen\hgt
\newdimen\xoff\newdimen\yoff
\newcommand\clap[1]{\hbox to 0pt{\hss{#1}\hss}}
\newcommand\vdisk[1]{{\font\dotf=cmr10 scaled #1\dotf.}}
\newcommand\varline[2]{\setbox\dotb\hbox{\vdisk{#1}}\xoff=-.5\wd\dotb
\wd\dotb=0pt\yoff=-.5\ht\dotb\psep=#2\ht\dotb}
\newcommand\varpt[1]{\setbox\ptbox\clap{\vdisk{#1}}\setbox\ptbox
\hbox{\raise-.5\ht\ptbox\box\ptbox}}
\newcommand\cpt{\copy\ptbox}
\newcommand\point[3]{\rlap{\kern#1\unit\raise#2\unit\hbox{#3}}}
\newcommand\setnd[4]{\dx=#3\unit\advance\dx-#1\unit\divide\dx by\psep
\dy=#4\unit\advance\dy-#2\unit\divide\dy by\psep \multiply\dx
by\dx\multiply\dy by\dy\advance\dx\dy\nd=1\advance\dx-1sp
\loop\ifnum\dx>0\advance\dx-\nd sp\advance\nd1\advance\dx-\nd
sp\repeat}
\newcommand\dl[4]{{\setnd{#1}{#2}{#3}{#4}\dline{#1}{#2}{#3}{#4}\nd}}
\newcommand\dline[5]{{\nd=#5\hgt=#2\unit\dx=#3\unit\advance\dx-#1\unit
\divide\dx by\nd\dy=#4\unit\advance\dy-#2\unit\divide\dy by\nd
\advance\hgt\yoff\rlap{\kern#1\unit\kern\xoff\loop\ifnum\nd>1\advance\nd-1
\advance\hgt\dy\kern\dx\raise\hgt\copy\dotb\repeat}}}
\newcommand\ellipse[4]{\qellip{#1}{#2}{#3}{#4}\qellip{#1}{#2}{#3}{-#4}%
\qellip{#1}{#2}{-#3}{#4}\qellip{#1}{#2}{-#3}{-#4}}
\newcommand\qellip[4]{{\setnd{0}{0}{#3}{#4}\dx=\unit\dy=0pt\raise\yoff\rlap{%
\kern#1\unit\kern\xoff\raise#2\unit\hbox{\loop\ifnum\dx>0\rlap{\kern#3\dx
\raise#4\dy\copy\dotb}\hgt=\dx\divide\hgt
by\nd\advance\dy\hgt\hgt=\dy \divide\hgt
by\nd\advance\dx-\hgt\repeat\rlap{\raise#4\dy\copy\dotb}}}}}
\newcommand\bez[6]{{\setnd{#1}{#2}{#3}{#4}\ndx=\nd\setnd{#3}{#4}{#5}{#6}
\ifnum\ndx>\nd\nd=\ndx\fi\dx=#3\unit\advance\dx-#1\unit\dy=#4\unit
\advance\dy-#2\unit\dxx=#5\unit\advance\dxx-#1\unit\dyy=#6\unit\advance
\dyy-#2\unit\advance\dxx-2\dx\advance\dyy-2\dy\divide\dxx
by\nd\divide\dyy
by\nd\advance\dx.25\dxx\advance\dy.25\dyy\divide\dx
by\nd\divide\dy by\nd \multiply\nd
by2\dx=100\dx\dy=100\dy\dxx=100\dxx\dyy=100\dyy\divide\dxx by\nd
\divide\dyy
by\nd\hgt=#2\unit\raise\yoff\rlap{\kern#1\unit\kern\xoff
\raise\hgt\copy\dotb\loop\ifnum\nd>0\advance\nd-1\advance\hgt0.01\dy
\kern0.01\dx\raise\hgt\copy\dotb\advance\dx\dxx\advance\dy\dyy\repeat}}}
\newcommand\ptu[3]{\point{#1}{#2}{\cpt\raise1ex\clap{$\scriptstyle{#3}$}}}
\newcommand\ptd[3]{\point{#1}{#2}{\cpt\raise-1.8ex\clap{$\scriptstyle{#3}$}}}
\newcommand\ptr[3]{\point{#1}{#2}{\cpt\raise-.4ex\rlap{$\ \scriptstyle{#3}$}}}
\newcommand\ptl[3]{\point{#1}{#2}{\cpt\raise-.4ex\llap{$\scriptstyle{#3}\ $}}}
\newcommand\ptlu[3]{\point{#1}{#2}{\raise.8ex\clap{$\scriptstyle{#3}$}}}
\newcommand\ptld[3]{\point{#1}{#2}{\raise-1.6ex\clap{$\scriptstyle{#3}$}}}
\newcommand\ptlr[3]{\point{#1}{#2}{\raise-.4ex\rlap{$\,\scriptstyle{#3}$}}}
\newcommand\ptll[3]{\point{#1}{#2}{\raise-.4ex\llap{$\scriptstyle{#3}\,$}}}
\newcommand\pt[2]{\point{#1}{#2}{\cpt}}

\newcommand\thnline{\varline{400}{.6}}
\newcommand\dotline{\varline{1000}{4}}
\varpt{2500}\thnline\unit=2em

\newtheorem{theorem}                   {Theorem}%[section]
\newtheorem{thm}             [theorem] {Theorem}

\newtheorem{lem}             [theorem] {Lemma}

\newtheorem{cor}             [theorem] {Corollary}

\newtheorem{prop}             [theorem] {Proposition}

\newtheorem{prob}            [theorem] {Problem}

\def\lrF{\overset{\textup{\hspace{0.05cm}\tiny$\leftrightarrow$}}{\phantom{\in}}\hspace{-0.34cm}F}
\def\lrG{\overset{\textup{\hspace{0.02cm}\tiny$\leftrightarrow$}}{\phantom{\in}}\hspace{-0.3cm}G}
\def\lrH{\overset{\textup{\hspace{0.05cm}\tiny$\leftrightarrow$}}{\phantom{\in}}\hspace{-0.34cm}H}
\def\lrWn{\overset{\textup{\hspace{0.09cm}\tiny$\leftrightarrow$}}{\phantom{\in}}\hspace{-0.34cm}W_n}
\def\rC{\overset{\textup{\hspace{0.05cm}\tiny$\rightarrow$}}{\phantom{\in}}\hspace{-0.32cm}C}
\def\rP{\overset{\textup{\hspace{0.05cm}\tiny$\rightarrow$}}{\phantom{\in}}\hspace{-0.32cm}P}
\def\lrP{\overset{\textup{\hspace{0.05cm}\tiny$\leftrightarrow$}}{\phantom{\in}}\hspace{-0.32cm}P}
\def\lrK{\overset{\textup{\hspace{0.08cm}\tiny$\leftrightarrow$}}{\phantom{\in}}\hspace{-0.34cm}K}
\def\lrC{\overset{\textup{\hspace{0.05cm}\tiny$\leftrightarrow$}}{\phantom{\in}}\hspace{-0.32cm}C}
\def\lrPten{\overset{\textup{\tiny$\leftrightarrow$}}{\phantom{\in}}\hspace{-0.25cm}{\sf P}_{10}}

\title{Total rainbow connection of digraphs}

\author{
Hui Lei$^1$, Henry Liu$^2$\footnote{Corresponding author}\,\:, Colton Magnant$^3$, Yongtang Shi$^1$\\
\\
\normalsize $^1$Center for Combinatorics and LPMC\\
\normalsize Nankai University, Tianjin 300071, China\\
\normalsize leihui0711@163.com, shi@nankai.edu.cn\\
\\
\normalsize $^2$School of Mathematics\\
\normalsize Sun Yat-sen University, Guangzhou 510275, China\\
\normalsize liaozhx5@mail.sysu.edu.cn\\
\\
\normalsize $^3$Department of Mathematical Sciences\\
\normalsize Georgia Southern University, Statesboro, GA 30460-8093, USA\\
\normalsize cmagnant@georgiasouthern.edu\\
}
\date{2 November 2017}

\begin{document}
\maketitle

\begin{abstract}
\noindent An edge-coloured path is \emph{rainbow} if its edges have distinct colours. For a connected graph $G$, the \emph{rainbow connection number} (resp.~\emph{strong rainbow connection number}) of $G$ is the minimum number of colours required to colour the edges of $G$ so that any two vertices of $G$ are connected by a rainbow path (resp.~rainbow geodesic). These two graph parameters were introduced by Chartrand, Johns, McKeon, and Zhang in 2008. Krivelevich and Yuster generalised this concept to the vertex-coloured setting. Similarly, Liu, Mestre, and Sousa introduced the version which involves total-colourings.

Dorbec, Schiermeyer, Sidorowicz, and Sopena extended the concept of the rainbow connection to digraphs. In this paper, we consider the (strong) total rainbow connection number of digraphs. Results on the (strong) total rainbow connection number of biorientations of graphs, tournaments, and cactus digraphs are presented.
\end{abstract}

\section{Introduction}

All graphs and digraphs considered in this paper are finite and simple. That is, we do not allow the existence of loops, multiple edges (for graphs), and multiple directed arcs (for digraphs). We follow the terminology and notation of Bollob\'as \cite{BB98} for those not defined here.

The concept of rainbow connection in graphs was introduced by Chartrand, Johns, McKeon, and Zhang \cite{CJMZ2008}. An edge-coloured path is \emph{rainbow} if its edges have distinct colours. An edge-colouring of a connected graph $G$ is \emph{rainbow connected} if any two vertices of $G$ are connected by a rainbow path. The \emph{rainbow connection number} of $G$, denoted by $rc(G)$, is the minimum number of colours in a rainbow connected edge-colouring of $G$. An edge-colouring of $G$ is \emph{strongly rainbow connected} if for every pair of vertices $u$ and $v$, there exists a rainbow $u-v$ geodesic, i.e., a path of length equal to the distance between $u$ and $v$. The minimum number of colours in a strongly rainbow connected edge-colouring of $G$ is the \emph{strong rainbow connection number} of $G$, denoted by $src(G)$.

As a natural counterpart to the rainbow connection of edge-coloured graphs, Krivelevich and Yuster \cite{KY2010}; and Li, Mao, and Shi \cite{LMS2012}, proposed the concept of (strong) rainbow vertex-connection. A vertex-coloured path is \emph{vertex-rainbow} if its internal vertices have distinct colours. A vertex-colouring of a connected graph $G$ is \emph{rainbow vertex-connected} (resp.~\emph{strongly rainbow vertex-connected}) if any two vertices of $G$ are connected by a vertex-rainbow path (resp.~geodesic). The \emph{rainbow vertex-connection number} of $G$, denoted by $rvc(G)$, is the minimum number of colours in a rainbow vertex-connected vertex-colouring of $G$. The \emph{strong rainbow vertex-connection number} of $G$, denoted by $srvc(G)$, is the minimum number of colours in a strongly rainbow vertex-connected vertex-colouring of $G$. %For more results on rainbow vertex-connection, we refer to \cite{LS2013}. 
We refer the reader to the survey \cite{LSS2013} and the monograph \cite{LS2012} on the subject of rainbow connection in graphs.

Liu, Mestre, and Sousa \cite{LMS2014}; and Chen, Li, Liu, and Liu \cite{CLLL2016}, proposed the concept of (strong) total rainbow connection. A total-coloured path is \emph{total-rainbow} if its edges and internal vertices have distinct colours. A total-colouring of a connected graph $G$ is \emph{total rainbow connected} (resp.~\emph{strongly total rainbow connected}) if any two vertices are connected by a total-rainbow path (resp.~geodesic). The \emph{total rainbow connection number} of $G$, denoted by $trc(G)$, is the minimum number of colours in a total rainbow connected total-colouring of $G$. The \emph{strong total rainbow connection number} of $G$, denoted by $strc(G)$, is the minimum number of colours in a strongly total rainbow connected total-colouring of $G$.

In \cite{PIE2014}, Dorbec, Schiermeyer, Sidorowicz, and Sopena introduced the concept of rainbow connection of digraphs. A \emph{directed path}, or simply a \emph{path} $P$, is a digraph consisting of a sequence of vertices $v_0,v_1,\dots,v_\ell$ and arcs $v_{i-1}v_i$  for $1\le i\le \ell$. We also say that $P$ is a \emph{$v_0-v_\ell$ path}, and its \emph{length} is the number of arcs $\ell$. A digraph $D$ is \emph{strongly connected} if for any ordered pair of vertices $(u,v)$ in $D$, there exists a $u-v$ path. An arc-coloured path is \emph{rainbow} if its arcs have distinct colours. Let $D$ be a strongly connected digraph. An arc-colouring of $D$ is \emph{rainbow connected} if for any ordered pair of vertices $(u,v)$ in $D$, there is a rainbow $u - v$ path. The \emph{rainbow connection number} of $D$, denoted by $\overset{\rightarrow}{rc}(D)$, is the minimum number of colours in a rainbow connected arc-colouring of $D$. Alva-Samos and Montellano-Ballesteros \cite{AM2016} then introduced the notion of strong rainbow connection of digraphs. An arc-colouring of $D$ is \emph{strongly rainbow connected} if for any ordered pair of vertices $(u, v)$, there is a rainbow $u-v$ geodesic, i.e., a rainbow $u-v$ path of minimum length. The \emph{strong rainbow connection number} of $D$, denoted by $\overset{\rightarrow}{src}(D)$, is the smallest possible number of colours in a strongly rainbow connected arc-colouring of $D$. We have $\textup{diam}(D)\le\overset{\rightarrow}{rc}(D) \le \overset{\rightarrow}{src}(D)$, where $\textup{diam}(D)$ denotes the diameter of $D$. Subsequently, there have been some results on this topic, which considered many different classes of digraphs \cite{AM2017, HMS2014,SS2016}. Very recently, Lei, Li, Liu, and Shi \cite{LLLS2017} introduced the (strong) rainbow vertex-connection of digraphs. A vertex-coloured directed path is \emph{vertex-rainbow} if its internal vertices have distinct colours.  A vertex-colouring of $D$ is \emph{rainbow vertex-connected} (resp.~\emph{strongly rainbow vertex-connected}) if for  any ordered pair of vertices $(u,v)$ in $D$, there exists a vertex-rainbow $u-v$ path (resp.~geodesic). The \emph{rainbow vertex-connection number} of $D$, denoted by $\overset{\rightarrow}{rvc}(D)$, is the minimum number of colours in a rainbow vertex-connected vertex-colouring of $D$. The \emph{strong rainbow vertex-connection number} of $D$, denoted by $\overset{\rightarrow}{srvc}(D)$, is the minimum number of colours in a strongly rainbow vertex-connected vertex-colouring of $D$. We have $\textup{diam}(D)-1\le\overset{\rightarrow}{rvc}(D) \le \overset{\rightarrow}{srvc}(D)$.

In this paper, we introduce the concept of total rainbow connection of digraphs. Let $D$ be a strongly connected digraph. A total-coloured directed path is \emph{total-rainbow} if its arcs and internal vertices have distinct colours. A total-colouring of $D$ is \emph{total rainbow connected} if for any ordered pair of vertices $(u,v)$ in $D$, there exists a total-rainbow $u-v$ path. The \emph{total rainbow connection number} $D$, denoted by $\overset{\rightarrow}{\smash{t}rc}(D)$, is the minimum number of colours in a total rainbow connected total-colouring of $D$. Likewise, a total-colouring of $D$ is \emph{strongly total rainbow connected} if for any ordered pair of vertices $(u,v)$, there exists a total-rainbow $u-v$ geodesic. The \emph{strong total rainbow connection number} of $D$, denoted by $\overset{\rightarrow}{s\smash{t}rc}(D)$, is the minimum number of colours in a strongly total rainbow connected total-colouring of $D$. 

This paper is organised as follows. In Section \ref{gensect}, we present several general results about the parameters $\overset{\rightarrow}{\smash{t}rc}(D)$ and $\overset{\rightarrow}{s\smash{t}rc}(D)$, as well as their relationships to the parameters $\overset{\rightarrow}{rc}(D),\overset{\rightarrow}{src}(D),\overset{\rightarrow}{rvc}(D)$, and $\overset{\rightarrow}{srvc}(D)$. In Section \ref{specsect}, we compute the parameters $\overset{\rightarrow}{\smash{t}rc}(D)$ and $\overset{\rightarrow}{s\smash{t}rc}(D)$ for some specific digraphs $D$. In Section \ref{tournamentssect}, we study the parameters $\overset{\rightarrow}{\smash{t}rc}(T)$ and $\overset{\rightarrow}{s\smash{t}rc}(T)$ for tournaments $T$. Finally in Section \ref{cactussect}, we consider the parameters $\overset{\rightarrow}{rvc}(Q)$ and $\overset{\rightarrow}{\smash{t}rc}(Q)$ for cactus digraphs $Q$.

\section{Definitions, remarks, and results for general digraphs}\label{gensect}

We begin with some definitions about digraphs. For a digraph $D$, its vertex and arc sets are denoted by $V(D)$ and $A(D)$. For an arc $uv\in A(D)$, we say that $v$ is an \emph{out-neighbour} of $u$, and $u$ is an \emph{in-neighbour} of $v$. Moreover, we call $uv$ an \emph{in-arc} of $v$ and an \emph{out-arc} of $u$. We denote the set of \emph{out-neighbours} (resp.~\emph{in-neighbours}) of $u$ in $D$ by $\Gamma^+(u)$ (resp.~$\Gamma^-(u)$). Let $\Gamma[u]=\Gamma^+(u)\cup \Gamma^-(u)\cup\{u\}$. For a strongly connected digraph $D$, and $u,v\in V(D)$, the distance from $u$ to $v$ (i.e., the length of a shortest $u - v$ path) in $D$ is denoted by $d(u,v)$, or $d_D(u,v)$ if we wish to emphasise that the distance is taken in the digraph $D$. Let $\textup{diam}(D)$ denote the diameter of $D$. 

If $uv, vu\in A(D)$, then we say that $uv $ and $vu$ are \emph{symmetric arcs}. If $uv\in A(D)$ and $vu\notin A(D)$, then $uv$ is an {\it asymmetric arc}. The digraph $D$ is an \emph{oriented graph} if every arc of $D$ is asymmetric. A \emph{tournament} is an oriented graph where every two vertices have one asymmetric arc joining them. A \emph{cactus} is a strongly connected oriented graph where each arc belongs to exactly one directed cycle. Given a graph $G$, its \emph{biorientation} is the digraph $\lrG$ obtained by replacing each edge $uv$ of $G$ by the pair of symmetric arcs $uv$ and $vu$. Let $\rP_n$ and $\rC_n$ denote the directed path and directed cycle of order $n$, respectively (where $n\ge 3$ for $\rC_n$), i.e., we may let $V(\rP_n)=V(\rC_n)=\{v_{0},\ldots,v_{n-1}\}$, and $A(\rP_n)=\{v_{0}v_{1}, v_{1}v_{2}\ldots, v_{n-2}v_{n-1}\}$ and $A(\rC_n)=A(\rP_n)\cup v_{n-1}v_{0}$. If $C$ is a directed cycle and $u,v\in V(C)$, we write $uCv$ for the unique $u-v$ directed path in $C$.

For a subset $X\subset V(D)$, we denote by $D[X]$ the subdigraph of $D$ induced by $X$. Given two digraphs $D$ and $H$, and $u\in V(D)$, we define $D_{u\to H}$ to be the digraph obtained from $D$ and $H$ by replacing the vertex $u$ by a copy of $H$, and replacing each arc $xu$ (resp.~$ux$) in $D$ by all the arcs $xv$ (resp.~$vx$) for $v\in V(H)$. We say that $D_{u\to H}$ is \emph{obtained from $D$ by expanding $u$ into $H$}. Note that the digraph obtained from $D$ by expanding every vertex into $H$ is also known as the \emph{lexicographic product $D\circ H$}.

Now, we shall present some remarks and basic results for the total rainbow connection and strong total rainbow connection numbers, for general digraphs and biorientations of graphs. We first note that in a total rainbow connected colouring of a strongly connected digraph $D$, there must be a path between some two vertices with at least $2\,\textup{diam}(D)-1$ colours. Thus, we have the following proposition.

\begin{prop}\label{prop1}
Let $D$ be a strongly connected digraph with $n$ vertices and $m$ arcs. Then
\begin{equation}
2\,\textup{diam}(D)-1 \leq \overset{\rightarrow}{\smash{t}rc}(D)\leq s\overset{\rightarrow}{\smash{t}rc}(D)\leq n+m.\label{eq}
\end{equation}
\end{prop}

It is easy to see that the bioriented paths $\lrP_n$, for $n\geq2$, form an infinite family of graphs where we have equalities in the first two inequalities in (\ref{eq}). Also, it is not difficult to see that for every directed cycle $\rC_n$ with $n\geq5$, we have equalities in the last two inequalities in (\ref{eq}). These two results will be included in Theorems \ref{thm2} and \ref{thm3}.

In the next result, we give equivalences and implications between several conditions, when the rainbow connection parameters are small.

\begin{thm}\label{pro2}
Let $D$ be a non-trivial, strongly connected digraph.
\begin{enumerate}
\item[(a)] The following are equivalent.
\begin{enumerate}
\item[(i)] $D$ is a bioriented complete graph.
\item[(ii)] \textup{diam}$(D)=1$.
\item[(iii)] $\overset{\rightarrow}{rc}(D)=1$.
\item[(iv)] $\overset{\rightarrow}{src}(D)=1$.
\item[(v)] $\overset{\rightarrow}{rvc}(D)=0$.
\item[(vi)] $\overset{\rightarrow}{srvc}(D)=0$.
\item[(vii)] $\overset{\rightarrow}{\smash{t}rc}(D)=1$.
\item[(viii)] $\overset{\rightarrow}{s\smash{t}rc}(D)=1$.
\end{enumerate}
\item[(b)] $\overset{\rightarrow}{s\smash{t}rc}(D)\ge\overset{\rightarrow}{\smash{t}rc}(D)\ge 3$ if and only if $D$ is not a bioriented complete graph.
\item[(c)] 
\begin{enumerate}
\item[(i)] $\overset{\rightarrow}{rc}(D)=2$ if and only if $\overset{\rightarrow}{src}(D)=2$.
\item[(ii)] $\overset{\rightarrow}{rvc}(D) = 1$, if and only if $\overset{\rightarrow}{srvc}(D) = 1$, if and only if \textup{diam}$(D)=2$. 
\item[(iii)] $\overset{\rightarrow}{rvc}(D) = 2$ if and only if $\overset{\rightarrow}{srvc}(D) = 2$.
\item[(iv)] $\overset{\rightarrow}{\smash{t}rc}(D)=3$ if and only if $\overset{\rightarrow}{s\smash{t}rc}(D)=3$.
\item[(v)] $\overset{\rightarrow}{\smash{t}rc}(D)=4$ if and only if $\overset{\rightarrow}{s\smash{t}rc}(D)=4$.
\end{enumerate}
Moreover, any of the conditions in (i) implies any of the conditions in (iv), and any of the conditions in (i), (iv) and (v) implies any of the conditions in (ii).
\end{enumerate}
\end{thm}

\begin{proof}
(a) In \cite{AM2016}, Theorem 2; and \cite{LLLS2017}, Theorem 2, it was proved that (i) to (vi) are equivalent. Now clearly, we have (i) $\Rightarrow$ (viii), and by (\ref{eq}), it is easy to see that (viii) $\Rightarrow$ (vii) $\Rightarrow$ (ii).\\[1ex]
\indent (b) If $D$ is not a bioriented complete graph, then $\textup{diam}(D)\ge 2$, and $\overset{\rightarrow}{s\smash{t}rc}(D)\ge\overset{\rightarrow}{\smash{t}rc}(D)\ge 3$ follows from (\ref{eq}). The converse clearly holds by (a).\\[1ex]
\indent (c) Part (i) was proved in \cite{AM2016}, Theorem 2; and parts (ii) and (iii) were proved in \cite{LLLS2017}, Theorem 2. We prove parts (iv) and (v). Suppose first that $\overset{\rightarrow}{\smash{t}rc}(D) = 3$. Then (\ref{eq}) implies $\overset{\rightarrow}{s\smash{t}rc}(D) \ge 3$. Also, there exists a total rainbow connected colouring for $D$, using 
$\overset{\rightarrow}{\smash{t}rc}(D) = 3$ colours. In such a total-colouring, for any $x,y\in V(D)$, either $xy\in A(D)$, or $xy\not\in A(D)$ and there is a total-rainbow $x-y$ path of length $2$, which is also a total-rainbow $x-y$ geodesic. Thus $\overset{\rightarrow}{s\smash{t}rc}(D) \le 3$, so that $\overset{\rightarrow}{s\smash{t}rc}(D) = 3$, and the first condition of (iv) implies the second. The same argument, with the role of ``$3$'' being replaced by ``$4$'', gives that the first condition of (v) implies the second.

Next, suppose that $\overset{\rightarrow}{s\smash{t}rc}(D) = 3$. Then by (a), we know that $D$ is not a bioriented complete graph. By (b), we have $3\le \overset{\rightarrow}{\smash{t}rc}(D)\le\overset{\rightarrow}{s\smash{t}rc}(D)=3$, and hence $\overset{\rightarrow}{\smash{t}rc}(D) = 3$, and (iv) is proved. Finally, suppose that $\overset{\rightarrow}{s\smash{t}rc}(D) = 4$. By (a) and (b), we have $3\le \overset{\rightarrow}{\smash{t}rc}(D)\le \overset{\rightarrow}{s\smash{t}rc}(D) = 4$. By (c)(iv), we have $\overset{\rightarrow}{\smash{t}rc}(D) \neq 3$, so that $\overset{\rightarrow}{\smash{t}rc}(D) = 4$. This completes the proof of (v).

Now, we consider the final part of (c). Firstly, suppose that either condition in (i) holds, so that $\overset{\rightarrow}{rc}(D)=2$. Then (a) implies $\textup{diam}(D)\ge 2$, and (\ref{eq}) implies $\overset{\rightarrow}{\smash{t}rc}(D) \ge 3$. Moreover, there exists a rainbow connected arc-colouring for $D$, using $\overset{\rightarrow}{rc}(D) = 2$ colours. Clearly by colouring all vertices of $D$ with a third colour, we have a total rainbow connected colouring for $D$, using $3$ colours. Thus, $\overset{\rightarrow}{\smash{t}rc}(D) \le 3$. We have $\overset{\rightarrow}{\smash{t}rc}(D) = 3$, and thus both conditions of (iv) hold. Secondly, suppose that either of the conditions in (i) holds. Then we have $\textup{diam}(D)\le \overset{\rightarrow}{rc}(D)=2$, and (a) implies that $\textup{diam}(D) = 2$. Similarly, if any of the conditions in (iv) or (v) holds, then it is easy to use (\ref{eq}) and (a) to again obtain $\textup{diam}(D) = 2$. Thus, any of the conditions in (i), (iv) and (v) implies any of the conditions in (ii).
\end{proof}

We remark that in Theorem \ref{pro2}(c), no other implication exists between the conditions of (i) to (v). Obviously, the conditions of (ii) and those of (iii) are mutually exclusive. Thus by the last part of (c), the conditions of (iii) are mutually exclusive to those of (i), (iv) and (v). Similarly, the conditions of (iv) and those of (v) are mutually exclusive, and thus the conditions of (i) and those of (v) are also mutually exclusive, since the conditions of (i) imply those of (iv). Clearly, the example of the bioriented stars $\lrK_{1,n}$ shows that there are infinitely many digraphs where the conditions of (ii) hold, but those of (v) do not hold. Indeed, for $n\ge 2$, we have $\overset{\rightarrow}{rvc}(\lrK_{1,n})=1$, while $\overset{\rightarrow}{\smash{t}rc}(\lrK_{1,n})=3$. Also, there are infinitely many digraphs $D$ such that the conditions of (ii) and (iv) hold, but those of (i) do not hold. For example, let $u$ be a vertex of the directed cycle $\rC_3$, and let $D$ be a digraph obtained by expanding $u$ into a bioriented clique $K$. That is, $D=(\rC_3)_{u\to K}$. Then, we have $\overset{\rightarrow}{rvc}(D)=1$ and $\overset{\rightarrow}{\smash{t}rc}(D)=3$, but $\overset{\rightarrow}{rc}(D)=3$. In the following lemma, we will see that there are infinitely many examples of digraphs $D$ where the conditions of (ii) hold, but the conditions of (iv) do not hold.

\begin{lem}\label{petersenlem}
There exist infinitely many digraphs $D$ with $\textup{diam}(D)=2$ and $\overset{\rightarrow}{\smash{t}rc}(D)=\overset{\rightarrow}{s\smash{t}rc}(D)=4$.
\end{lem}

\begin{proof}
We will construct the digraphs $D$ from the Petersen graph ${\sf P}_{10}$. We first consider the digraph $\lrPten$. See Figure 1(a), where we have $V({\sf P}_{10})=V(\lrPten)=\{u_i, v_i:0\le i\le 4\}$. Clearly, we have $\textup{diam}(\lrPten)=2$. We will show that $\overset{\rightarrow}{\smash{t}rc}(\lrPten)=\overset{\rightarrow}{s\smash{t}rc}(\lrPten)=4$. By (\ref{eq}), it suffices to prove that $\overset{\rightarrow}{\smash{t}rc}(\lrPten)\ge 4$ and $\overset{\rightarrow}{s\smash{t}rc}(\lrPten)\le 4$.\\[1ex]
\[ \unit = 1cm
%left
\pt{ -3 }{ 2 }
\pt{ -1.09788696741 }{ 0.61803398875 }
\pt{ -1.82442949542 }{ -1.61803398875 }
\pt{ -4.17557050458 }{ -1.61803398875 }
\pt{ -4.90211303259 }{ 0.61803398875 }
\pt{ -3 }{ 1 }
\pt{ -2.0489434837 }{ 0.309016994375 }
\pt{ -2.41221474771 }{ -0.809016994375 }
\pt{ -3.58778525229 }{ -0.809016994375 }
\pt{ -3.9510565163 }{ 0.309016994375 }
\dl{ -3 }{ 2 }{ -1.09788696741 }{ 0.61803398875 }
\dl{ -1.09788696741 }{ 0.61803398875 }{ -1.82442949542 }{ -1.61803398875 }
\dl{ -1.82442949542 }{ -1.61803398875 }{ -4.17557050458 }{ -1.61803398875 }
\dl{ -4.17557050458 }{ -1.61803398875 }{ -4.90211303259 }{ 0.61803398875 }
\dl{ -4.90211303259 }{ 0.61803398875 }{ -3 }{ 2 }
\dl{ -3 }{ 2 }{ -3 }{ 1 }
\dl{ -1.09788696741 }{ 0.61803398875 }{ -2.0489434837 }{ 0.309016994375 }
\dl{ -1.82442949542 }{ -1.61803398875 }{ -2.41221474771 }{ -0.809016994375 }
\dl{ -4.17557050458 }{ -1.61803398875 }{ -3.58778525229 }{ -0.809016994375 }
\dl{ -4.90211303259 }{ 0.61803398875 }{ -3.9510565163 }{ 0.309016994375 }
\dl{ -3 }{ 1 }{ -2.41221474771 }{ -0.809016994375 }
\dl{ -2.41221474771 }{ -0.809016994375 }{ -3.9510565163 }{ 0.309016994375 }
\dl{ -3.9510565163 }{ 0.309016994375 }{ -2.0489434837 }{ 0.309016994375 }
\dl{ -2.0489434837 }{ 0.309016994375 }{ -3.58778525229 }{ -0.809016994375 }
\dl{ -3.58778525229 }{ -0.809016994375 }{ -3 }{ 1 }
\point{-3.24}{-2.5}{\small (a)}
%vertexlabels
\point{-3.19}{2.17}{\small $u_0$}\point{-1}{0.57}{\small $u_1$}\point{-2}{-1.9}{\small $u_2$}\point{-4.37}{-1.9}{\small $u_3$}\point{-5.35}{0.57}{\small $u_4$}
\point{-2.93}{1.05}{\small $v_0$}\point{-2.1}{0.03}{\small $v_1$}\point{-2.7}{-1.03}{\small $v_2$}\point{-3.56}{-1.03}{\small $v_3$}\point{-4.18}{0.03}{\small $v_4$}
%right
\pt{ 3 }{ 2 }
\pt{ 4.90211303259 }{ 0.61803398875 }
\pt{ 4.17557050458 }{ -1.61803398875 }
\pt{ 1.82442949542 }{ -1.61803398875 }
\pt{ 1.09788696741 }{ 0.61803398875 }
\pt{ 3 }{ 1 }
\pt{ 3.9510565163 }{ 0.309016994375 }
\pt{ 3.58778525229 }{ -0.809016994375 }
\pt{ 2.41221474771 }{ -0.809016994375 }
\pt{ 2.0489434837 }{ 0.309016994375 }
\dl{ 3 }{ 2 }{ 4.90211303259 }{ 0.61803398875 }
\dl{ 4.90211303259 }{ 0.61803398875 }{ 4.17557050458 }{ -1.61803398875 }
\dl{ 4.17557050458 }{ -1.61803398875 }{ 1.82442949542 }{ -1.61803398875 }
\dl{ 1.82442949542 }{ -1.61803398875 }{ 1.09788696741 }{ 0.61803398875 }
\dl{ 1.09788696741 }{ 0.61803398875 }{ 3 }{ 2 }
\dl{ 3 }{ 2 }{ 3 }{ 1 }
\dl{ 4.90211303259 }{ 0.61803398875 }{ 3.9510565163 }{ 0.309016994375 }
\dl{ 4.17557050458 }{ -1.61803398875 }{ 3.58778525229 }{ -0.809016994375 }
\dl{ 1.82442949542 }{ -1.61803398875 }{ 2.41221474771 }{ -0.809016994375 }
\dl{ 1.09788696741 }{ 0.61803398875 }{ 2.0489434837 }{ 0.309016994375 }
\dl{ 3 }{ 1 }{ 3.58778525229 }{ -0.809016994375 }
\dl{ 3.58778525229 }{ -0.809016994375 }{ 2.0489434837 }{ 0.309016994375 }
\dl{ 2.0489434837 }{ 0.309016994375 }{ 3.9510565163 }{ 0.309016994375 }
\dl{ 3.9510565163 }{ 0.309016994375 }{ 2.41221474771 }{ -0.809016994375 }
\dl{ 2.41221474771 }{ -0.809016994375 }{ 3 }{ 1 }
\point{2.8}{-2.5}{\small (b)}
%colours
\ptlu{3}{2}{1}\ptlr{4.93}{0.61}{3}\ptld{4.2}{-1.65}{2}\ptld{1.8}{-1.65}{2}\ptll{1.04}{0.61}{2}
\ptlr{3.01}{1.05}{1}\ptld{3.97}{0.28}{2}\ptld{3.48}{-0.78}{3}\ptld{2.52}{-0.78}{2}\ptld{2.03}{0.28}{2}
\ptlu{3}{0.23}{3}\ptll{3.54}{0.1}{2}\ptlr{2.46}{0.1}{3}\ptlu{3.17}{-0.3}{1}\ptlu{2.83}{-0.3}{1}
\ptlu{4.05}{1.25}{2}\ptlu{1.95}{1.25}{3}\ptlr{4.55}{-0.58}{1}\ptll{1.45}{-0.58}{1}\ptld{3}{-1.65}{3}
\ptll{2.99}{1.45}{4}\ptlu{4.3}{0.38}{4}\ptlu{1.7}{0.38}{4}\ptlr{3.85}{-1.1}{4}\ptll{2.15}{-1.1}{4}
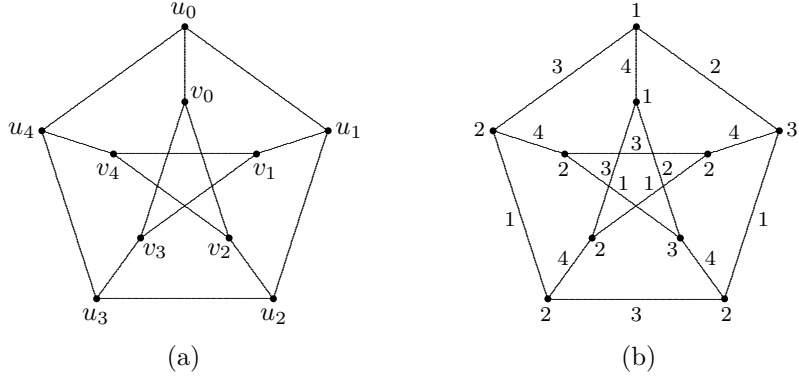
\ptlu{0}{-3.5}{\textup{\small Figure 1. (a) The Petersen graph ${\sf P}_{10}$; (b) A strongly total rainbow connected colouring for ${\sf P}_{10}$.}}
\]\\[-3ex]

Suppose first that there is a total rainbow connected colouring $c$ of $\lrPten$, using at most three colours, say $1,2,3$. Then whenever we have $x,y\in V(\lrPten)$ and $xy\not\in A(\lrPten)$, there must be a total-rainbow $x-y$ path of length $2$ in $\lrPten$. By considering the unique paths of length $2$ between the vertices $u_0,u_2,v_1$, we may assume that $c(u_0u_1)=c(u_2u_1)=c(v_1u_1)=1$, $c(u_1u_0)=c(u_1u_2)=c(u_1v_1)=2$, and $c(u_1)=3$. Then by considering the unique paths of length $2$ between $u_1,u_3$, we must have $c(u_2u_3)=1$, $c(u_3u_2)=2$, and $c(u_2)=3$. Similarly by considering the pairs $u_2,u_4$, then $u_3,u_0$, we must have $c(u_4u_3)=c(u_4u_0)=1$, $c(u_3u_4)=c(u_0u_4)=2$, and $c(u_3)=c(u_4)=3$. But then, we do not have total-rainbow paths between $u_1,u_4$, which is a contradiction. Therefore, $\overset{\rightarrow}{\smash{t}rc}(\lrPten)\ge 4$.

Now we define a total-colouring of ${\sf P}_{10}$ as follows. We assign strongly total rainbow connected colourings for the two cycles $u_0u_1u_2u_3u_4u_0$ and $v_0v_2v_4v_1v_3v_0$ with colours $1,2,3$, and then colour $4$ to the edges $u_iv_i$, for $0\le i\le 4$. For example, see Figure 1(b). We can then extend this to a total-colouring $c'$ of $\lrPten$ where two symmetric arcs of $\lrPten$ both receive the same colour as the corresponding edge in ${\sf P}_{10}$. Then it is easy to check that $c'$ is a strongly total rainbow connected colouring for $\lrPten$. Therefore, $\overset{\rightarrow}{s\smash{t}rc}(\lrPten)\le 4$.

Finally for every $n\ge 11$, we can obtain the digraph $D_n$ on $n$ vertices from $\lrPten$ by expanding $v_0$ into $K\cong\lrK_{n-9}$. That is, $D_n=(\lrPten)_{v_0\to K}$. Then note that $\textup{diam}(D_n)=2$. We may again apply the first argument above to obtain $\overset{\rightarrow}{\smash{t}rc}(D_n)\ge 4$. Moreover, we can extend the total-colouring $c'$ on $\lrPten$ to a total-colouring $c''$ on $D_n$ where $c''(uv)=c'(uv)$ if $uv\not\in A(K)$; $c''(uv)=c'(v_0v)$ if $u\in V(K)$ and $v\not\in V(K)$; $c''(uv)=c'(uv_0)$ if $u\not\in V(K)$ and $v\in V(K)$; $c''(uv)=1$ if $uv\in A(K)$; $c''(w)=c'(w)$ if $w\not\in V(K)$; and $c''(w)=1$ if $w\in V(K)$. Then it is easy to see that $c''$ is a strongly total rainbow connected colouring for $D_n$. Indeed, if $x,y\in V(D_n)$ and $xy\not\in A(D_n)$, then at most one of $x,y$ is in $V(K)$, and $x,y$ correspond to distinct vertices $x',y'$ in $\lrPten$. Then a total-rainbow $x-y$ geodesic of length $2$ in $D_n$ corresponds to a total-rainbow $x'-y'$ geodesic of length $2$ in $\lrPten$. Therefore, $\overset{\rightarrow}{s\smash{t}rc}(D_n)\le 4$, and $\overset{\rightarrow}{\smash{t}rc}(D_n)=\overset{\rightarrow}{s\smash{t}rc}(D_n)=4$.
\end{proof}

We propose the following problem.

\begin{prob}
Among all digraphs $D$ with diameter $2$, are the parameters $\overset{\rightarrow}{rc}(D)$, $\overset{\rightarrow}{src}(D)$, $\overset{\rightarrow}{\smash{t}rc}(D)$, and $\overset{\rightarrow}{s\smash{t}rc}(D)$ unbounded?
\end{prob}
%\textcolor{red}{Are there digraphs $D$ with diam$(D)=2$ and $\overset{\rightarrow}{\smash{t}rc}(D)\ge \overset{\rightarrow}{rc}(D)\ge 5$? Can we make use of Kneser or Johnson graphs?}

\indent Alva-Samos and Montellano-Ballesteros \cite{AM2016} showed that for a connected graph $G$, 
\begin{equation}
\overset{\rightarrow}{rc}(\lrG)\le rc(G)\quad\textup{and}\quad\overset{\rightarrow}{src}(\lrG)\le src(G).\label{ASMBrmk}
\end{equation}
Furthermore, for each inequality, there is an infinite family of graphs where equality holds, and also with the difference between the two parameters unbounded. For example, from \cite{AM2016,CJMZ2008}, we have $rc(C_n)=\overset{\rightarrow}{rc}(\lrC_n)=src(C_n)=\overset{\rightarrow}{src}(\lrC_n)=\lceil\frac{n}{2}\rceil$ for $n\ge 4$. Also, we have $\overset{\rightarrow}{rc}(\lrK_{1,n})=\overset{\rightarrow}{src}(\lrK_{1,n})=2$ and $rc(K_{1,n})=src(K_{1,n})=n$ for $n\ge 2$, where $K_{1,n}$ is the star with $n$ edges. On the other hand, for rainbow vertex-connection, Lei, Li, Liu, and Shi \cite{LLLS2017} showed that
\begin{equation}
\overset{\rightarrow}{rvc}(\lrG)=rvc(G)\quad\textup{and}\quad \overset{\rightarrow}{srvc}(\lrG)=srvc(G).\label{LLLSrmk}
\end{equation}

For total rainbow connection, we have the analogous inequalities to (\ref{ASMBrmk}).
\begin{prop}\label{pro3}
For a connected graph $G$, we have 
\begin{equation}
\overset{\rightarrow}{\smash{t}rc}(\lrG)\le trc(G)\quad\textup{and}\quad \overset{\rightarrow}{s\smash{t}rc}(\lrG)\le strc(G).\label{prop3eq}
\end{equation}
\end{prop}

\begin{proof}
Given a (strongly) total rainbow connected colouring of $G$, it is not hard to see that the total-colouring of $\lrG$, obtained by assigning the colour of the edge $uv$ to both arcs $uv,vu$, and the colour of the vertex $u$ of $G$ to the corresponding vertex $u$ of $\lrG$, is a (strongly) total rainbow connected colouring of $\lrG$. 
\end{proof}

As in the case for rainbow connection, for each inequality of (\ref{prop3eq}), there is an infinite family of graphs where equality holds, and also with the difference between the two parameters unbounded. We simply use the same examples. For $n\ge 3$, we have $trc(C_n)=\overset{\rightarrow}{\smash{t}rc}(\lrC_n)=strc(C_n)=\overset{\rightarrow}{s\smash{t}rc}(\lrC_n)$, which we will see in Theorem \ref{thm2} later. Also, for $n\ge 2$, we have $\overset{\rightarrow}{\smash{t}rc}(\lrK_{1,n})=\overset{\rightarrow}{s\smash{t}rc}(\lrK_{1,n})=3$ and $trc(K_{1,n})=strc(K_{1,n})=n+1$.

For the relationship between the total rainbow connection numbers of a digraph and its spanning subdigraphs, it is not hard to see that the following holds.
\begin{prop}
Let $D$ and $H$ be strongly connected digraphs such that, $H$ is a spanning subdigraph of $D$. Then $\overset{\rightarrow}{\smash{t}rc}(D)\leq \overset{\rightarrow}{\smash{t}rc}(H)$.
\end{prop}
However, this is not true for the strong total rainbow connection number, as we will see in the next lemma.
\begin{lem}\label{lem1}
There are strongly connected digraphs $D$ and $H$ such that, $H$ is a spanning
subdigraph of $D$, and $\overset{\rightarrow}{s\smash{t}rc}(D)> \overset{\rightarrow}{s\smash{t}rc}(H)$.
\end{lem}
\[ \unit = 0.6cm
%vertices
\pt{0}{0.8}\pt{0}{-0.8}\pt{0}{2.4}\pt{0}{-2.4}
\pt{3.9}{0}\pt{-3.9}{0}\pt{6.15}{1.3}\pt{6.15}{-1.3}\pt{-6.15}{1.3}\pt{-6.15}{-1.3}
\pt{8.28}{2.79}\pt{8.62}{0.5}\pt{8.28}{-2.79}\pt{8.62}{-0.5}\pt{-8.28}{2.79}\pt{-8.62}{0.5}\pt{-8.28}{-2.79}\pt{-8.62}{-0.5}
%arcs
\varline{500}{0.6}
\bez{0}{2.4}{2.17}{1.88}{3.9}{0}\bez{0}{2.4}{1.83}{0.92}{3.9}{0}
\bez{0}{0.8}{2.06}{0.9}{3.9}{0}\bez{0}{0.8}{1.94}{0.01}{3.9}{0}
\bez{0}{-2.4}{2.17}{-1.88}{3.9}{0}\bez{0}{-2.4}{1.83}{-0.92}{3.9}{0}
\bez{0}{-0.8}{2.06}{-0.9}{3.9}{0}\bez{0}{-0.8}{1.94}{-0.01}{3.9}{0}
\bez{0}{2.4}{-2.17}{1.88}{-3.9}{0}\bez{0}{2.4}{-1.83}{0.92}{-3.9}{0}
\bez{0}{0.8}{-2.06}{0.9}{-3.9}{0}\bez{0}{0.8}{-1.94}{0.01}{-3.9}{0}
\bez{0}{-2.4}{-2.17}{-1.88}{-3.9}{0}\bez{0}{-2.4}{-1.83}{-0.92}{-3.9}{0}
\bez{0}{-0.8}{-2.06}{-0.9}{-3.9}{0}\bez{0}{-0.8}{-1.94}{-0.01}{-3.9}{0}
%
%\bez{6.15}{1.3}{6.55}{0}{6.15}{-1.3}\bez{6.15}{1.3}{5.75}{0}{6.15}{-1.3}
\bez{3.9}{0}{4.83}{1}{6.15}{1.3}\bez{3.9}{0}{5.23}{0.3}{6.15}{1.3}
\bez{3.9}{0}{4.83}{-1}{6.15}{-1.3}\bez{3.9}{0}{5.23}{-0.3}{6.15}{-1.3}
%\bez{-6.15}{1.3}{-6.55}{0}{-6.15}{-1.3}\bez{-6.15}{1.3}{-5.75}{0}{-6.15}{-1.3}
\bez{-3.9}{0}{-4.83}{1}{-6.15}{1.3}\bez{-3.9}{0}{-5.23}{0.3}{-6.15}{1.3}
\bez{-3.9}{0}{-4.83}{-1}{-6.15}{-1.3}\bez{-3.9}{0}{-5.23}{-0.3}{-6.15}{-1.3}
\dl{8.28}{2.79}{8.62}{0.5}\dl{8.28}{2.79}{6.15}{1.3}\dl{8.62}{0.5}{6.15}{1.3}
\dl{-8.28}{2.79}{-8.62}{0.5}\dl{-8.28}{2.79}{-6.15}{1.3}\dl{-8.62}{0.5}{-6.15}{1.3}
\dl{8.28}{-2.79}{8.62}{-0.5}\dl{8.28}{-2.79}{6.15}{-1.3}\dl{8.62}{-0.5}{6.15}{-1.3}
\dl{-8.28}{-2.79}{-8.62}{-0.5}\dl{-8.28}{-2.79}{-6.15}{-1.3}\dl{-8.62}{-0.5}{-6.15}{-1.3}
\varline{850}{5.5}
\bez{-3.9}{0}{0}{-5.84}{3.9}{0}
%\bez{-3.9}{0}{0}{5.84}{3.9}{0}
\thnline
%right 8 arrows
\dl{2.22}{1.44}{2.07}{1.67}\dl{2.18}{1.47}{2.07}{1.67}\dl{2.15}{1.49}{2.07}{1.67}\dl{2.12}{1.51}{2.07}{1.67}\dl{2.1}{1.53}{2.07}{1.67}\dl{2.08}{1.54}{2.07}{1.67}\dl{2.06}{1.55}{2.07}{1.67}
\dl{2.22}{1.44}{1.95}{1.48}\dl{2.18}{1.47}{1.95}{1.48}\dl{2.15}{1.49}{1.95}{1.48}\dl{2.12}{1.51}{1.95}{1.48}\dl{2.1}{1.53}{1.95}{1.48}\dl{2.08}{1.54}{1.95}{1.48}\dl{2.06}{1.55}{1.95}{1.48}
\dl{1.78}{1.12}{1.95}{0.9}\dl{1.82}{1.09}{1.95}{0.9}\dl{1.86}{1.07}{1.95}{0.9}\dl{1.89}{1.05}{1.95}{0.9}\dl{1.92}{1.04}{1.95}{0.9}\dl{1.94}{1.03}{1.95}{0.9}\dl{1.96}{1.02}{1.95}{0.9}
\dl{1.78}{1.12}{2.05}{1.09}\dl{1.82}{1.09}{2.05}{1.09}\dl{1.86}{1.07}{2.05}{1.09}\dl{1.89}{1.05}{2.05}{1.09}\dl{1.92}{1.04}{2.05}{1.09}\dl{1.94}{1.03}{2.05}{1.09}\dl{1.96}{1.02}{2.05}{1.09}\dl{1.96}{1.02}{2.05}{1.09}
\dl{2.12}{0.63}{1.89}{0.8}\dl{2.07}{0.64}{1.89}{0.8}\dl{2.03}{0.65}{1.89}{0.8}\dl{2}{0.65}{1.89}{0.8}\dl{1.97}{0.66}{1.89}{0.8}\dl{1.94}{0.66}{1.89}{0.8}\dl{1.92}{0.67}{1.89}{0.8}
\dl{2.12}{0.63}{1.84}{0.57}\dl{2.07}{0.64}{1.84}{0.57}\dl{2.03}{0.65}{1.84}{0.57}\dl{2}{0.65}{1.84}{0.57}\dl{1.97}{0.66}{1.84}{0.57}\dl{1.94}{0.66}{1.84}{0.57}\dl{1.92}{0.67}{1.84}{0.57}
\dl{1.81}{0.23}{2.04}{0.06}\dl{1.86}{0.22}{2.04}{0.06}\dl{1.9}{0.21}{2.04}{0.06}\dl{1.93}{0.21}{2.04}{0.06}\dl{1.96}{0.2}{2.04}{0.06}\dl{1.99}{0.2}{2.04}{0.06}\dl{2.01}{0.19}{2.04}{0.06}
\dl{1.81}{0.23}{2.09}{0.3}\dl{1.86}{0.22}{2.09}{0.3}\dl{1.9}{0.21}{2.09}{0.3}\dl{1.93}{0.21}{2.09}{0.3}\dl{1.96}{0.2}{2.09}{0.3}\dl{1.99}{0.2}{2.09}{0.3}\dl{2.01}{0.19}{2.09}{0.3}
\dl{2.03}{-0.18}{1.8}{-0.35}\dl{1.98}{-0.19}{1.8}{-0.35}\dl{1.94}{-0.2}{1.8}{-0.35}\dl{1.91}{-0.2}{1.8}{-0.35}\dl{1.88}{-0.21}{1.8}{-0.35}\dl{1.85}{-0.21}{1.8}{-0.35}\dl{1.83}{-0.22}{1.8}{-0.35}
\dl{2.03}{-0.18}{1.75}{-0.12}\dl{1.98}{-0.19}{1.75}{-0.12}\dl{1.94}{-0.2}{1.75}{-0.12}\dl{1.91}{-0.2}{1.75}{-0.12}\dl{1.88}{-0.21}{1.75}{-0.12}\dl{1.85}{-0.21}{1.75}{-0.12}\dl{1.83}{-0.22}{1.75}{-0.12}
\dl{1.88}{-0.68}{2.11}{-0.51}\dl{1.93}{-0.67}{2.11}{-0.51}\dl{1.97}{-0.66}{2.11}{-0.51}\dl{2}{-0.66}{2.11}{-0.51}\dl{2.03}{-0.65}{2.11}{-0.51}\dl{2.06}{-0.65}{2.11}{-0.51}\dl{2.08}{-0.64}{2.11}{-0.51}
\dl{1.88}{-0.68}{2.16}{-0.74}\dl{1.93}{-0.67}{2.16}{-0.74}\dl{1.97}{-0.66}{2.16}{-0.74}\dl{2}{-0.66}{2.16}{-0.74}\dl{2.03}{-0.65}{2.16}{-0.74}\dl{2.06}{-0.65}{2.16}{-0.74}\dl{2.08}{-0.64}{2.16}{-0.74}
\dl{2.03}{-1.57}{2.18}{-1.34}\dl{2.07}{-1.54}{2.18}{-1.34}\dl{2.1}{-1.52}{2.18}{-1.34}\dl{2.13}{-1.5}{2.18}{-1.34}\dl{2.15}{-1.48}{2.18}{-1.34}\dl{2.17}{-1.47}{2.18}{-1.34}\dl{2.19}{-1.46}{2.18}{-1.34}
\dl{2.03}{-1.57}{2.3}{-1.53}\dl{2.07}{-1.54}{2.3}{-1.53}\dl{2.1}{-1.52}{2.3}{-1.53}\dl{2.13}{-1.5}{2.3}{-1.53}\dl{2.15}{-1.48}{2.3}{-1.53}\dl{2.17}{-1.47}{2.3}{-1.53}\dl{2.19}{-1.46}{2.3}{-1.53}
\dl{1.96}{-1.02}{1.83}{-1.23}\dl{1.93}{-1.02}{1.83}{-1.23}\dl{1.9}{-1.04}{1.83}{-1.23}\dl{1.88}{-1.06}{1.83}{-1.23}\dl{1.86}{-1.08}{1.83}{-1.23}\dl{1.84}{-1.09}{1.83}{-1.23}\dl{1.82}{-1.1}{1.83}{-1.23}\dl{1.8}{-1.11}{1.83}{-1.23}
\dl{1.96}{-1.02}{1.71}{-1.04}\dl{1.93}{-1.02}{1.71}{-1.04}\dl{1.9}{-1.04}{1.71}{-1.04}\dl{1.88}{-1.06}{1.71}{-1.04}\dl{1.86}{-1.08}{1.71}{-1.04}\dl{1.84}{-1.09}{1.71}{-1.04}\dl{1.82}{-1.1}{1.71}{-1.04}\dl{1.8}{-1.11}{1.71}{-1.04}
%right bitriangle arrows
%\dl{5.95}{0.13}{6.05}{-0.13}\dl{5.95}{0.06}{6.05}{-0.13}\dl{5.95}{0.01}{6.05}{-0.13}\dl{5.95}{-0.03}{6.05}{-0.13}\dl{5.95}{-0.05}{6.05}{-0.13}\dl{5.95}{-0.07}{6.05}{-0.13}
%\dl{5.95}{0.13}{5.85}{-0.13}\dl{5.95}{0.06}{5.85}{-0.13}\dl{5.95}{0.01}{5.85}{-0.13}\dl{5.95}{-0.03}{5.85}{-0.13}\dl{5.95}{-0.05}{5.85}{-0.13}\dl{5.95}{-0.07}{5.85}{-0.13}
%
\dl{5.04}{0.88}{4.87}{0.67}\dl{5}{0.86}{4.87}{0.67}\dl{4.96}{0.84}{4.87}{0.67}\dl{4.92}{0.82}{4.87}{0.67}\dl{4.89}{0.8}{4.87}{0.67}\dl{4.87}{0.78}{4.87}{0.67}
\dl{5.04}{0.88}{4.77}{0.85}\dl{5}{0.86}{4.77}{0.85}\dl{4.96}{0.84}{4.77}{0.85}\dl{4.92}{0.82}{4.77}{0.85}\dl{4.89}{0.8}{4.77}{0.85}\dl{4.87}{0.78}{4.77}{0.85}
\dl{5.01}{0.42}{5.18}{0.63}\dl{5.05}{0.44}{5.18}{0.63}\dl{5.09}{0.46}{5.18}{0.63}\dl{5.13}{0.48}{5.18}{0.63}\dl{5.16}{0.5}{5.18}{0.63}\dl{5.18}{0.52}{5.18}{0.63}
\dl{5.01}{0.42}{5.28}{0.45}\dl{5.05}{0.44}{5.28}{0.45}\dl{5.09}{0.46}{5.28}{0.45}\dl{5.13}{0.48}{5.28}{0.45}\dl{5.16}{0.5}{5.28}{0.45}\dl{5.18}{0.52}{5.28}{0.45}
\dl{5.24}{-0.54}{5.07}{-0.33}\dl{5.2}{-0.52}{5.07}{-0.33}\dl{5.16}{-0.5}{5.07}{-0.33}\dl{5.12}{-0.48}{5.07}{-0.33}\dl{5.09}{-0.46}{5.07}{-0.33}\dl{5.07}{-0.44}{5.07}{-0.33}
\dl{5.24}{-0.54}{4.97}{-0.51}\dl{5.2}{-0.52}{4.97}{-0.51}\dl{5.16}{-0.5}{4.97}{-0.51}\dl{5.12}{-0.48}{4.97}{-0.51}\dl{5.09}{-0.46}{4.97}{-0.51}\dl{5.07}{-0.44}{4.97}{-0.51}
\dl{4.81}{-0.76}{4.98}{-0.97}\dl{4.85}{-0.78}{4.98}{-0.97}\dl{4.89}{-0.8}{4.98}{-0.97}\dl{4.93}{-0.82}{4.98}{-0.97}\dl{4.96}{-0.84}{4.98}{-0.97}\dl{4.98}{-0.86}{4.98}{-0.97}
\dl{4.81}{-0.76}{5.08}{-0.79}\dl{4.85}{-0.78}{5.08}{-0.79}\dl{4.89}{-0.8}{5.08}{-0.79}\dl{4.93}{-0.82}{5.08}{-0.79}\dl{4.96}{-0.84}{5.08}{-0.79}\dl{4.98}{-0.86}{5.08}{-0.79}
%
%\dl{6.35}{-0.13}{6.45}{0.13}\dl{6.35}{-0.06}{6.45}{0.13}\dl{6.35}{-0.01}{6.45}{0.13}\dl{6.35}{0.03}{6.45}{0.13}\dl{6.35}{0.05}{6.45}{0.13}\dl{6.35}{0.07}{6.45}{0.13}
%\dl{6.35}{-0.13}{6.25}{0.13}\dl{6.35}{-0.06}{6.25}{0.13}\dl{6.35}{-0.01}{6.25}{0.13}\dl{6.35}{0.03}{6.25}{0.13}\dl{6.35}{0.05}{6.25}{0.13}\dl{6.35}{0.07}{6.25}{0.13}
%rightmost arrows
\dl{7.31}{2.11}{7.05}{2.06}\dl{7.27}{2.08}{7.05}{2.06}\dl{7.24}{2.05}{7.05}{2.06}\dl{7.21}{2.03}{7.05}{2.06}\dl{7.18}{2.01}{7.05}{2.06}\dl{7.15}{2}{7.05}{2.06}
\dl{7.31}{2.11}{7.17}{1.89}\dl{7.27}{2.08}{7.17}{1.89}\dl{7.24}{2.05}{7.17}{1.89}\dl{7.21}{2.03}{7.17}{1.89}\dl{7.18}{2.01}{7.17}{1.89}\dl{7.15}{2}{7.17}{1.89}
\dl{7.26}{0.94}{7.52}{0.97}\dl{7.32}{0.92}{7.52}{0.97}\dl{7.36}{0.91}{7.52}{0.97}\dl{7.39}{0.9}{7.52}{0.97}\dl{7.42}{0.89}{7.52}{0.97}\dl{7.45}{0.88}{7.52}{0.97}
\dl{7.26}{0.94}{7.46}{0.77}\dl{7.32}{0.92}{7.46}{0.77}\dl{7.36}{0.91}{7.46}{0.77}\dl{7.39}{0.9}{7.46}{0.77}\dl{7.42}{0.89}{7.46}{0.77}\dl{7.45}{0.88}{7.46}{0.77}
\dl{8.44}{1.7}{8.53}{1.75}\dl{8.44}{1.68}{8.53}{1.75}\dl{8.45}{1.65}{8.53}{1.75}\dl{8.45}{1.62}{8.53}{1.75}\dl{8.46}{1.59}{8.53}{1.75}\dl{8.46}{1.55}{8.53}{1.75}\dl{8.47}{1.5}{8.53}{1.75}
\dl{8.44}{1.7}{8.33}{1.71}\dl{8.44}{1.68}{8.33}{1.71}\dl{8.45}{1.65}{8.33}{1.71}\dl{8.45}{1.62}{8.33}{1.71}\dl{8.46}{1.59}{8.33}{1.71}\dl{8.46}{1.55}{8.33}{1.71}\dl{8.47}{1.5}{8.33}{1.71}
\dl{7.12}{-1.98}{7.38}{-2.03}\dl{7.16}{-2.01}{7.38}{-2.03}\dl{7.19}{-2.04}{7.38}{-2.03}\dl{7.22}{-2.06}{7.38}{-2.03}\dl{7.25}{-2.08}{7.38}{-2.03}\dl{7.28}{-2.09}{7.38}{-2.03}
\dl{7.12}{-1.98}{7.26}{-2.2}\dl{7.16}{-2.01}{7.26}{-2.2}\dl{7.19}{-2.04}{7.26}{-2.2}\dl{7.22}{-2.06}{7.26}{-2.2}\dl{7.25}{-2.08}{7.26}{-2.2}\dl{7.28}{-2.09}{7.26}{-2.2}
\dl{7.51}{-0.86}{7.25}{-0.83}\dl{7.45}{-0.88}{7.25}{-0.83}\dl{7.41}{-0.89}{7.25}{-0.83}\dl{7.38}{-0.9}{7.25}{-0.83}\dl{7.35}{-0.91}{7.25}{-0.83}\dl{7.32}{-0.92}{7.25}{-0.83}
\dl{7.51}{-0.86}{7.31}{-1.03}\dl{7.45}{-0.88}{7.31}{-1.03}\dl{7.41}{-0.89}{7.31}{-1.03}\dl{7.38}{-0.9}{7.31}{-1.03}\dl{7.35}{-0.91}{7.31}{-1.03}\dl{7.32}{-0.92}{7.31}{-1.03}
\dl{8.46}{-1.58}{8.37}{-1.53}\dl{8.46}{-1.6}{8.37}{-1.53}\dl{8.45}{-1.63}{8.37}{-1.53}\dl{8.45}{-1.66}{8.37}{-1.53}\dl{8.44}{-1.69}{8.37}{-1.53}\dl{8.44}{-1.73}{8.37}{-1.53}\dl{8.43}{-1.78}{8.37}{-1.53}
\dl{8.46}{-1.58}{8.57}{-1.57}\dl{8.46}{-1.6}{8.57}{-1.57}\dl{8.45}{-1.63}{8.57}{-1.57}\dl{8.45}{-1.66}{8.57}{-1.57}\dl{8.44}{-1.69}{8.57}{-1.57}\dl{8.44}{-1.73}{8.57}{-1.57}\dl{8.43}{-1.78}{8.57}{-1.57}
%left 8 arrows
\dl{-2.22}{-1.44}{-2.07}{-1.67}\dl{-2.18}{-1.47}{-2.07}{-1.67}\dl{-2.15}{-1.49}{-2.07}{-1.67}\dl{-2.12}{-1.51}{-2.07}{-1.67}\dl{-2.1}{-1.53}{-2.07}{-1.67}\dl{-2.08}{-1.54}{-2.07}{-1.67}\dl{-2.06}{-1.55}{-2.07}{-1.67}
\dl{-2.22}{-1.44}{-1.95}{-1.48}\dl{-2.18}{-1.47}{-1.95}{-1.48}\dl{-2.15}{-1.49}{-1.95}{-1.48}\dl{-2.12}{-1.51}{-1.95}{-1.48}\dl{-2.1}{-1.53}{-1.95}{-1.48}\dl{-2.08}{-1.54}{-1.95}{-1.48}\dl{-2.06}{-1.55}{-1.95}{-1.48}
\dl{-1.78}{-1.12}{-1.95}{-0.9}\dl{-1.82}{-1.09}{-1.95}{-0.9}\dl{-1.86}{-1.07}{-1.95}{-0.9}\dl{-1.89}{-1.05}{-1.95}{-0.9}\dl{-1.92}{-1.04}{-1.95}{-0.9}\dl{-1.94}{-1.03}{-1.95}{-0.9}\dl{-1.96}{-1.02}{-1.95}{-0.9}
\dl{-1.78}{-1.12}{-2.05}{-1.09}\dl{-1.82}{-1.09}{-2.05}{-1.09}\dl{-1.86}{-1.07}{-2.05}{-1.09}\dl{-1.89}{-1.05}{-2.05}{-1.09}\dl{-1.92}{-1.04}{-2.05}{-1.09}\dl{-1.94}{-1.03}{-2.05}{-1.09}\dl{-1.96}{-1.02}{-2.05}{-1.09}\dl{-1.96}{-1.02}{-2.05}{-1.09}
\dl{-2.12}{-0.63}{-1.89}{-0.8}\dl{-2.07}{-0.64}{-1.89}{-0.8}\dl{-2.03}{-0.65}{-1.89}{-0.8}\dl{-2}{-0.65}{-1.89}{-0.8}\dl{-1.97}{-0.66}{-1.89}{-0.8}\dl{-1.94}{-0.66}{-1.89}{-0.8}\dl{-1.92}{-0.67}{-1.89}{-0.8}
\dl{-2.12}{-0.63}{-1.84}{-0.57}\dl{-2.07}{-0.64}{-1.84}{-0.57}\dl{-2.03}{-0.65}{-1.84}{-0.57}\dl{-2}{-0.65}{-1.84}{-0.57}\dl{-1.97}{-0.66}{-1.84}{-0.57}\dl{-1.94}{-0.66}{-1.84}{-0.57}\dl{-1.92}{-0.67}{-1.84}{-0.57}
\dl{-1.81}{-0.23}{-2.04}{-0.06}\dl{-1.86}{-0.22}{-2.04}{-0.06}\dl{-1.9}{-0.21}{-2.04}{-0.06}\dl{-1.93}{-0.21}{-2.04}{-0.06}\dl{-1.96}{-0.2}{-2.04}{-0.06}\dl{-1.99}{-0.2}{-2.04}{-0.06}\dl{-2.01}{-0.19}{-2.04}{-0.06}
\dl{-1.81}{-0.23}{-2.09}{-0.3}\dl{-1.86}{-0.22}{-2.09}{-0.3}\dl{-1.9}{-0.21}{-2.09}{-0.3}\dl{-1.93}{-0.21}{-2.09}{-0.3}\dl{-1.96}{-0.2}{-2.09}{-0.3}\dl{-1.99}{-0.2}{-2.09}{-0.3}\dl{-2.01}{-0.19}{-2.09}{-0.3}
\dl{-2.03}{0.18}{-1.8}{0.35}\dl{-1.98}{0.19}{-1.8}{0.35}\dl{-1.94}{0.2}{-1.8}{0.35}\dl{-1.91}{0.2}{-1.8}{0.35}\dl{-1.88}{0.21}{-1.8}{0.35}\dl{-1.85}{0.21}{-1.8}{0.35}\dl{-1.83}{0.22}{-1.8}{0.35}
\dl{-2.03}{0.18}{-1.75}{0.12}\dl{-1.98}{0.19}{-1.75}{0.12}\dl{-1.94}{0.2}{-1.75}{0.12}\dl{-1.91}{0.2}{-1.75}{0.12}\dl{-1.88}{0.21}{-1.75}{0.12}\dl{-1.85}{0.21}{-1.75}{0.12}\dl{-1.83}{0.22}{-1.75}{0.12}
\dl{-1.88}{0.68}{-2.11}{0.51}\dl{-1.93}{0.67}{-2.11}{0.51}\dl{-1.97}{0.66}{-2.11}{0.51}\dl{-2}{0.66}{-2.11}{0.51}\dl{-2.03}{0.65}{-2.11}{0.51}\dl{-2.06}{0.65}{-2.11}{0.51}\dl{-2.08}{0.64}{-2.11}{0.51}
\dl{-1.88}{0.68}{-2.16}{0.74}\dl{-1.93}{0.67}{-2.16}{0.74}\dl{-1.97}{0.66}{-2.16}{0.74}\dl{-2}{0.66}{-2.16}{0.74}\dl{-2.03}{0.65}{-2.16}{0.74}\dl{-2.06}{0.65}{-2.16}{0.74}\dl{-2.08}{0.64}{-2.16}{0.74}
\dl{-2.03}{1.57}{-2.18}{1.34}\dl{-2.07}{1.54}{-2.18}{1.34}\dl{-2.1}{1.52}{-2.18}{1.34}\dl{-2.13}{1.5}{-2.18}{1.34}\dl{-2.15}{1.48}{-2.18}{1.34}\dl{-2.17}{1.47}{-2.18}{1.34}\dl{-2.19}{1.46}{-2.18}{1.34}
\dl{-2.03}{1.57}{-2.3}{1.53}\dl{-2.07}{1.54}{-2.3}{1.53}\dl{-2.1}{1.52}{-2.3}{1.53}\dl{-2.13}{1.5}{-2.3}{1.53}\dl{-2.15}{1.48}{-2.3}{1.53}\dl{-2.17}{1.47}{-2.3}{1.53}\dl{-2.19}{1.46}{-2.3}{1.53}
\dl{-1.96}{1.02}{-1.83}{1.23}\dl{-1.93}{1.02}{-1.83}{1.23}\dl{-1.9}{1.04}{-1.83}{1.23}\dl{-1.88}{1.06}{-1.83}{1.23}\dl{-1.86}{1.08}{-1.83}{1.23}\dl{-1.84}{1.09}{-1.83}{1.23}\dl{-1.82}{1.1}{-1.83}{1.23}\dl{-1.8}{1.11}{-1.83}{1.23}
\dl{-1.96}{1.02}{-1.71}{1.04}\dl{-1.93}{1.02}{-1.71}{1.04}\dl{-1.9}{1.04}{-1.71}{1.04}\dl{-1.88}{1.06}{-1.71}{1.04}\dl{-1.86}{1.08}{-1.71}{1.04}\dl{-1.84}{1.09}{-1.71}{1.04}\dl{-1.82}{1.1}{-1.71}{1.04}\dl{-1.8}{1.11}{-1.71}{1.04}
%left bitriangle arrows
%\dl{-5.95}{-0.13}{-6.05}{0.13}\dl{-5.95}{-0.06}{-6.05}{0.13}\dl{-5.95}{-0.01}{-6.05}{0.13}\dl{-5.95}{0.03}{-6.05}{0.13}\dl{-5.95}{0.05}{-6.05}{0.13}\dl{-5.95}{0.07}{-6.05}{0.13}
%\dl{-5.95}{-0.13}{-5.85}{0.13}\dl{-5.95}{-0.06}{-5.85}{0.13}\dl{-5.95}{-0.01}{-5.85}{0.13}\dl{-5.95}{0.03}{-5.85}{0.13}\dl{-5.95}{0.05}{-5.85}{0.13}\dl{-5.95}{0.07}{-5.85}{0.13}
%
\dl{-5.04}{-0.88}{-4.87}{-0.67}\dl{-5}{-0.86}{-4.87}{-0.67}\dl{-4.96}{-0.84}{-4.87}{-0.67}\dl{-4.92}{-0.82}{-4.87}{-0.67}\dl{-4.89}{-0.8}{-4.87}{-0.67}\dl{-4.87}{-0.78}{-4.87}{-0.67}
\dl{-5.04}{-0.88}{-4.77}{-0.85}\dl{-5}{-0.86}{-4.77}{-0.85}\dl{-4.96}{-0.84}{-4.77}{-0.85}\dl{-4.92}{-0.82}{-4.77}{-0.85}\dl{-4.89}{-0.8}{-4.77}{-0.85}\dl{-4.87}{-0.78}{-4.77}{-0.85}
\dl{-5.01}{-0.42}{-5.18}{-0.63}\dl{-5.05}{-0.44}{-5.18}{-0.63}\dl{-5.09}{-0.46}{-5.18}{-0.63}\dl{-5.13}{-0.48}{-5.18}{-0.63}\dl{-5.16}{-0.5}{-5.18}{-0.63}\dl{-5.18}{-0.52}{-5.18}{-0.63}
\dl{-5.01}{-0.42}{-5.28}{-0.45}\dl{-5.05}{-0.44}{-5.28}{-0.45}\dl{-5.09}{-0.46}{-5.28}{-0.45}\dl{-5.13}{-0.48}{-5.28}{-0.45}\dl{-5.16}{-0.5}{-5.28}{-0.45}\dl{-5.18}{-0.52}{-5.28}{-0.45}
\dl{-5.24}{0.54}{-5.07}{0.33}\dl{-5.2}{0.52}{-5.07}{0.33}\dl{-5.16}{0.5}{-5.07}{0.33}\dl{-5.12}{0.48}{-5.07}{0.33}\dl{-5.09}{0.46}{-5.07}{0.33}\dl{-5.07}{0.44}{-5.07}{0.33}
\dl{-5.24}{0.54}{-4.97}{0.51}\dl{-5.2}{0.52}{-4.97}{0.51}\dl{-5.16}{0.5}{-4.97}{0.51}\dl{-5.12}{0.48}{-4.97}{0.51}\dl{-5.09}{0.46}{-4.97}{0.51}\dl{-5.07}{0.44}{-4.97}{0.51}
\dl{-4.81}{0.76}{-4.98}{0.97}\dl{-4.85}{0.78}{-4.98}{0.97}\dl{-4.89}{0.8}{-4.98}{0.97}\dl{-4.93}{0.82}{-4.98}{0.97}\dl{-4.96}{0.84}{-4.98}{0.97}\dl{-4.98}{0.86}{-4.98}{0.97}
\dl{-4.81}{0.76}{-5.08}{0.79}\dl{-4.85}{0.78}{-5.08}{0.79}\dl{-4.89}{0.8}{-5.08}{0.79}\dl{-4.93}{0.82}{-5.08}{0.79}\dl{-4.96}{0.84}{-5.08}{0.79}\dl{-4.98}{0.86}{-5.08}{0.79}
%
%\dl{-6.35}{0.13}{-6.45}{-0.13}\dl{-6.35}{0.06}{-6.45}{-0.13}\dl{-6.35}{0.01}{-6.45}{-0.13}\dl{-6.35}{-0.03}{-6.45}{-0.13}\dl{-6.35}{-0.05}{-6.45}{-0.13}\dl{-6.35}{-0.07}{-6.45}{-0.13}
%\dl{-6.35}{0.13}{-6.25}{-0.13}\dl{-6.35}{0.06}{-6.25}{-0.13}\dl{-6.35}{0.01}{-6.25}{-0.13}\dl{-6.35}{-0.03}{-6.25}{-0.13}\dl{-6.35}{-0.05}{-6.25}{-0.13}\dl{-6.35}{-0.07}{-6.25}{-0.13}
%leftmost arrows
\dl{-7.31}{-2.11}{-7.05}{-2.06}\dl{-7.27}{-2.08}{-7.05}{-2.06}\dl{-7.24}{-2.05}{-7.05}{-2.06}\dl{-7.21}{-2.03}{-7.05}{-2.06}\dl{-7.18}{-2.01}{-7.05}{-2.06}\dl{-7.15}{-2}{-7.05}{-2.06}
\dl{-7.31}{-2.11}{-7.17}{-1.89}\dl{-7.27}{-2.08}{-7.17}{-1.89}\dl{-7.24}{-2.05}{-7.17}{-1.89}\dl{-7.21}{-2.03}{-7.17}{-1.89}\dl{-7.18}{-2.01}{-7.17}{-1.89}\dl{-7.15}{-2}{-7.17}{-1.89}
\dl{-7.26}{-0.94}{-7.52}{-0.97}\dl{-7.32}{-0.92}{-7.52}{-0.97}\dl{-7.36}{-0.91}{-7.52}{-0.97}\dl{-7.39}{-0.9}{-7.52}{-0.97}\dl{-7.42}{-0.89}{-7.52}{-0.97}\dl{-7.45}{-0.88}{-7.52}{-0.97}
\dl{-7.26}{-0.94}{-7.46}{-0.77}\dl{-7.32}{-0.92}{-7.46}{-0.77}\dl{-7.36}{-0.91}{-7.46}{-0.77}\dl{-7.39}{-0.9}{-7.46}{-0.77}\dl{-7.42}{-0.89}{-7.46}{-0.77}\dl{-7.45}{-0.88}{-7.46}{-0.77}
\dl{-8.44}{-1.7}{-8.53}{-1.75}\dl{-8.44}{-1.68}{-8.53}{-1.75}\dl{-8.45}{-1.65}{-8.53}{-1.75}\dl{-8.45}{-1.62}{-8.53}{-1.75}\dl{-8.46}{-1.59}{-8.53}{-1.75}\dl{-8.46}{-1.55}{-8.53}{-1.75}\dl{-8.47}{-1.5}{-8.53}{-1.75}
\dl{-8.44}{-1.7}{-8.33}{-1.71}\dl{-8.44}{-1.68}{-8.33}{-1.71}\dl{-8.45}{-1.65}{-8.33}{-1.71}\dl{-8.45}{-1.62}{-8.33}{-1.71}\dl{-8.46}{-1.59}{-8.33}{-1.71}\dl{-8.46}{-1.55}{-8.33}{-1.71}\dl{-8.47}{-1.5}{-8.33}{-1.71}
\dl{-7.12}{1.98}{-7.38}{2.03}\dl{-7.16}{2.01}{-7.38}{2.03}\dl{-7.19}{2.04}{-7.38}{2.03}\dl{-7.22}{2.06}{-7.38}{2.03}\dl{-7.25}{2.08}{-7.38}{2.03}\dl{-7.28}{2.09}{-7.38}{2.03}
\dl{-7.12}{1.98}{-7.26}{2.2}\dl{-7.16}{2.01}{-7.26}{2.2}\dl{-7.19}{2.04}{-7.26}{2.2}\dl{-7.22}{2.06}{-7.26}{2.2}\dl{-7.25}{2.08}{-7.26}{2.2}\dl{-7.28}{2.09}{-7.26}{2.2}
\dl{-7.51}{0.86}{-7.25}{0.83}\dl{-7.45}{0.88}{-7.25}{0.83}\dl{-7.41}{0.89}{-7.25}{0.83}\dl{-7.38}{0.9}{-7.25}{0.83}\dl{-7.35}{0.91}{-7.25}{0.83}\dl{-7.32}{0.92}{-7.25}{0.83}
\dl{-7.51}{0.86}{-7.31}{1.03}\dl{-7.45}{0.88}{-7.31}{1.03}\dl{-7.41}{0.89}{-7.31}{1.03}\dl{-7.38}{0.9}{-7.31}{1.03}\dl{-7.35}{0.91}{-7.31}{1.03}\dl{-7.32}{0.92}{-7.31}{1.03}
\dl{-8.46}{1.58}{-8.37}{1.53}\dl{-8.46}{1.6}{-8.37}{1.53}\dl{-8.45}{1.63}{-8.37}{1.53}\dl{-8.45}{1.66}{-8.37}{1.53}\dl{-8.44}{1.69}{-8.37}{1.53}\dl{-8.44}{1.73}{-8.37}{1.53}\dl{-8.43}{1.78}{-8.37}{1.53}
\dl{-8.46}{1.58}{-8.57}{1.57}\dl{-8.46}{1.6}{-8.57}{1.57}\dl{-8.45}{1.63}{-8.57}{1.57}\dl{-8.45}{1.66}{-8.57}{1.57}\dl{-8.44}{1.69}{-8.57}{1.57}\dl{-8.44}{1.73}{-8.57}{1.57}\dl{-8.43}{1.78}{-8.57}{1.57}
%dotted arrow
\dl{0.13}{-2.91}{-0.07}{-2.91}
\dl{0.13}{-2.91}{-0.13}{-2.81}\dl{0.09}{-2.91}{-0.13}{-2.81}\dl{0.05}{-2.91}{-0.13}{-2.81}\dl{0.01}{-2.91}{-0.13}{-2.81}\dl{-0.03}{-2.91}{-0.13}{-2.81}\dl{-0.07}{-2.91}{-0.13}{-2.81}
\dl{0.13}{-2.91}{-0.13}{-3.01}\dl{0.09}{-2.91}{-0.13}{-3.01}\dl{0.05}{-2.91}{-0.13}{-3.01}\dl{0.01}{-2.91}{-0.13}{-3.01}\dl{-0.03}{-2.91}{-0.13}{-3.01}\dl{-0.07}{-2.91}{-0.13}{-3.01}
%vertices labels
\point{-9.35}{0.37}{\small $u_1$}\point{-9.07}{-2.9}{\small $u_2$}\point{8.87}{-0.62}{\small $u_4$}\point{8.52}{2.67}{\small $u_3$}
\point{-8.97}{2.67}{\small $v_1$}\point{-9.35}{-0.62}{\small $v_2$}\point{8.52}{-2.9}{\small $v_4$}\point{8.87}{0.37}{\small $v_3$}
\point{-6.3}{1.6}{\small $w_1$}\point{-6.3}{-1.8}{\small $w_2$}\point{5.6}{-1.8}{\small $w_4$}\point{5.6}{1.6}{\small $w_3$}
\point{-4.06}{0.38}{\small $x$}\point{3.78}{0.38}{\small $y$}
%colours
\ptll{-6.6}{1.37}{1}\ptll{-6.6}{-1.45}{2}\ptld{-3.9}{-0.2}{5}\ptlr{6.6}{-1.45}{4}\ptlr{6.6}{1.37}{3}\ptld{3.9}{-0.2}{6}
\ptld{-1.15}{2.6}{9}\ptld{-1.15}{1.3}{10}\ptlu{-1.05}{-1.65}{10}\ptlu{-1.15}{-0.54}{9}
\ptld{1.15}{2.6}{7}\ptld{1.15}{1.3}{7}\ptlu{1.15}{-1.65}{8}\ptlu{1.15}{-0.54}{8}
\ptlu{-8.31}{0.53}{11}\ptlu{-8.07}{-2.6}{11}
\ptld{-8.07}{2.57}{12}\ptld{-8.27}{-0.63}{12}
\ptld{8.23}{-0.65}{11}\ptld{8.07}{2.56}{11}
\ptlu{8.07}{-2.6}{12}\ptlu{8.23}{0.56}{12}
\ptlu{-8.76}{1.23}{7}\ptlu{-8.86}{-2.03}{8}
\ptld{8.86}{1.95}{9}\ptld{8.86}{-1.35}{10}
\ptlu{-7.1}{2.05}{13}\ptld{-7.31}{0.85}{14}
\ptlu{-7.31}{-0.85}{13}\ptld{-7.1}{-2.05}{14}
\ptlu{7.1}{2.05}{14}\ptld{7.31}{0.85}{13}
\ptlu{7.31}{-0.85}{14}\ptld{7.1}{-2.15}{13}
%
%\ptlu{-6.63}{-0.38}{5}
\ptlu{-4.88}{0.86}{15}\ptld{-4.88}{-0.86}{16}
\ptld{-5.6}{0.63}{16}\ptlu{-5.6}{-0.63}{15}
%\ptld{6.71}{0.31}{6}
\ptlu{4.88}{0.82}{16}\ptld{4.88}{-0.9}{15}
\ptld{5.6}{0.63}{15}\ptlu{5.6}{-0.63}{16}
\ptlu{0}{2.4}{1}\ptlu{0}{0.8}{1}\ptlu{0}{-0.8}{2}\ptlu{0}{-2.4}{2}
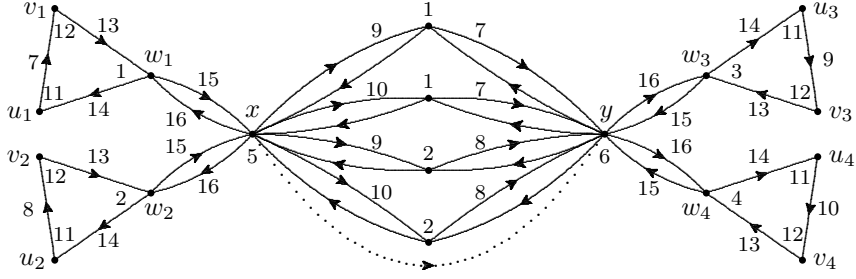
\ptlu{0}{-5.2}{\textup{\small Figure 2. The digraphs $D$ and $H$ in Lemma \ref{lem1}.}}
\]\\[-3ex]
\begin{proof}
Let $H$ be the digraph consisting of the solid arcs as in Figure 2, and $D$ be the digraph obtained from $H$ by adding the dotted arc $xy$. It is not hard to see that the colouring for $H$ as shown in Figure 2, is strongly total rainbow connected, where for the eight pairs of symmetric arcs in the middle, the two arcs in each pair have the same colour. Thus, $\overset{\rightarrow}{s\smash{t}rc}(H)\leq16$. 

Now, we will show that $\overset{\rightarrow}{s\smash{t}rc}(D)\geq17$. Suppose that we have a strongly total rainbow connected colouring $c$ for $D$, using at most $16$ colours. Notice that any two cut-vertices of $D$ must receive distinct colours, and thus without loss of generality, we may assume that $w_{1},w_{2},w_{3},w_{4},x,y$ have colours $1,2,3,4,5,6$. Next, for $1\leq i\leq4$, we see that the arcs $u_{i}v_{i}$ must have distinct colours, and different from colours $1,2,3,4,5,6$. Otherwise we can find two vertices such that any geodesic in $D$ connecting them is not total-rainbow. Hence, we may assume that $c(u_iv_i)=i+6$, for $1\le i\le 4$. For the same reason, the arcs $w_iu_i,v_iw_i$ and vertices $u_i,v_i$ for $1\le i\le 4$, and the arcs $w_ix,xw_i$ for $i=1,2$, and $w_iy,yw_i$ for $i=3,4$, cannot use the colours $1,2,\dots,10$. Since the unique $u_1-v_2$ geodesic in $D$ is $u_1v_1w_1xw_2u_2v_2$, we may assume that $c(v_1)=12$, $c(v_1w_1)=13$, $c(w_1x)=15$, $c(xw_2)=16$, $c(w_2u_2)=14$, and $c(u_2)=11$. Finally, the unique $u_1-v_3$ geodesic in $D$ is $u_1v_1w_1xyw_3u_3v_3$, and the arcs $xy,yw_3,w_3u_3$ and internal vertex $u_3$ in this geodesic are not yet coloured. These elements cannot use the colours $1,2,\dots,10,12,13,15$. Thus the remaining possible colours are $11,14,16$, and this is insufficient to make the geodesic total-rainbow. We have a contradiction.  Hence $\overset{\rightarrow}{s\smash{t}rc}(D)\geq17$ and the result follows.
\end{proof}

Our final aim in this section is to compare the rainbow connection parameters. In \cite{KY2010}, Krivelevich and Yuster observed that for $rc(G)$ and $rvc(G)$, we cannot generally find an upper bound for one of the parameters in terms of the other. Indeed, by taking $G=K_{1,s}$, we have $rc(G) = s$ and $rvc(G) = 1$. On the other hand, let the graph $G_s$ be constructed as follows. Take $s$ vertex-disjoint triangles and, by designating a vertex from each triangle, add a complete graph $K_s$ on the designated vertices. Then $rc(G_s) \le 4$ and $rvc(G_s) = s$. In \cite{LLLS2017}, similar results were obtained for $\overset{\rightarrow}{rc}(D)$ and $\overset{\rightarrow}{rvc}(D)$, and for $\overset{\rightarrow}{src}(D)$ and $\overset{\rightarrow}{srvc}(D)$.

When considering the total connection number in addition, we have the following trivial inequalities.
\begin{align}
\overset{\rightarrow}{\smash{t}rc}(D) & \ge\max(\overset{\rightarrow}{rc}(D),\overset{\rightarrow}{rvc}(D)),\label{maxeq1}\\
\overset{\rightarrow}{s\smash{t}rc}(D) & \ge\max(\overset{\rightarrow}{src}(D),\overset{\rightarrow}{srvc}(D)).\label{maxeq2}
\end{align}

In the following result, we see that there are infinitely many digraphs where the inequalities (\ref{maxeq1}) and (\ref{maxeq2}) are best possible.

\begin{thm}\label{compequalthm}
\indent\\[-3ex]
\begin{enumerate}
\item[(a)] There exist infinitely many strongly connected digraphs $D$ with $\overset{\rightarrow}{\smash{t}rc}(D) =\overset{\rightarrow}{s\smash{t}rc}(D) = \overset{\rightarrow}{rc}(D)=\overset{\rightarrow}{src}(D)=3$.
\item[(b)] Given $s\ge 13$, there exists a strongly connected digraph $D$ with $\overset{\rightarrow}{\smash{t}rc}(D) = \overset{\rightarrow}{rvc}(D)=s$.
\item[(c)] Given $s\ge 13$, there exists a strongly connected digraph $D$ with $\overset{\rightarrow}{s\smash{t}rc}(D) =\overset{\rightarrow}{srvc}(D)=s$.
\end{enumerate}
\end{thm}

\begin{proof}
(a) We may simply consider the digraphs $D_n$ as described after the proof of Theorem \ref{pro2}. That is, $D_n$ is the digraph on $n\ge 3$ vertices, obtained by expanding a vertex of $\rC_3$ into $\lrK_{n-2}$. We have $\overset{\rightarrow}{\smash{t}rc}(D_n)=\overset{\rightarrow}{s\smash{t}rc}(D_n)=\overset{\rightarrow}{rc}(D_n)=\overset{\rightarrow}{src}(D_n)=3$.\\[1ex]
\indent (b) For $s\ge 13$, let $G_s$ be the simple graph with $s$ disjoint triangles attached to $K_s$, as described above. Let $D_s=\lrG_s$. Then by (\ref{prop3eq}), (\ref{maxeq1}), and (\ref{LLLSrmk}), we have $trc(G_s)\ge \overset{\rightarrow}{\smash{t}rc}(D_s)\ge \overset{\rightarrow}{rvc}(D_s)= rvc(G_s)=s$. In \cite{LMS2014}, Theorem 11, it was proved that $trc(G_s)=s$. Thus we have $\overset{\rightarrow}{\smash{t}rc}(D_s) = \overset{\rightarrow}{rvc}(D_s)=s$.\\[1ex]
\indent(c) For $s\ge 13$, we construct the simple graph $H_s$ as follows. First, we take the graph $G_s$, and let $u_1,\dots, u_s$ be the vertices of the $K_s$, and the remaining vertices are $v_i, w_i$, where $u_iv_iw_i$ is a triangle,  for $1\le i\le s$. We then add new vertices $z_1,\dots,z_s$, and add the edges $u_iz_i, u_{i+1}z_i, v_iz_i, w_iz_{i+4}$, for $1\le i\le s$. Throughout, all indices are taken modulo $s$. See Figure 3 for the case $s=13$.\\[1ex]
\[\unit = 0.6cm
%vertices
%ui
\pt{ 0 }{ 4.5 }
\pt{ 2.0912542742 }{ 3.98455211544 }
\pt{ 3.70342739652 }{ 2.55629136029 }
\pt{ 4.46718993344 }{ 0.542415061149 }
\pt{ 4.20757309208 }{ -1.59572199169 }
\pt{ 2.98405196208 }{ -3.36829836677 }
\pt{ 1.07692048929 }{ -4.36923817842 }
\pt{ -1.07692048929 }{ -4.36923817842 }
\pt{ -2.98405196208 }{ -3.36829836677 }
\pt{ -4.20757309208 }{ -1.59572199169 }
\pt{ -4.46718993344 }{ 0.542415061149 }
\pt{ -3.70342739652 }{ 2.55629136029 }
\pt{ -2.0912542742 }{ 3.98455211544 }
%
%viwi
\pt{ -0.723220081532 }{ 5.95625324459 }
\pt{ 0.723220081532 }{ 5.95625324459 }
\pt{ 2.12762932226 }{ 5.61009745611 }
\pt{ 3.40838848039 }{ 4.93790319536 }
\pt{ 4.49106448903 }{ 3.97873594944 }
\pt{ 5.31273615392 }{ 2.78833903226 }
\pt{ 5.82565090456 }{ 1.43589398573 }
\pt{ 6 }{ 0 }
\pt{ 5.82565090456 }{ -1.43589398573 }
\pt{ 5.31273615392 }{ -2.78833903226 }
\pt{ 4.49106448903 }{ -3.97873594944 }
\pt{ 3.40838848039 }{ -4.93790319536 }
\pt{ 2.12762932226 }{ -5.61009745611 }
\pt{ 0.723220081532 }{ -5.95625324459 }
\pt{ -0.723220081532 }{ -5.95625324459 }
\pt{ -2.12762932226 }{ -5.61009745611 }
\pt{ -3.40838848039 }{ -4.93790319536 }
\pt{ -4.49106448903 }{ -3.97873594944 }
\pt{ -5.31273615392 }{ -2.78833903226 }
\pt{ -5.82565090456 }{ -1.43589398573 }
\pt{ -6 }{ 0 }
\pt{ -5.82565090456 }{ 1.43589398573 }
\pt{ -5.31273615392 }{ 2.78833903226 }
\pt{ -4.49106448903 }{ 3.97873594944 }
\pt{ -3.40838848039 }{ 4.93790319536 }
\pt{ -2.12762932226 }{ 5.61009745611 }
\pt{ 1.26837302072 }{ 5.14599163236 }
\pt{ 3.51455008868 }{ 3.96710696531 }
\pt{ 4.95558608623 }{ 1.87940590133 }
\pt{ 5.26135703272 }{ -0.638844405353 }
\pt{ 4.36181448924 }{ -3.01074315768 }
\pt{ 2.46303281183 }{ -4.69291693596 }
\pt{ 0 }{ -5.3 }
\pt{ -2.46303281183 }{ -4.69291693596 }
\pt{ -4.36181448924 }{ -3.01074315768 }
\pt{ -5.26135703272 }{ -0.638844405353 }
\pt{ -4.95558608623 }{ 1.87940590133 }
\pt{ -3.51455008868 }{ 3.96710696531 }
\pt{ -1.26837302072 }{ 5.14599163236 }
%K13
%\varline{1000}{0.6}
\dotline
\dl{ 0 }{ 4.5 }{ 2.0912542742 }{ 3.98455211544 }
\dl{ 2.0912542742 }{ 3.98455211544 }{ 3.70342739652 }{ 2.55629136029 }
\dl{ 3.70342739652 }{ 2.55629136029 }{ 4.46718993344 }{ 0.542415061149 }
\dl{ 4.46718993344 }{ 0.542415061149 }{ 4.20757309208 }{ -1.59572199169 }
\dl{ 4.20757309208 }{ -1.59572199169 }{ 2.98405196208 }{ -3.36829836677 }
\dl{ 2.98405196208 }{ -3.36829836677 }{ 1.07692048929 }{ -4.36923817842 }
\dl{ 1.07692048929 }{ -4.36923817842 }{ -1.07692048929 }{ -4.36923817842 }
\dl{ -1.07692048929 }{ -4.36923817842 }{ -2.98405196208 }{ -3.36829836677 }
\dl{ -2.98405196208 }{ -3.36829836677 }{ -4.20757309208 }{ -1.59572199169 }
\dl{ -4.20757309208 }{ -1.59572199169 }{ -4.46718993344 }{ 0.542415061149 }
\dl{ -4.46718993344 }{ 0.542415061149 }{ -3.70342739652 }{ 2.55629136029 }
\dl{ -3.70342739652 }{ 2.55629136029 }{ -2.0912542742 }{ 3.98455211544 }
\dl{ -2.0912542742 }{ 3.98455211544 }{ 0 }{ 4.5 }
\thnline
\dl{ 0 }{ 4.5 }{ -0.723220081532 }{ 5.95625324459 }
\dl{ 2.0912542742 }{ 3.98455211544 }{ 2.12762932226 }{ 5.61009745611 }
\dl{ 3.70342739652 }{ 2.55629136029 }{ 4.49106448903 }{ 3.97873594944 }
\dl{ 4.46718993344 }{ 0.542415061149 }{ 5.82565090456 }{ 1.43589398573 }
\dl{ 4.20757309208 }{ -1.59572199169 }{ 5.82565090456 }{ -1.43589398573 }
\dl{ 2.98405196208 }{ -3.36829836677 }{ 4.49106448903 }{ -3.97873594944 }
\dl{ 1.07692048929 }{ -4.36923817842 }{ 2.12762932226 }{ -5.61009745611 }
\dl{ -1.07692048929 }{ -4.36923817842 }{ -0.723220081532 }{ -5.95625324459 }
\dl{ -2.98405196208 }{ -3.36829836677 }{ -3.40838848039 }{ -4.93790319536 }
\dl{ -4.20757309208 }{ -1.59572199169 }{ -5.31273615392 }{ -2.78833903226 }
\dl{ -4.46718993344 }{ 0.542415061149 }{ -6 }{ 0 }
\dl{ -3.70342739652 }{ 2.55629136029 }{ -5.31273615392 }{ 2.78833903226 }
\dl{ -2.0912542742 }{ 3.98455211544 }{ -3.40838848039 }{ 4.93790319536 }
\dl{ -0.723220081532 }{ 5.95625324459 }{ 0.723220081532 }{ 5.95625324459 }
\dl{ 2.12762932226 }{ 5.61009745611 }{ 3.40838848039 }{ 4.93790319536 }
\dl{ 4.49106448903 }{ 3.97873594944 }{ 5.31273615392 }{ 2.78833903226 }
\dl{ 5.82565090456 }{ 1.43589398573 }{ 6 }{ 0 }
\dl{ 5.82565090456 }{ -1.43589398573 }{ 5.31273615392 }{ -2.78833903226 }
\dl{ 4.49106448903 }{ -3.97873594944 }{ 3.40838848039 }{ -4.93790319536 }
\dl{ 2.12762932226 }{ -5.61009745611 }{ 0.723220081532 }{ -5.95625324459 }
\dl{ -0.723220081532 }{ -5.95625324459 }{ -2.12762932226 }{ -5.61009745611 }
\dl{ -3.40838848039 }{ -4.93790319536 }{ -4.49106448903 }{ -3.97873594944 }
\dl{ -5.31273615392 }{ -2.78833903226 }{ -5.82565090456 }{ -1.43589398573 }
\dl{ -6 }{ 0 }{ -5.82565090456 }{ 1.43589398573 }
\dl{ -5.31273615392 }{ 2.78833903226 }{ -4.49106448903 }{ 3.97873594944 }
\dl{ -3.40838848039 }{ 4.93790319536 }{ -2.12762932226 }{ 5.61009745611 }
\dl{ 0.723220081532 }{ 5.95625324459 }{ 0 }{ 4.5 }
\dl{ 3.40838848039 }{ 4.93790319536 }{ 2.0912542742 }{ 3.98455211544 }
\dl{ 5.31273615392 }{ 2.78833903226 }{ 3.70342739652 }{ 2.55629136029 }
\dl{ 6 }{ 0 }{ 4.46718993344 }{ 0.542415061149 }
\dl{ 5.31273615392 }{ -2.78833903226 }{ 4.20757309208 }{ -1.59572199169 }
\dl{ 3.40838848039 }{ -4.93790319536 }{ 2.98405196208 }{ -3.36829836677 }
\dl{ 0.723220081532 }{ -5.95625324459 }{ 1.07692048929 }{ -4.36923817842 }
\dl{ -2.12762932226 }{ -5.61009745611 }{ -1.07692048929 }{ -4.36923817842 }
\dl{ -4.49106448903 }{ -3.97873594944 }{ -2.98405196208 }{ -3.36829836677 }
\dl{ -5.82565090456 }{ -1.43589398573 }{ -4.20757309208 }{ -1.59572199169 }
\dl{ -5.82565090456 }{ 1.43589398573 }{ -4.46718993344 }{ 0.542415061149 }
\dl{ -4.49106448903 }{ 3.97873594944 }{ -3.70342739652 }{ 2.55629136029 }
\dl{ -2.12762932226 }{ 5.61009745611 }{ -2.0912542742 }{ 3.98455211544 }
\dl{ 0 }{ 4.5 }{ 1.26837302072 }{ 5.14599163236 }
\dl{ 1.26837302072 }{ 5.14599163236 }{ 2.0912542742 }{ 3.98455211544 }
\dl{ 2.0912542742 }{ 3.98455211544 }{ 3.51455008868 }{ 3.96710696531 }
\dl{ 3.51455008868 }{ 3.96710696531 }{ 3.70342739652 }{ 2.55629136029 }
\dl{ 3.70342739652 }{ 2.55629136029 }{ 4.95558608623 }{ 1.87940590133 }
\dl{ 4.95558608623 }{ 1.87940590133 }{ 4.46718993344 }{ 0.542415061149 }
\dl{ 4.46718993344 }{ 0.542415061149 }{ 5.26135703272 }{ -0.638844405353 }
\dl{ 5.26135703272 }{ -0.638844405353 }{ 4.20757309208 }{ -1.59572199169 }
\dl{ 4.20757309208 }{ -1.59572199169 }{ 4.36181448924 }{ -3.01074315768 }
\dl{ 4.36181448924 }{ -3.01074315768 }{ 2.98405196208 }{ -3.36829836677 }
\dl{ 2.98405196208 }{ -3.36829836677 }{ 2.46303281183 }{ -4.69291693596 }
\dl{ 2.46303281183 }{ -4.69291693596 }{ 1.07692048929 }{ -4.36923817842 }
\dl{ 1.07692048929 }{ -4.36923817842 }{ 0 }{ -5.3 }
\dl{ 0 }{ -5.3 }{ -1.07692048929 }{ -4.36923817842 }
\dl{ -1.07692048929 }{ -4.36923817842 }{ -2.46303281183 }{ -4.69291693596 }
\dl{ -2.46303281183 }{ -4.69291693596 }{ -2.98405196208 }{ -3.36829836677 }
\dl{ -2.98405196208 }{ -3.36829836677 }{ -4.36181448924 }{ -3.01074315768 }
\dl{ -4.36181448924 }{ -3.01074315768 }{ -4.20757309208 }{ -1.59572199169 }
\dl{ -4.20757309208 }{ -1.59572199169 }{ -5.26135703272 }{ -0.638844405353 }
\dl{ -5.26135703272 }{ -0.638844405353 }{ -4.46718993344 }{ 0.542415061149 }
\dl{ -4.46718993344 }{ 0.542415061149 }{ -4.95558608623 }{ 1.87940590133 }
\dl{ -4.95558608623 }{ 1.87940590133 }{ -3.70342739652 }{ 2.55629136029 }
\dl{ -3.70342739652 }{ 2.55629136029 }{ -3.51455008868 }{ 3.96710696531 }
\dl{ -3.51455008868 }{ 3.96710696531 }{ -2.0912542742 }{ 3.98455211544 }
\dl{ -2.0912542742 }{ 3.98455211544 }{ -1.26837302072 }{ 5.14599163236 }
\dl{ -1.26837302072 }{ 5.14599163236 }{ 0 }{ 4.5 }
\dl{ -0.723220081532 }{ 5.95625324459 }{ 1.26837302072 }{ 5.14599163236 }
\dl{ 2.12762932226 }{ 5.61009745611 }{ 3.51455008868 }{ 3.96710696531 }
\dl{ 4.49106448903 }{ 3.97873594944 }{ 4.95558608623 }{ 1.87940590133 }
\dl{ 5.82565090456 }{ 1.43589398573 }{ 5.26135703272 }{ -0.638844405353 }
\dl{ 5.82565090456 }{ -1.43589398573 }{ 4.36181448924 }{ -3.01074315768 }
\dl{ 4.49106448903 }{ -3.97873594944 }{ 2.46303281183 }{ -4.69291693596 }
\dl{ 2.12762932226 }{ -5.61009745611 }{ 0 }{ -5.3 }
\dl{ -0.723220081532 }{ -5.95625324459 }{ -2.46303281183 }{ -4.69291693596 }
\dl{ -3.40838848039 }{ -4.93790319536 }{ -4.36181448924 }{ -3.01074315768 }
\dl{ -5.31273615392 }{ -2.78833903226 }{ -5.26135703272 }{ -0.638844405353 }
\dl{ -6 }{ 0 }{ -4.95558608623 }{ 1.87940590133 }
\dl{ -5.31273615392 }{ 2.78833903226 }{ -3.51455008868 }{ 3.96710696531 }
\dl{ -3.40838848039 }{ 4.93790319536 }{ -1.26837302072 }{ 5.14599163236 }
\dl{ 0.723220081532 }{ 5.95625324459 }{ 4.36181448924 }{ -3.01074315768 }
\dl{ 3.40838848039 }{ 4.93790319536 }{ 2.46303281183 }{ -4.69291693596 }
\dl{ 5.31273615392 }{ 2.78833903226 }{ 0 }{ -5.3 }
\dl{ 6 }{ 0 }{ -2.46303281183 }{ -4.69291693596 }
\dl{ 5.31273615392 }{ -2.78833903226 }{ -4.36181448924 }{ -3.01074315768 }
\dl{ 3.40838848039 }{ -4.93790319536 }{ -5.26135703272 }{ -0.638844405353 }
\dl{ 0.723220081532 }{ -5.95625324459 }{ -4.95558608623 }{ 1.87940590133 }
\dl{ -2.12762932226 }{ -5.61009745611 }{ -3.51455008868 }{ 3.96710696531 }
\dl{ -4.49106448903 }{ -3.97873594944 }{ -1.26837302072 }{ 5.14599163236 }
\dl{ -5.82565090456 }{ -1.43589398573 }{ 1.26837302072 }{ 5.14599163236 }
\dl{ -5.82565090456 }{ 1.43589398573 }{ 3.51455008868 }{ 3.96710696531 }
\dl{ -4.49106448903 }{ 3.97873594944 }{ 4.95558608623 }{ 1.87940590133 }
\dl{ -2.12762932226 }{ 5.61009745611 }{ 5.26135703272 }{ -0.638844405353 }
%
%vertex labels
\point{-0.3}{4}{\small $u_1$}
\point{ 1.77 }{ 3.55}{\small $u_2$}
\point{ 3.1 }{ 2.2 }{\small $u_3$}
\point{ 3.75 }{ 0.43 }{\small $u_4$}
\point{ 3.48 }{ -1.6 }{\small $u_5$}
\point{ 2.5 }{ -3.1 }{\small $u_6$}
\point{ 0.75 }{ -4.08 }{\small $u_7$}
\point{ -1.2 }{ -4.08 }{\small $u_8$}
\point{ -2.95 }{ -3.2 }{\small $u_9$}
\point{ -4.1 }{ -1.6 }{\small $u_{10}$}
\point{ -4.3 }{ 0.43 }{\small $u_{11}$}
\point{ -3.65 }{ 2.2 }{\small $u_{12}$}
\point{ -2.05 }{ 3.55 }{\small $u_{13}$}
\point{ -0.95 }{ 6.2 }{\small $v_1$}
\point{ 0.47 }{ 6.2 }{\small $w_1$}
\point{ 1.9 }{ 5.85 }{\small $v_2$}
\point{ 3.2 }{ 5.13 }{\small $w_2$}
\point{ 4.3 }{ 4.2 }{\small $v_3$}
\point{ 5.45 }{ 2.6 }{\small $w_3$}
\point{ 5.95 }{ 1.44 }{\small $v_4$}
\point{ 6.15 }{ -0.1 }{\small $w_4$}
\point{ 5.98 }{ -1.55 }{\small $v_5$}
\point{ 5.45 }{ -3 }{\small $w_5$}
\point{ 4.65 }{ -4.1 }{\small $v_6$}
\point{ 3.2 }{ -5.4 }{\small $w_6$}
\point{ 2 }{ -6.1 }{\small $v_7$}
\point{ 0.4 }{ -6.4 }{\small $w_7$}
\point{ -0.95 }{ -6.4 }{\small $v_8$}
\point{ -2.55 }{ -6.1 }{\small $w_8$}
\point{ -3.7 }{ -5.4 }{\small $v_9$}
\point{ -5.3 }{ -4.1 }{\small $w_9$}
\point{ -6.18 }{ -3 }{\small $v_{10}$}
\point{ -6.85 }{ -1.55 }{\small $w_{10}$}
\point{ -6.88 }{ -0.1 }{\small $v_{11}$}
\point{ -6.8 }{ 1.44 }{\small $w_{11}$}
\point{ -6.25 }{ 2.6 }{\small $v_{12}$}
\point{ -5.5 }{ 4.2 }{\small $w_{12}$}
\point{ -4.3 }{ 5.13 }{\small $v_{13}$}
\point{ -2.6 }{ 5.85 }{\small $w_{13}$}
\point{ 1.15 }{ 5.35 }{\small $z_1$}
\point{ 3.52 }{ 4.2 }{\small $z_2$}
\point{ 5.08 }{ 1.88 }{\small $z_3$}
\point{ 5.4 }{ -0.75 }{\small $z_4$}
\point{ 4.4 }{ -3.3 }{\small $z_5$}
\point{ 2.3 }{ -5.1 }{\small $z_6$}
\point{ -0.3 }{ -5.7 }{\small $z_7$}
\point{ -2.8 }{ -5.1 }{\small $z_8$}
\point{ -4.98 }{ -3.3 }{\small $z_9$}
\point{ -6.15 }{ -0.75 }{\small $z_{10}$}
\point{ -5.85 }{ 1.88 }{\small $z_{11}$}
\point{ -3.9 }{ 4.2 }{\small $z_{12}$}
\point{ -1.5 }{ 5.35 }{\small $z_{13}$}
\point{ -0.4 }{ -0.2 }{\small $K_{13}$}
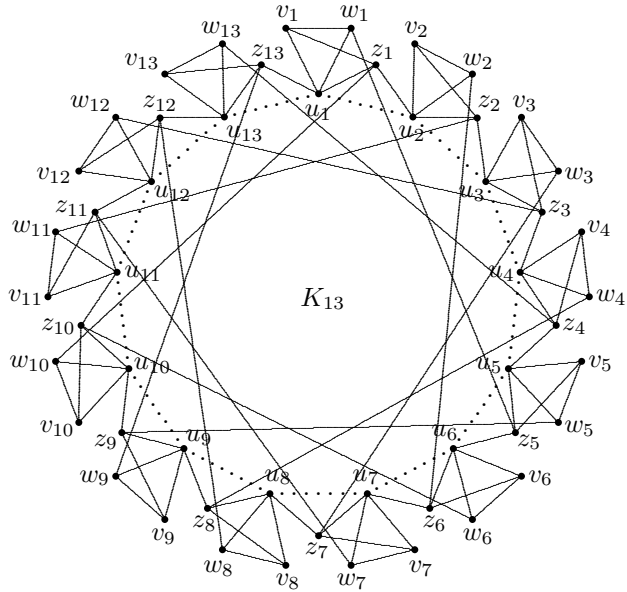
\ptlu{0}{-8}{\textup{\small Figure 3. The graph $H_s$ when $s=13$. The polygon with dotted lines represents the copy of $K_s$.}}
\]\\[-2ex]

Now let $D_s=\lrH_s$. Then by (\ref{prop3eq}), (\ref{maxeq2}), and (\ref{LLLSrmk}), we have $strc(H_s)\ge \overset{\rightarrow}{s\smash{t}rc}(D_s)\ge \overset{\rightarrow}{srvc}(D_s)= srvc(H_s)$. Thus it suffices to prove that $srvc(H_s)\ge s$ and $strc(H_s)\le s$, which would imply $\overset{\rightarrow}{s\smash{t}rc}(D_s)=\overset{\rightarrow}{srvc}(D_s)=s$.

We first prove that $srvc(H_s)\ge s$. Suppose that we have a vertex-colouring of $H_s$, using fewer than $s$ colours. Then some two vertices of $u_1,\dots, u_s$ have the same colour. We may assume that $u_1$ and $u_i$ have the same colour, for some $2\le i\le s$. If $i\not\in\{5,6,s-4,s-3\}$, then the unique $w_1-w_i$ geodesic is $w_1u_1u_iw_i$. If $i\in\{5,6\}$, then the unique $v_1-w_i$ geodesic is $v_1u_1u_iw_i$. If $i\in\{s-4,s-3\}$, then the unique $w_1-v_i$ geodesic is $w_1u_1u_iv_i$. In each case, we have two vertices in $H_s$ such that there is no vertex-rainbow geodesic connecting them. Thus $srvc(H_s)\ge s$.

Now we prove that $strc(H_s)\le s$. We define a total-colouring $c$ of $H_s$, using the colours $1,\dots,s$, as follows. For $1\le i\le s$, let $c(u_i)=c(w_iz_{i+4})=i$, $c(u_iv_i)=i+1$, $c(v_i)=i+2$, $c(v_iw_i)=c(u_{i+1}z_i)=i+3$, $c(w_i)=c(z_i)=i+4$, $c(w_iu_i)=c(u_iz_i)=c(v_iz_i)=i+5$, all modulo $s$. For $i\neq j$, let $c(u_iu_j)$ be a colour not in $\{i,\dots,i+5\}\cup\{j,\dots,j+5\}$. Such a colour exists since $s\ge 13$. We will show that $c$ is a strongly total rainbow connected colouring for $H_s$, which implies that $strc(H_s)\le s$. Let $x,y\in V(H_s)$. We show that there always exists a total-rainbow $x-y$ geodesic $P$. If at least one of $x,y$ belongs to $\{u_1,\dots,u_s\}$, or if $x,y\in \{v_i,w_i,z_i\}$ for some $i$, then clearly such an $x-y$ geodesic exists, with length at most $2$. Thus, it suffices to consider $x\in\{v_1,w_1,z_1\}$ and the six cases as shown in the following table, where $2\le i\le s$ in each case.

\begin{table}[ht]
\begin{center}
\begin{tabular}{|c|c|ll|l|}
\hline
$x$ & $y$ &
\multicolumn{2}{c}{$P$}&\multicolumn{1}{|c|}{Internal colours of $P$}\\
\hline
\multirow{3}{*}{$v_1$} & \multirow{3}{*}{$v_i$} & $v_1u_1u_iv_i$ & if $i\not\in\{2,s\}$ & $2,1,t,i,i+1$\\
&& $v_1z_1u_2v_2$ & if $i=2$ & $6,5,4,2,3$\\
&& $v_1u_1z_sv_s$ & if $i=s$ & $2,1,3,4,5$\\
\hline
\multirow{4}{*}{$v_1$} & \multirow{4}{*}{$w_i$} & $v_1u_1u_iw_i$ & if $i\not\in\{2,s-4,s-3\}$ & $2,1,t,i,i+5$\\
&& $v_1z_1u_2w_2$ & if $i=2$ & $6,5,4,2,7$\\
&& $v_1u_1z_sw_{s-4}$ & if $i=s-4$ & $2,1,3,4,s-4$\\
&& $v_1z_1w_{s-3}$ & if $i=s-3$ & $6,5,s-3$\\
\hline
\multirow{4}{*}{$v_1$} & \multirow{4}{*}{$z_i$} & $v_1u_1u_iz_i$ & if $i\not\in\{2,5,s-4,s-3,s\}$ & $2,1,t,i,i+5$\\
&& $v_1u_1u_{i+1}z_i$ & if $i\in\{2,s-4,s-3\}$ & $2,1,t,i+1,i+3$\\
&& $v_1w_1z_5$ & if $i=5$ & $4,5,1$\\
&& $v_1u_1z_s$ & if $i=s$ & $2,1,3$\\
\hline
\multirow{3}{*}{$w_1$} & \multirow{3}{*}{$w_i$} & $w_1u_1u_iw_i$ & if $i\not\in\{6,s-4\}$ & $6,1,t,i,i+5$\\
&& $w_1z_5u_6w_6$ & if $i=6$ & $1,9,8,6,11$\\
&& $w_1u_1z_sw_{s-4}$ & if $i=s-4$ & $6,1,3,4,s-4$\\
\hline
\multirow{4}{*}{$w_1$} & \multirow{4}{*}{$z_i$} & $w_1u_1u_iz_i$ & if $i\not\in\{5,6,s-4,s\}$ & $6,1,t,i,i+5$\\
&& $w_1z_5$ & if $i=5$ & $1$\\
&& $w_1u_1u_{i+1}z_i$ & if $i\in\{6,s-4\}$ & $6,1,t,i+1,i+3$\\
&& $w_1u_1z_s$ & if $i=s$ & $6,1,3$\\
\hline
\multirow{4}{*}{$z_1$} & \multirow{4}{*}{$z_i$} & $z_1u_1u_iz_i$ & if $i\not\in\{2,6,s-4,s\}$ & $6,1,t,i,i+5$\\
&& $z_1u_2z_2$ & if $i=2$ & $4,2,7$\\
&& $z_1u_1u_{i+1}z_i$ & if $i\in\{6,s-4\}$ & $6,1,t,i+1,i+3$\\
&& $z_1u_1z_s$ & if $i=s$ & $6,1,3$\\
\hline
\end{tabular}
\end{center}
\indent\\[-1.9cm]
\end{table}
\indent\\[0.5cm]
\indent In each case, we have a desired $x-y$ geodesic $P$, where $t$ is some colour different from the remaining four colours. For example, in the very first case, we have $t\not\in\{1,2,i,i+1\}$. The proof is thus complete.
\end{proof}

We may also consider how far from equality we can be in (\ref{maxeq1}) and (\ref{maxeq2}). In the following result, we see that there is an infinite family of digraphs $\mathcal D$ such that $\overset{\rightarrow}{\smash{t}rc}(D)$ is unbounded on $\mathcal D$, while $\overset{\rightarrow}{rc}(D)$ is bounded. Similar results also hold for $\overset{\rightarrow}{\smash{t}rc}(D)$ in comparison with $\overset{\rightarrow}{rvc}(D)$, and for $\overset{\rightarrow}{s\smash{t}rc}(D)$ in comparison with each of $\overset{\rightarrow}{src}(D)$ and $\overset{\rightarrow}{srvc}(D)$.

\begin{thm}\label{compfarthm}
\indent\\[-3ex]
\begin{enumerate}
\item[(a)] Given $s\ge 2$, there exists a strongly connected digraph $D$ such that $\overset{\rightarrow}{s\smash{t}rc}(D)\ge \overset{\rightarrow}{\smash{t}rc}(D)\ge s$ and $\overset{\rightarrow}{rc}(D)=\overset{\rightarrow}{src}(D)=3$.
\item[(b)]Given $s\ge 4$, there exists a strongly connected digraph $D$ such that $\overset{\rightarrow}{\smash{t}rc}(D)=\overset{\rightarrow}{s\smash{t}rc}(D)\ge s$ and $\overset{\rightarrow}{rvc}(D)=\overset{\rightarrow}{srvc}(D)=3$.
\end{enumerate}
\end{thm}

\begin{proof}
(a) Let $F_s$ be the simple graph consisting of $K_s$ with a pendent edge attached to each of the $s$ vertices of $K_s$, and $D_s=\lrF_s$. Let $u_1,\dots,u_s$ be the vertices of the $K_s$, and $v_1,\dots,v_s$ be the pendent vertices, where $u_iv_i\in E(F_s)$ for $1\le i\le s$. Then in any total rainbow connected colouring of $D_s$, the vertices $u_i$ must receive distinct colours, and thus by (\ref{eq}), $\overset{\rightarrow}{s\smash{t}rc}(D)\ge \overset{\rightarrow}{\smash{t}rc}(D)\ge s$. Also, the arc-colouring $c$ of $D_s$ where $c(u_iu_j)=1$ for all $1\le i\neq  j\le s$; $c(u_iv_i)=2$ and $c(v_iu_i)=3$, for $1\le i\le s$, is strongly rainbow connected, and thus $\overset{\rightarrow}{src}(D)\le 3$. Since $\textup{diam}(D_s)=3$, we have $\overset{\rightarrow}{rc}(D)\ge 3$, and thus $\overset{\rightarrow}{rc}(D)=\overset{\rightarrow}{src}(D)=3$.\\[1ex]
\indent(b) Let $D$ be the digraph consisting of $t={s-1 \choose 3}+1$ copies of the triangle $\rC_3$, all having one vertex in common. Let $V(D)=\{v,x_1,y_1,\dots,x_t,y_t\}$, and let the arcs of $A(D)$ be $vx_i,x_iy_i,y_iv$, for $1\le i\le t$. It was proved in Lemma 6 of \cite{LLLS2017} that $\overset{\rightarrow}{rc}(D)\ge s$, and $\overset{\rightarrow}{rvc}(D)=\overset{\rightarrow}{srvc}(D)=3$. Now, note that for any two vertices $x,y\in V(D)$, there is a unique $x-y$ path, and thus $\overset{\rightarrow}{\smash{t}rc}(D)=\overset{\rightarrow}{s\smash{t}rc}(D)$. The assertion then follows from $\overset{\rightarrow}{\smash{t}rc}(D)\ge \overset{\rightarrow}{rc}(D)\ge s$.
\end{proof}

Finally, we consider how far apart the quantities $\overset{\rightarrow}{\smash{t}rc}(D)$ and $\max(\overset{\rightarrow}{rc}(D),\overset{\rightarrow}{rvc}(D))$ in (\ref{maxeq1}) can be, and similarly for $\overset{\rightarrow}{s\smash{t}rc}(D)$ and $\max(\overset{\rightarrow}{src}(D),\overset{\rightarrow}{srvc}(D))$ in (\ref{maxeq2}). For example, by considering the bioriented path $\lrP_n$, we have $\overset{\rightarrow}{\smash{t}rc}(\lrP_n)=\overset{\rightarrow}{s\smash{t}rc}(\lrP_n)=2n-3$ and $\max(\overset{\rightarrow}{rc}(\lrP_n),\overset{\rightarrow}{rvc}(\lrP_n))=\max(\overset{\rightarrow}{src}(\lrP_n),\overset{\rightarrow}{srvc}(\lrP_n))=n-1$, so that
\[
\overset{\rightarrow}{\smash{t}rc}(\lrP_n)-\max(\overset{\rightarrow}{rc}(\lrP_n),\overset{\rightarrow}{rvc}(\lrP_n))=\overset{\rightarrow}{s\smash{t}rc}(\lrP_n)-\max(\overset{\rightarrow}{src}(\lrP_n),\overset{\rightarrow}{srvc}(\lrP_n))=n-2,
\]
and each difference can be arbitrarily large. However, in this example, all four quantities in consideration are unbounded in $n$. Thus, we propose the following problem.
%
%Thus, a better question to ask would be the following.
\begin{prob}\label{sepprob}
Does there exist an infinite family of digraphs $\mathcal D$ such that $\overset{\rightarrow}{\smash{t}rc}(D)$ is unbounded on $\mathcal D$, while $\max(\overset{\rightarrow}{rc}(D),\overset{\rightarrow}{rvc}(D))$ is bounded? Similarly, does there exist an infinite family of digraphs $\mathcal D$ such that $\overset{\rightarrow}{s\smash{t}rc}(D)$ is unbounded on $\mathcal D$, while $\max(\overset{\rightarrow}{src}(D),\overset{\rightarrow}{srvc}(D))$ is bounded?
\end{prob}

%We have not been able to find such an infinite family in Problem \ref{sepprob}, and we believe that this problem is not simple. \textcolor{red}{Check that Problem \ref{sepprob} really is not easy.}

\section{Total rainbow connection of some specific digraphs}\label{specsect}

In this section, we determine the (strong) total rainbow connection number of some specific digraphs. For $n\ge 3$, the \emph{wheel} $W_n$ is the graph obtained by taking the cycle $C_n$, and joining a new vertex $v$ to every vertex of $C_n$. The vertex $v$ is the \emph{centre} of $W_n$. For $t \ge 2$, let $K_{n_1,\dots,n_t}$ denote the complete $t$-partite graph with class-sizes $n_1,\dots,n_t$. The following theorem determines the two parameters for the biorientations of paths, cycles, wheels, and complete multipartite graphs.

\begin{thm}\label{thm2}
\indent\\[-3ex]
\begin{enumerate}
\item[(a)] For $n\geq 2$, $\overset{\rightarrow}{\smash{t}rc}(\lrP_n)=\overset{\rightarrow}{s\smash{t}rc}(\lrP_n)=2n-3$.
\item[(b)] For $n\geq 3$, $\overset{\rightarrow}{\smash{t}rc}(\lrC_n)=\overset{\rightarrow}{s\smash{t}rc}(\lrC_n)=g(n)$, where
\[
g(n)=\left\{
\begin{array}{ll}
n-2 & \textup{\emph{if} $n=3,5$\emph{,}}\\
n-1 & \textup{\emph{if} $n=4,6,7,8,9,10,12$\emph{,}}\\
n & \textup{\emph{if} $n=11$ \emph{or} $n\ge 13$.}
\end{array}
\right.
\]
\item[(c)] For $n\geq4$, $\overset{\rightarrow}{\smash{t}rc}(\lrWn)=\overset{\rightarrow}{s\smash{t}rc}(\lrWn)=3$.
\item[(d)] Let $t\geq2$, and let $\lrK_{n_{1},n_{2},\ldots,n_{t}}$ be a complete $t$-partite digraph with $n_{i}\geq2$ for some $i$. Then, $\overset{\rightarrow}{\smash{t}rc}(\lrK_{n_{1},n_{2},\ldots,n_{t}})=
\overset{\rightarrow}{s\smash{t}rc}(\lrK_{n_{1},n_{2},\ldots,n_{t}})=3$.
\end{enumerate}
\end{thm}

\begin{proof}
(a) In \cite{CLLL2016}, Proposition 6, it was shown that $strc(P_{n})=2n-3$ for $n\ge 2$. Since $\textup{diam}(\lrP_n)=n-1$, it follows from (\ref{eq}) and (\ref{prop3eq}) that
\[
2n-3\leq \overset{\rightarrow}{\smash{t}rc}(\lrP_n)\leq \overset{\rightarrow}{s\smash{t}rc}(\lrP_n)\leq
strc(P_{n})=2n-3,
\]
and thus part (a) holds.\\[1ex]
\indent (b) For $n\ge 3$, it was shown in \cite{LMS2014}, Theorem 2, that $trc(C_n)=g(n)$; and in \cite{CLLL2016}, Theorem 8, that  $strc(C_n)=trc(C_n)$. Therefore, $\overset{\rightarrow}{\smash{t}rc}(\lrC_n)\le\overset{\rightarrow}{s\smash{t}rc}(\lrC_n)\le strc(C_n)=g(n)$. It suffices to prove that $\overset{\rightarrow}{\smash{t}rc}(\lrC_n)\ge g(n)$. We have
$\overset{\rightarrow}{\smash{t}rc}(\lrC_n)\geq 2\,\textup{diam}(\lrC_n)-1=2\lfloor\frac{n}{2}\rfloor-1$, and this gives $\overset{\rightarrow}{\smash{t}rc}(\lrC_n)\ge g(n)$ for $n=3,4,5,6,8,10,12$. To see that $\overset{\rightarrow}{\smash{t}rc}(\lrC_7)\geq6$ (resp.~$\overset{\rightarrow}{\smash{t}rc}(\lrC_9)\geq 8)$, one can check that in any total-colouring of $\lrC_7$ (resp.~$\lrC_9$) with at most $5$ (resp.~$7$) colours, there exist vertices $u$ and $v$ with $d(u,v)=3$ (resp.~$d(u,v)=4$) such that the $u-v$ path with length $3$ (resp.~$4$) is not total-rainbow. Since the other $u-v$ path has length $4$ (resp.~$5$), it is also not total-rainbow. Hence, there does not exist a total-rainbow $u-v$ path, and the results for $n=7,9$ follow. Finally, let $n=11$ or $n\geq13$. Let $C_n=v_0v_1\cdots v_{n-1}v_0$. Suppose that we have a total-colouring of $\lrC_n$ with fewer than $n$ colours. Let $A=\{v_iv_{i+1}:0\le i\le n-1\}$. Then for some $m\in\{0,1,2,3\}$, we have $m$ arcs of $A$ and $3-m$ vertices with the same colour. Without loss of generality, for some $1<i\le\lfloor\frac{n-m}{3}\rfloor+1$, either $v_1$ and $v_i$, or $v_1$ and $v_iv_{i+1}$, or $v_0v_1$ and $v_iv_{i+1}$, have the same colour. Consider the two $v_0-v_{i+1}$ paths in $\lrC_n$. The path $v_0v_1\cdots v_{i+1}$ is not total-rainbow. The path $v_0v_{n-1}v_{n-2}\cdots v_{i+1}$ has at least $2n-2\lfloor\frac{n-m}{3}\rfloor-5>n-1$ arcs and internal vertices, and hence is also not total-rainbow. Therefore $\overset{\rightarrow}{\smash{t}rc}(\lrC_n)\geq n$.\\[1ex]
\indent (c) Let $v$ be the centre of $\lrWn$. Note that any total-colouring of $\lrWn$ where every in-arc (resp.~out-arc) of $v$ has colour $1$ (resp.~$2$), and $v$ has colour $3$, is strongly total rainbow connected. Since $n\ge 4$, we have $\textup{diam}(\lrWn)=2$, and $3\le \overset{\rightarrow}{\smash{t}rc}(\lrWn)\le \overset{\rightarrow}{s\smash{t}rc}(\lrWn)\le 3$ by (\ref{eq}). Thus the result holds.\\[1ex]
\indent (d) Let $V_{1},V_{2},\ldots,V_{t}$ be the partition classes of $\lrK_{n_{1},n_{2},\ldots,n_{t}}$. For each arc $uv$ with $u \in V_{i}$ and $v\in V_{j}$, assign colour $1$ to $uv$ if $i<j$ and colour $2$ if $i>j$, and  assign colour $3$ to each vertex of $\lrK_{n_{1},n_{2},\ldots,n_{k}}$. Then it is easy to check that this is a strongly total rainbow connected colouring. Also, since $n_{i}\geq2$ for some $i$, we have $\textup{diam}(\lrK_{n_{1},n_{2},\ldots,n_{t}})=2$. Hence by (\ref{eq}), we have $3\le\overset{\rightarrow}{\smash{t}rc}(\lrK_{n_{1},n_{2},\ldots,n_{k}})\le\overset{\rightarrow}{s\smash{t}rc}(\lrK_{n_{1},n_{2},\ldots,n_{k}})\le 3$, and the result follows.
\end{proof}

In the next result, we determine the (strong) total rainbow connection numbers for directed cycles.

\begin{thm}\label{thm3}
Let $n\geq3$. Then,
\[
\overset{\rightarrow}{\smash{t}rc}(\rC_n)=\overset{\rightarrow}{s\smash{t}rc}(\rC_n)=
\left\{
\begin{array}{ll}
3 & \textup{\emph{if} $n=3$\emph{,}}\\
6 & \textup{\emph{if} $n=4$\emph{,}}\\
2n & \textup{\emph{if} $n\ge 5$.}
\end{array}
\right.
\]
\end{thm}
\begin{proof}
Let $\rC_n=v_{0}v_1\cdots v_{n-1}v_0$. One can easily verify that $\overset{\rightarrow}{\smash{t}rc}(\rC_3)
=\overset{\rightarrow}{s\smash{t}rc}(\rC_3)=3$, and $\overset{\rightarrow}{\smash{t}rc}(\rC_4)=\overset{\rightarrow}{s\smash{t}rc}(\rC_4)=6$. Now let $n\ge 5$. By (\ref{eq}), we have $\overset{\rightarrow}{\smash{t}rc}(\rC_n)
\leq \overset{\rightarrow}{s\smash{t}rc}(\rC_n)\leq2n$, and hence it suffices to prove that $\overset{\rightarrow}{\smash{t}rc}(\rC_n)\ge 2n$. Suppose that we have a total rainbow connected colouring for $\rC_n$, using fewer than $2n$ colours. Then, there exist two elements of $V(\rC_n)\cup A(\rC_n)$ with the same colour. We recall that for $n\ge 5$, we have $\overset{\rightarrow}{rc}(\rC_n)=n$ from \cite{AM2016}, Theorem 4; and $\overset{\rightarrow}{rvc}(\rC_n)=n$ from \cite{LLLS2017}, Proposition 8. It follows that we must have a vertex and an arc with the same colour. We may assume that $v_1$ and $v_iv_{i+1}$ have the same colour, for some $0\le i\le n-1$. If $i\not\in\{0,n-1\}$, then the unique $v_0-v_{i+1}$ path is not total-rainbow. Otherwise, the unique $v_i-v_2$ path is not total-rainbow. We have a contradiction, and therefore, $\overset{\rightarrow}{\smash{t}rc}(\rC_n)\ge 2n$.
\end{proof}

\section{Tournaments}\label{tournamentssect}

We now study the (strong) total rainbow connection number of tournaments. Our first aim is to consider the range of values that the (strong) total rainbow connection number can take, over all strongly connected tournaments of a given order $n\ge 3$. Clearly, ${\rC}_3$ is the only such tournament of order $3$. For $n=4$, there is also a unique tournament, namely $T_4$, which is the union of ${\rC}_4=v_0v_1v_2v_3v_0$ and the arcs $v_0v_2,v_1v_3$. 

For the rainbow connection analogue, we have $\overset{\rightarrow}{rc}({\rC}_3)=\overset{\rightarrow}{src}({\rC}_3)=\overset{\rightarrow}{rc}(T_4)=\overset{\rightarrow}{src}(T_4)=3$. Dorbec, Schiermeyer, Sidorowicz, and Sopena \cite{PIE2014}; and Sidorowicz and Sopena \cite{SS2016}, proved the following results.

\begin{thm}\label{Dorbecetaltourthm1}\textup{\cite{PIE2014,SS2016}}
If $T$ is a strongly connected tournament with $n\ge 4$ vertices, then $2\le\overset{\rightarrow}{rc}(T)\le\overset{\rightarrow}{src}(T)\le n-1$.
\end{thm}

\begin{thm}\label{Dorbecetaltourthm2}\textup{\cite{PIE2014,SS2016}}
For every $n$ and $k$ such that $3\le k\le n-1$, there exists a tournament $T_{n,k}$ on $n$ vertices such that $\overset{\rightarrow}{rc}(T_{n,k})=\overset{\rightarrow}{src}(T_{n,k})=k$.
\end{thm}

It was also remarked in \cite{PIE2014} that there does not exist a tournament on $4$ or $5$ vertices with rainbow connection number $2$, and that such a tournament exists if the order is $8$ (mod $12)$. In response, Alva-Samos and Montellano-Ballesteros \cite{AM2017} proved the following result.

\begin{thm}\label{ASMBtourthm}\textup{\cite{AM2017}}
For every $n\ge 6$, there exists a tournament $T_n$ on $n$ vertices such that $\overset{\rightarrow}{rc}(T_n)=\overset{\rightarrow}{src}(T_n)=2$.
\end{thm}

Lei, Li, Liu, and Shi \cite{LLLS2017} considered both the rainbow vertex-connection and strong rainbow vertex-connection analogues. We have $\overset{\rightarrow}{rvc}({\rC}_3)=\overset{\rightarrow}{srvc}({\rC}_3)=1$, and $\overset{\rightarrow}{rvc}(T_4)=\overset{\rightarrow}{srvc}(T_4)=2$. For general values of $n$, they proved the following results.

\begin{thm}\label{LLLStourthm1}\textup{\cite{LLLS2017}}
If $T$ is a strongly connected tournament on $n\ge 3$ vertices, then $1\le \overset{\rightarrow}{rvc}(T)\le \overset{\rightarrow}{srvc}(T)\le n-2$.
\end{thm}

\begin{thm}\label{LLLStourthm2}\textup{\cite{LLLS2017}}
For $n\ge 5$ and $1\leq k \leq n-2$, there exists a tournament $T_{n,k}$ on $n$ vertices such that $\overset{\rightarrow}{rvc}(T_{n,k})=\overset{\rightarrow}{srvc}(T_{n,k})=k$.
\end{thm}

Here, we consider the same situation for the (strong) total rainbow connection number. Clearly, we have $\overset{\rightarrow}{\smash{t}rc}({\rC}_3)=\overset{\rightarrow}{s\smash{t}rc}({\rC}_3)=3$, and $\overset{\rightarrow}{\smash{t}rc}(T_4)=\overset{\rightarrow}{s\smash{t}rc}(T_4)=5$. Now, Theorem \ref{tournamentthm1} below is the analogous result to Theorems \ref{Dorbecetaltourthm1} and \ref{LLLStourthm1}.

\begin{thm}\label{tournamentthm1}
If $T$ is a strongly connected tournament on $n\ge 3$ vertices, then
\[
3\le \overset{\rightarrow}{\smash{t}rc}(T)\le \overset{\rightarrow}{s\smash{t}rc}(T)\le 2n-3.
\]
\end{thm}

\begin{proof}
The theorem holds for $n=3,4$. Now, let $n\ge 5$. Clearly by (\ref{eq}), we have $\overset{\rightarrow}{s\smash{t}rc}(T)\ge\overset{\rightarrow}{\smash{t}rc}(T)\ge 3$, so it remains to prove that $\overset{\rightarrow}{s\smash{t}rc}(T)\le 2n-3$.

Let $d=\textup{diam}(T)\ge 2$. If $d=2$, then by Theorem \ref{Dorbecetaltourthm1}, there exists a strongly rainbow connected arc-colouring of $T$, using at most $n-1$ colours. By assigning a new colour to all vertices of $T$, we obtain a strongly total rainbow connected total-colouring of $T$, using at most $n\le 2n-3$ colours. Indeed, if $x,y\in V(T)$, then either $xy\in A(T)$; or $xy\not\in A(T)$ and there exists a total-rainbow $x-y$ path of length $2$, which is also an $x-y$ geodesic.

Therefore, we may assume that $d\ge 3$. We shall define a total-colouring $c$ of $T$ with $2n-3$ colours, and show that $c$ is a strongly total rainbow connected colouring. Our approach will use ideas in the proofs of Theorems \ref{Dorbecetaltourthm1} and \ref{LLLStourthm1} (see \cite{SS2016}, Theorem 16; and \cite{LLLS2017}, Theorem 20). Define an arc-colouring $c':A(T)\to\{1,2,\dots,n-1\}$ and a vertex-colouring $c'':V(T)\to\{\alpha_1,\alpha_2,\dots,\alpha_{n-2}\}$ as follows. Let $a,b\in V(T)$ be two vertices with $d(a,b)=d$, and let $aa_1a_2\cdots a_{d-1}b$ be an $a-b$ geodesic. Note that we have $ba\in A(T)$. Let $f:V(T)\to\{1,2,\dots,n\}$ be a bijective mapping such that $f(b)=n$. We let
\[
\begin{array}{rcll}
c'(uv) &\!\!\!=\!\!\!& f(u), & \textup{for every arc $uv\in A(T)$ with $u\neq b$,}\\
c'(ba) &\!\!\!=\!\!\!& f(a_1), &\\
c'(bu) &\!\!\!=\!\!\!& f(a), & \textup{for every $u\in V(T-\{a,b\})$ with $bu\in A(T)$.}
\end{array}
\]
Also, note that $a,a_1,a_{d-1},b$ are distinct vertices. For $c''$, we let $c''(a)=c''(a_{d-1})=\alpha_1$, $c''(a_1)=c''(b)=\alpha_2$, and the remaining vertices are given the distinct colours $\alpha_3,\dots,\alpha_{n-2}$. 

Now let $c$ be the total-colouring of $T$ obtained by combining $c'$ and $c''$. That is,
\[
\begin{array}{rcll}
c(uv) &\!\!\!=\!\!\!& c'(uv), & \textup{for every $uv\in A(T)$,}\\
c(w) &\!\!\!=\!\!\!& c''(w), & \textup{for every $w\in V(T)$.}
\end{array}
\]
Let $x,y\in V(T)$. From the proofs of Theorems \ref{Dorbecetaltourthm1} and \ref{LLLStourthm1}, the following facts can be seen.
\begin{enumerate}
\item[(i)] In the arc-colouring $c'$, every $x-y$ geodesic is rainbow.
\item[(ii)] In the vertex-colouring $c''$, every $x-y$ geodesic is vertex-rainbow, if $d\ge 4$.
\item[(iii)] In the vertex-colouring $c''$, there exists a vertex-rainbow $x-y$ geodesic, if $d=3$.
\end{enumerate}
By (ii) and (iii), we can choose a vertex-rainbow $x-y$ geodesic $P$, with respect to $c''$. Then by (i), $P$ is also a rainbow $x-y$ geodesic, with respect to $c'$. Therefore, $P$ is a total-rainbow $x-y$ geodesic, with respect to $c$. It follows that $c$ is a strongly total rainbow connected colouring for $T$. This completes the proof.
\end{proof}

Next, Theorem \ref{tournamentthm2} below is an analogue to Theorems \ref{Dorbecetaltourthm2}, \ref{ASMBtourthm}, and \ref{LLLStourthm2}.

\begin{thm}\label{tournamentthm2}
For $n\ge 5$ and $3\leq k \leq 2n-3$ with $k$ odd, there exists a tournament $T_{n,k}$ on $n$ vertices  such that 
\begin{equation}\label{tournamenteq1}
\overset{\rightarrow}{\smash{t}rc}(T_{n,k})=\overset{\rightarrow}{s\smash{t}rc}(T_{n,k})=k.
\end{equation}
\end{thm}

\begin{proof}
We first consider the case $k=3$. By Theorem \ref{ASMBtourthm}, we know that for every $n\ge 6$, there exists a tournament $T_{n,3}$ with $\overset{\rightarrow}{rc}(T_{n,3})=2$. Thus by Theorem \ref{pro2}(c), we have $\overset{\rightarrow}{\smash{t}rc}(T_{n,3})=\overset{\rightarrow}{s\smash{t}rc}(T_{n,3})=3$. For $n=5$, we let $T_{5,3}$ be the tournament which is the union of two copies of ${\rC}_5$: $v_0v_1v_2v_3v_4v_0$ and $v_0v_3v_1v_4v_2v_0$. Then it is easy to check that $\overset{\rightarrow}{\smash{t}rc}(T_{5,3})=\overset{\rightarrow}{s\smash{t}rc}(T_{5,3})=3$.

Now, let $k\ge 5$ be odd. We first consider the case $N=N_k=\frac{k+3}{2}\ge 4$, and construct a tournament $T_{N,k}$ on $N$ vertices such that (\ref{tournamenteq1}) holds. Let $V(T_{N,k})=\{v_0,\dots,v_{N-1}\}$, and $A(T_{N,k})=\{v_{i-1}v_i:1\le i\le N-1\}\cup\{v_iv_j:0\le j<i\le N-1$ and $i-j\ge 2\}$. Note that we have $T_{4,5}=T_4$. Since $v_0\cdots v_{N-1}$ is the only $v_0-v_{N-1}$ path in $T_{N,k}$, we have diam$(T_{N,k})\ge N-1$, and thus $\overset{\rightarrow}{\smash{t}rc}(T_{N,k})\ge 2N-3=k$ by (\ref{eq}). Now, consider the total-colouring $c$ of $T_{N,k}$ where $c(v_{i-1}v_i)=i$ for $1\le i\le N-1$; $c(v_iv_j)=N-1$ if $0\le j<i\le N-1$ and $i-j\ge 2$; $c(v_i)=\alpha_i$ for $1\le i\le N-2$; and $c(v_0)=c(v_{N-1})=\alpha_1$. Then $c$ uses $2N-3=k$ colours. Let $v_i,v_j\in V(T_{N,k})$. If $j>i$, then $v_iv_{i+1}\cdots v_j$ is the unique $v_i-v_j$ path, and if $j\le i-2$, then $v_iv_j\in A(T_{N,k})$. If $j=i-1$, then $v_iv_j\not\in A(T_{N,k})$, and a $v_i-v_j$ geodesic is $v_iv_{i-2}v_{i-1}$ if $2\le i\le N-1$, and $v_1v_2v_0$ if $i=1$. In each case, we always have a total-rainbow $v_i-v_j$ geodesic, so that $\overset{\rightarrow}{s\smash{t}rc}(T_{N,k})\le k$. Therefore, (\ref{tournamenteq1}) holds for $T_{N,k}$, by (\ref{eq}).

Finally, let $n>N$. We construct a tournament $T_{n,k}$ on $n$ vertices such that (\ref{tournamenteq1}) holds. Let $T$ be a tournament on $|V(T)|=n-N+1\ge 2$ vertices, where $T$ is a single arc if $|V(T)|=2$; $T=\rC_3$ if $|V(T)|=3$; $T=T_4=T_{4,5}$ if $|V(T)|=4$; and $T=T_{n-N+1,3}$ if $|V(T)|\ge 5$. All of these tournaments exist from our previous arguments. Let $T_{n,k}=(T_{N,k})_{v_0\to T}$, which is obtained from $T_{N,k}$ by expanding $v_0$ into $T$. Then, note that we have diam$(T_{n,k})\ge N-1$, which again gives $\overset{\rightarrow}{\smash{t}rc}(T_{n,k})\ge k$. Now, we extend the colouring $c$ on $T_{N,k}$ to a colouring $c'$ on $T_{n,k}$ by letting $c'(uv)=c(uv)$ if $uv\not\in A(T)$; $c'(uv)=c(v_0v)$ if $u\in V(T)$ and $v\not\in V(T)$; $c'(uv)=c(uv_0)$ if $u\not\in V(T)$ and $v\in V(T)$; $c'(w)=c(w)$ if $w\not\in V(T)$; and
\begin{itemize}
\item if $|V(T)|=2$, let $c'(uv)=1$ and $c'(u)=c'(v)=\alpha_1$, where $uv$ is the only arc of $T$;
\item if $|V(T)|\not\in\{2,4\}$ (resp.~$|V(T)|=4$), let $c'$ be a strongly total rainbow connected colouring when restricted to $T$, using three (resp.~five) colours that are already used.
\end{itemize}
From previous arguments and remarks, the strongly total rainbow connected colourings on $T$ exist for $|V(T)|\ge 3$. We claim that $c'$ is a strongly total rainbow connected colouring for $T_{n,k}$. Let $x,y\in V(T_{n,k})$. A similar argument as for $T_{N,k}$ shows that, there exists a total-rainbow $x-y$ geodesic in $T_{n,k}$, if we do not have $x,y\in V(T)$. Now, let $x,y\in V(T)$. Then, we have exactly one of the following three situations.
\begin{itemize}
\item $xy\in A(T)\subset A(T_{n,k})$.
\item $xy\not\in A(T)$, and there exists $z\in V(T)$ such that $xzy$ is a path in $T\subset T_{n,k}$. Then there exists a total-rainbow $x-y$ path of length $2$ in $T$, which is also a total-rainbow $x-y$ geodesic in $T_{n,k}$.
\item $xy\not\in A(T)$, and there does not exist $z\in V(T)$ such that $xzy$ is a path in $T\subset T_{n,k}$. Then, $xv_1v_2y$ is a total-rainbow $x-y$ geodesic in $T_{n,k}$, since $N\ge 4$.
\end{itemize}
In each case, we have a total-rainbow $x-y$ geodesic in $T_{n,k}$, and the claim holds. Therefore, $\overset{\rightarrow}{s\smash{t}rc}(T_{n,k})\le k$. Again (\ref{tournamenteq1}) holds, by (\ref{eq}).
\end{proof}

In Theorem \ref{tournamentthm2}, we have not been able to consider the case when $k$ is even. Thus we pose the following problem.

\begin{prob}
Do there exist $n,k$ with $4\le k\le 2n-4$ and $k$ even such that, there exists a tournament $T_{n,k}$ on $n$ vertices with $\overset{\rightarrow}{\smash{t}rc}(T_{n,k})=k$? Similarly, what happens for $\overset{\rightarrow}{s\smash{t}rc}(T_{n,k})$?
%
%For all sufficiently large $n$ and $4\le k\le 2n-4$ with $k$ even, find a tournament $T_{n,k}$ on $n$ vertices such that $\overset{\rightarrow}{\smash{t}rc}(T_{n,k})=k$, and similarly for $\overset{\rightarrow}{s\smash{t}rc}(T_{n,k})$.
\end{prob}

In \cite{PIE2014}, Dorbec, Schiermeyer, Sidorowicz, and Sopena also proved the following result, which shows that the rainbow connection number of a tournament is in fact closely related to its diameter.

\begin{thm}\label{Dorbecetaltourthm3}\textup{\cite{PIE2014}}
Let $T$ be a tournament of diameter $d$. We have $d\le\overset{\rightarrow}{rc}(T)\le d+2$.
\end{thm}

Lei, Li, Liu, and Shi \cite{LLLS2017} then proved an analogous result for the rainbow vertex-connection number, as follows.

\begin{thm}\label{LLLStourthm3}\textup{\cite{LLLS2017}}
Let $T$ be a tournament of diameter $d$. We have $d-1\le\overset{\rightarrow}{rvc}(T)\le d+3$.
\end{thm}

Here, we shall prove the following analogue for the total rainbow connection number.

\begin{thm}\label{tournamentthm3}
Let $T$ be a tournament of diameter $d$. 
\begin{enumerate}
\item[(a)] If $d=2$, then $3\le\overset{\rightarrow}{\smash{t}rc}(T)\le 5$.
\item[(b)] If $d\ge 3$, then $2d-1\le\overset{\rightarrow}{\smash{t}rc}(T)\le 2d+7$.
\end{enumerate}
\end{thm}

\begin{proof}
Clearly for both parts, the lower bound follows from (\ref{eq}). We prove the upper bounds.\\[1ex]
\indent(a) Let $d=2$. Let $a\in V(T)$, and $V_1= \Gamma^+(a)$, $V_2=\Gamma^-(a)$. Note that $V_1,V_2\neq\emptyset$; that $V_2$ are the vertices at distance $2$ from $a$; and that $ua\in A(T)$ for all $u\in V_2$. Also, note that for every vertex $u\in V_1$, there must be an out-neighbour of $u$ in $V_2$, since there must exist a $u-a$ path in $T$ with length $2$.

We define a total-colouring $c$ of $T$ as follows. Let $c(uv)=1$ if either $u=a$, or $uv$ is an arc in $V_1$ or $V_2$; $c(uv)=2$ if one of $u,v$ is in $V_1$ and the other is in $V_2$; $c(ua)=3$ for all $u\in V_2$; $c(a)=4$; and $c(u)=5$ for all $u\in V(T)\setminus\{a\}$. We claim that $c$ is a total rainbow connected colouring for $T$. Let $x,y\in V(T)$, and suppose that $xy\not\in A(T)$. We show that a total-rainbow $x-y$ path $P$ always exists. Clearly this is true if $x=a$. Next, let $x\in V_1$, and let $z\in V_2$ be an out-neighbour of $x$. If $y=a$ then we take $P=xza$, and if $y\in V_1$, then we take $P=xzay$. If $y\in V_2$, then since $d=2$, there must exist a suitable path $P=xwy$ for some $w\in V_1\cup V_2$. Finally, let $x\in V_2$. If $y\in V_1$ then we take $P=xay$, and if $y\in V_2$, then we can take $P=xawy$ for some $w\in V_1$. Thus the claim holds, and we have $\overset{\rightarrow}{\smash{t}rc}(T)\le 5$.\\[1ex]
\indent(b) Let $d\ge 3$. We prove the upper bound by constructing a total-colouring $c$ of $T$, using colours $\underline{1},\underline{2},\dots,\underline{d\,\smash{+}\,4}$ for the arcs, and $1,2,\dots,d+3$ for the vertices. As in the proofs of Theorems \ref{Dorbecetaltourthm3} and \ref{LLLStourthm3} (see \cite{PIE2014}, Theorem 24; and \cite{LLLS2017}, Theorem 23), we consider the following decomposition of $T$. Let $a\in V(T)$ be a vertex with eccentricity $d$, i.e., there exists a vertex of $T$ at distance $d$ from $a$. For $1\le i\le d$, let $V_i$ denote the set of vertices at distance $i$ from $a$. Then, every $V_i$ is non-empty. Note that $ua\in A(T)$ for every $u\in V_i$ with $2\le i\le d$, and $uv\in A(T)$ whenever $u\in V_i$, $v\in V_j$, and $i-j\ge 2$. Let $p\in V_1$ have maximum in-degree in $T[V_1]$, and $q\in V_d$ have maximum out-degree in $T[V_d]$. Now, we define the total-colouring $c$ as follows.
\begin{itemize}
\item $c(av)=\underline{1}$ if $v\in V_1$, and $c(uv)=\underline{i}$ if $u\in V_{i-1}$ and $v\in V_i$, for $2\le i\le d$.
\item $c(uv)=\underline{d\,\smash{+}\,1}$ if $u\in V_i$ and $v\in V_j$, where $1\le j<i\le d$.
\item $c(vp)=\underline{1}$ if $v\in V_1$, and $c(uv)=\underline{d\,\smash{+}\,2}$ for every other arc $uv$ in $V_1,V_2,\dots,V_{d-1}$.
\item $c(qv)=\underline{d\,\smash{-}\,1}$ if $v\in V_d$, and $c(uv)=\underline{d\,\smash{+}\,1}$ for every other arc $uv$ in $V_d$.
\item $c(v)=i$ if $v\in V_i$, for $2\le i\le d-1$.
\item $c(p)=d$, and $c(v)=1$ for $v\in V_1\setminus\{p\}$.
\item $c(a)=c(q)=d+1$, and $c(v)=d-2$ for $v\in V_d\setminus\{q\}$.
\end{itemize}
Next, we update $c$ by considering a $p-q$ geodesic $P$. Note that we have $|A(P)|=d-1$ or $|A(P)|=d$. Thus, we consider these two cases.
\begin{enumerate}
\item[(i)] If $|A(P)|=d-1$, then $P$ must contain one arc from $V_i$ to $V_{i+1}$, for every $1\le i\le d-1$. Let $r$ be the vertex of $P$ in $V_{d-2}$. We recolour $r$ by letting $c(r)=d+2$, and the out-arc of $r$ in $P$ with colour $\underline{d\,\smash{+}\,3}$.
\item[(ii)] If $|A(P)|=d$, then $P$ must contain one arc from $V_i$ to $V_{i+1}$, for every $1\le i\le d-1$, and one arc from $V_k$ to $V_k$, for some $1\le k\le d$. If $k\neq d-2$, then let $r$ be the vertex of $P$ in $V_{d-2}$, and $st$ be the arc of $P$ in $V_k$. If $k=d-2$, then let $sr$ be the arc of $P$ in $V_{d-2}$. In either case, we recolour $r$ and $s$ by letting $c(r)=d+2$ and $c(s)=d+3$, the out-arc of $r$ in $P$ with colour $\underline{d\,\smash{+}\,3}$, and the arc in $V_k$ (either $st$ or $sr$) with colour $\underline{d\,\smash{+}\,4}$.
\end{enumerate}
The situation in (ii) is shown in Figure 4.\\[1ex]
\[ \unit = 0.65cm
%Vi
\color{lightgray}{
\varline{500}{0.6}
\dl{-7.3}{0}{-6.7}{0}\dl{-7.3}{6}{-6.7}{6}\dl{-8}{0.7}{-8}{5.3}\dl{-6}{0.7}{-6}{5.3}\bez{-8}{0.7}{-8}{0}{-7.3}{0}\bez{-6}{0.7}{-6}{0}{-6.7}{0}\bez{-8}{5.3}{-8}{6}{-7.3}{6}\bez{-6}{5.3}{-6}{6}{-6.7}{6}
\dl{-4.7}{0}{-4.1}{0}\dl{-4.7}{6}{-4.1}{6}\dl{-5.4}{0.7}{-5.4}{5.3}\dl{-3.4}{0.7}{-3.4}{5.3}\bez{-5.4}{0.7}{-5.4}{0}{-4.7}{0}\bez{-3.4}{0.7}{-3.4}{0}{-4.1}{0}\bez{-5.4}{5.3}{-5.4}{6}{-4.7}{6}\bez{-3.4}{5.3}{-3.4}{6}{-4.1}{6}
\dl{-0.3}{0}{0.3}{0}\dl{-0.3}{6}{0.3}{6}\dl{-1}{0.7}{-1}{5.3}\dl{1}{0.7}{1}{5.3}\bez{-1}{0.7}{-1}{0}{-0.3}{0}\bez{1}{0.7}{1}{0}{0.3}{0}\bez{-1}{5.3}{-1}{6}{-0.3}{6}\bez{1}{5.3}{1}{6}{0.3}{6}
\dl{4.7}{0}{4.1}{0}\dl{4.7}{6}{4.1}{6}\dl{5.4}{0.7}{5.4}{5.3}\dl{3.4}{0.7}{3.4}{5.3}\bez{5.4}{0.7}{5.4}{0}{4.7}{0}\bez{3.4}{0.7}{3.4}{0}{4.1}{0}\bez{5.4}{5.3}{5.4}{6}{4.7}{6}\bez{3.4}{5.3}{3.4}{6}{4.1}{6}
\dl{7.3}{0}{6.7}{0}\dl{7.3}{6}{6.7}{6}\dl{8}{0.7}{8}{5.3}\dl{6}{0.7}{6}{5.3}\bez{8}{0.7}{8}{0}{7.3}{0}\bez{6}{0.7}{6}{0}{6.7}{0}\bez{8}{5.3}{8}{6}{7.3}{6}\bez{6}{5.3}{6}{6}{6.7}{6}
\dl{9.9}{0}{9.3}{0}\dl{9.9}{6}{9.3}{6}\dl{10.6}{0.7}{10.6}{5.3}\dl{8.6}{0.7}{8.6}{5.3}\bez{10.6}{0.7}{10.6}{0}{9.9}{0}\bez{8.6}{0.7}{8.6}{0}{9.3}{0}\bez{10.6}{5.3}{10.6}{6}{9.9}{6}\bez{8.6}{5.3}{8.6}{6}{9.3}{6}
}
\color{black}{
\point{-7.3}{6.2}{\small $V_1$}\point{-4.7}{6.2}{\small $V_2$}\point{-0.3}{6.2}{\small $V_k$}\point{3.82}{6.2}{\small $V_{d-2}$}\point{6.42}{6.2}{\small $V_{d-1}$}\point{9.3}{6.2}{\small $V_d$}
\point{-2.53}{3.96}{\small $\dots$}\point{1.86}{3.96}{\small $\dots$}
%vertices
\pt{-9.6}{3}
\pt{-7}{1}\pt{-7}{4}\pt{-7}{5}
\pt{-4.4}{1}\pt{-4.4}{4}\pt{-4.4}{5}
\pt{0}{1}\pt{0}{2}\pt{0}{4}\pt{0}{5}
\pt{4.4}{2}\pt{4.4}{4}\pt{4.4}{5}
\pt{7}{2}\pt{7}{4}\pt{7}{5}
\pt{9.6}{2}\pt{9.6}{4}\pt{9.6}{5}
%vertex labels
\point{-10.18}{2.88}{\footnotesize $a$}\ellipse{-10.05}{3}{0.27}{0.27}
\point{-7.13}{0.48}{\footnotesize $p$}\ellipse{-7}{0.55}{0.27}{0.27}
\point{-0.32}{0.44}{\footnotesize $s$}\ellipse{-0.2}{0.55}{0.27}{0.27}
\point{0.11}{2.29}{\footnotesize $t$}\ellipse{0.2}{2.45}{0.27}{0.27}
\point{4.28}{1.43}{\footnotesize $r$}\ellipse{4.4}{1.55}{0.27}{0.27}
\point{9.97}{1.93}{\footnotesize $q$}\ellipse{10.1}{2}{0.27}{0.27}
%P
\point{1.93}{1.36}{\footnotesize $P$}
%arccolours
\point{-8.46}{7.1}{\tiny\underline{$d\smash{+}1$}}\point{-8.6}{4.17}{\tiny\underline{1}}\point{-8.6}{1.65}{\tiny\underline{1}}\point{-8.35}{3.1}{\tiny\underline{1}}\point{-7.8}{2.68}{\tiny\underline{$1$}}\point{-6.86}{2.3}{\tiny\underline{$d\smash{+}2$}}\point{-6.86}{4.6}{\tiny\underline{$d\smash{+}2$}}\point{-5.5}{4.3}{\tiny\underline{$d\smash{+}1$}}\point{-6.08}{3.6}{\tiny\underline{$d\smash{+}1$}}\point{-5.82}{5.23}{\tiny\underline{2}}\point{-5.82}{2.8}{\tiny\underline{2}}\point{-5.82}{0.6}{\tiny\underline{2}}\point{-4.26}{4.4}{\tiny\underline{$d\smash{+}2$}}\point{-3.15}{2.31}{\tiny\underline{$d\smash{+}1$}}\point{0.13}{1.4}{\tiny\underline{$d\smash{+}4$}}\point{0.13}{4.4}{\tiny\underline{$d\smash{+}2$}}\point{-0.39}{-1.42}{\tiny\underline{$d\smash{+}1$}}\point{0.92}{7.25}{\tiny\underline{$d\smash{+}1$}}\point{6.62}{7.25}{\tiny\underline{$d\smash{+}1$}}\point{5.3}{5.23}{\tiny\underline{$d\smash{-}1$}}\point{5.3}{3.6}{\tiny\underline{$d\smash{+}1$}}\point{5.3}{2.23}{\tiny\underline{$d\smash{+}3$}}\point{4.6}{4.35}{\tiny\underline{$d\smash{-}1$}}\point{6.08}{4.6}{\tiny\underline{$d\smash{+}2$}}\point{8.7}{4.6}{\tiny\underline{$d\smash{+}1$}}\point{8.16}{4.12}{\tiny\underline{$d$}}\point{8.16}{2.23}{\tiny\underline{$d$}}\point{8.16}{5.23}{\tiny\underline{$d\smash{+}1$}}\point{8.7}{2.9}{\tiny\underline{$d\smash{+}1$}}\point{10.27}{3.4}{\tiny\underline{$d\smash{-}1$}}
%vertexcolours
\point{-9.35}{2.85}{\tiny $d\smash{+}1$}\point{-7.12}{5.17}{\tiny $1$}\point{-6.95}{3.45}{\tiny $1$}\point{-6.95}{1.2}{\tiny $d$}\point{-4.45}{1.16}{\tiny $2$}\point{-4.52}{3.62}{\tiny $2$}\point{-4.52}{5.17}{\tiny $2$}\point{0.13}{0.9}{\tiny $d\smash{+}3$}\point{-0.4}{1.9}{\tiny $k$}\point{-0.12}{3.62}{\tiny $k$}\point{-0.12}{5.17}{\tiny $k$}\point{4.01}{3.62}{\tiny $d\smash{-}2$}\point{3.66}{5.17}{\tiny $d\smash{-}2$}\point{4.01}{2.2}{\tiny $d\smash{+}2$}\point{6.61}{2.2}{\tiny $d\smash{-}1$}\point{6.61}{3.62}{\tiny $d\smash{-}1$}\point{7.02}{5.17}{\tiny $d\smash{-}1$}\point{9.62}{5.17}{\tiny $d\smash{-}2$}\point{8.73}{3.73}{\tiny $d\smash{-}2$}\point{8.73}{2.2}{\tiny $d\smash{+}1$}
%arcs
\dl{-5.6}{-1}{6.6}{-1}\bez{-9.6}{3}{-9.6}{-1}{-5.6}{-1}\bez{9.6}{2}{9.6}{-1}{6.6}{-1}
\bez{-9.6}{3}{-9.2}{9.7}{-4.4}{5}
\dl{-0.7}{7}{3.3}{7}\bez{7}{5}{5.3}{7}{3.3}{7}\bez{-4.4}{5}{-2.7}{7}{-0.7}{7}
\dl{-9.6}{3}{-7}{1}\dl{-9.6}{3}{-7}{4}\dl{-9.6}{3}{-7}{5}\dl{-7}{1}{-7}{5}\dl{-7}{5}{-4.4}{5}\dl{-7}{4}{-4.4}{4}\dl{-7}{4}{-4.4}{5}\dl{-4.4}{4}{-4.4}{5}\dl{-7}{4}{-4.4}{1}\dl{-7}{1}{0}{4}\dl{0}{4}{0}{5}\dl{7}{4}{7}{5}\dl{7}{5}{9.6}{5}\dl{7}{5}{9.6}{4}\dl{7}{4}{4.4}{5}\dl{4.4}{5}{7}{5}\dl{4.4}{4}{7}{4}\bez{4.4}{5}{7}{9}{9.6}{5}
\bez{-7}{1}{-8}{3}{-7}{5}\dl{9.6}{2}{9.6}{5}\bez{9.6}{2}{10.6}{3.5}{9.6}{5}
\varline{850}{3.8}
\dl{-7}{1}{0}{1}\dl{0}{1}{0}{2}\dl{0}{2}{9.6}{2}
%arrows
\thnline
\dl{-1.63}{1}{-1.89}{1.1}\dl{-1.67}{1}{-1.89}{1.1}\dl{-1.71}{1}{-1.89}{1.1}\dl{-1.75}{1}{-1.89}{1.1}\dl{-1.79}{1}{-1.89}{1.1}\dl{-1.83}{1}{-1.89}{1.1}
\dl{-1.63}{1}{-1.89}{0.9}\dl{-1.67}{1}{-1.89}{0.9}\dl{-1.71}{1}{-1.89}{0.9}\dl{-1.75}{1}{-1.89}{0.9}\dl{-1.79}{1}{-1.89}{0.9}\dl{-1.83}{1}{-1.89}{0.9}
\dl{-1.63}{1}{-1.83}{1}
\dl{-2.51}{1}{-2.77}{1.1}\dl{-2.55}{1}{-2.77}{1.1}\dl{-2.59}{1}{-2.77}{1.1}\dl{-2.63}{1}{-2.77}{1.1}\dl{-2.67}{1}{-2.77}{1.1}\dl{-2.71}{1}{-2.77}{1.1}
\dl{-2.51}{1}{-2.77}{0.9}\dl{-2.55}{1}{-2.77}{0.9}\dl{-2.59}{1}{-2.77}{0.9}\dl{-2.63}{1}{-2.77}{0.9}\dl{-2.67}{1}{-2.77}{0.9}\dl{-2.71}{1}{-2.77}{0.9}
\dl{-2.51}{1}{-2.71}{1}
\dl{-5.57}{1}{-5.83}{1.1}\dl{-5.61}{1}{-5.83}{1.1}\dl{-5.65}{1}{-5.83}{1.1}\dl{-5.69}{1}{-5.83}{1.1}\dl{-5.73}{1}{-5.83}{1.1}\dl{-5.77}{1}{-5.83}{1.1}
\dl{-5.57}{1}{-5.83}{0.9}\dl{-5.61}{1}{-5.83}{0.9}\dl{-5.65}{1}{-5.83}{0.9}\dl{-5.69}{1}{-5.83}{0.9}\dl{-5.73}{1}{-5.83}{0.9}\dl{-5.77}{1}{-5.83}{0.9}
\dl{-5.57}{1}{-5.77}{1}
\dl{1.89}{2}{1.63}{2.1}\dl{1.85}{2}{1.63}{2.1}\dl{1.81}{2}{1.63}{2.1}\dl{1.77}{2}{1.63}{2.1}\dl{1.73}{2}{1.63}{2.1}\dl{1.69}{2}{1.63}{2.1}
\dl{1.89}{2}{1.63}{1.9}\dl{1.85}{2}{1.63}{1.9}\dl{1.81}{2}{1.63}{1.9}\dl{1.77}{2}{1.63}{1.9}\dl{1.73}{2}{1.63}{1.9}\dl{1.69}{2}{1.63}{1.9}
\dl{1.89}{2}{1.69}{2}
\dl{2.77}{2}{2.51}{2.1}\dl{2.73}{2}{2.51}{2.1}\dl{2.69}{2}{2.51}{2.1}\dl{2.65}{2}{2.51}{2.1}\dl{2.61}{2}{2.51}{2.1}\dl{2.57}{2}{2.51}{2.1}
\dl{2.77}{2}{2.51}{1.9}\dl{2.73}{2}{2.51}{1.9}\dl{2.69}{2}{2.51}{1.9}\dl{2.65}{2}{2.51}{1.9}\dl{2.61}{2}{2.51}{1.9}\dl{2.57}{2}{2.51}{1.9}
\dl{2.77}{2}{2.57}{2}
\dl{5.83}{2}{5.57}{2.1}\dl{5.79}{2}{5.57}{2.1}\dl{5.75}{2}{5.57}{2.1}\dl{5.71}{2}{5.57}{2.1}\dl{5.67}{2}{5.57}{2.1}\dl{5.63}{2}{5.57}{2.1}
\dl{5.83}{2}{5.57}{1.9}\dl{5.79}{2}{5.57}{1.9}\dl{5.75}{2}{5.57}{1.9}\dl{5.71}{2}{5.57}{1.9}\dl{5.67}{2}{5.57}{1.9}\dl{5.63}{2}{5.57}{1.9}
\dl{5.83}{2}{5.63}{2}
\dl{8.43}{2}{8.17}{2.1}\dl{8.39}{2}{8.17}{2.1}\dl{8.35}{2}{8.17}{2.1}\dl{8.31}{2}{8.17}{2.1}\dl{8.27}{2}{8.17}{2.1}\dl{8.23}{2}{8.17}{2.1}
\dl{8.43}{2}{8.17}{1.9}\dl{8.39}{2}{8.17}{1.9}\dl{8.35}{2}{8.17}{1.9}\dl{8.31}{2}{8.17}{1.9}\dl{8.27}{2}{8.17}{1.9}\dl{8.23}{2}{8.17}{1.9}
\dl{8.43}{2}{8.23}{2}
\dl{-0.13}{-1}{0.13}{-0.9}\dl{-0.09}{-1}{0.13}{-0.9}\dl{-0.05}{-1}{0.13}{-0.9}\dl{-0.01}{-1}{0.13}{-0.9}\dl{0.03}{-1}{0.13}{-0.9}\dl{0.07}{-1}{0.13}{-0.9}
\dl{-0.13}{-1}{0.13}{-1.1}\dl{-0.09}{-1}{0.13}{-1.1}\dl{-0.05}{-1}{0.13}{-1.1}\dl{-0.01}{-1}{0.13}{-1.1}\dl{0.03}{-1}{0.13}{-1.1}\dl{0.07}{-1}{0.13}{-1.1}
\dl{1.17}{7}{1.43}{6.9}\dl{1.21}{7}{1.43}{6.9}\dl{1.25}{7}{1.43}{6.9}\dl{1.29}{7}{1.43}{6.9}\dl{1.33}{7}{1.43}{6.9}\dl{1.37}{7}{1.43}{6.9}
\dl{1.17}{7}{1.43}{7.1}\dl{1.21}{7}{1.43}{7.1}\dl{1.25}{7}{1.43}{7.1}\dl{1.29}{7}{1.43}{7.1}\dl{1.33}{7}{1.43}{7.1}\dl{1.37}{7}{1.43}{7.1}
\dl{6.87}{7}{7.13}{6.9}\dl{6.91}{7}{7.13}{6.9}\dl{6.95}{7}{7.13}{6.9}\dl{6.99}{7}{7.13}{6.9}\dl{7.03}{7}{7.13}{6.9}\dl{7.07}{7}{7.13}{6.9}
\dl{6.87}{7}{7.13}{7.1}\dl{6.91}{7}{7.13}{7.1}\dl{6.95}{7}{7.13}{7.1}\dl{6.99}{7}{7.13}{7.1}\dl{7.03}{7}{7.13}{7.1}\dl{7.07}{7}{7.13}{7.1}
\dl{0}{1.63}{0.1}{1.37}\dl{0}{1.56}{0.1}{1.37}\dl{0}{1.51}{0.1}{1.37}\dl{0}{1.47}{0.1}{1.37}\dl{0}{1.45}{0.1}{1.37}\dl{0}{1.43}{0.1}{1.37}
\dl{0}{1.63}{-0.1}{1.37}\dl{0}{1.56}{-0.1}{1.37}\dl{0}{1.51}{-0.1}{1.37}\dl{0}{1.47}{-0.1}{1.37}\dl{0}{1.45}{-0.1}{1.37}\dl{0}{1.43}{-0.1}{1.37}
\dl{0}{1.43}{0}{1.63}
\dl{7}{4.37}{7.1}{4.63}\dl{7}{4.44}{7.1}{4.63}\dl{7}{4.49}{7.1}{4.63}\dl{7}{4.53}{7.1}{4.63}\dl{7}{4.55}{7.1}{4.63}\dl{7}{4.57}{7.1}{4.63}
\dl{7}{4.37}{6.9}{4.63}\dl{7}{4.44}{6.9}{4.63}\dl{7}{4.49}{6.9}{4.63}\dl{7}{4.53}{6.9}{4.63}\dl{7}{4.55}{6.9}{4.63}\dl{7}{4.57}{6.9}{4.63}
\dl{9.6}{4.37}{9.7}{4.63}\dl{9.6}{4.44}{9.7}{4.63}\dl{9.6}{4.49}{9.7}{4.63}\dl{9.6}{4.53}{9.7}{4.63}\dl{9.6}{4.55}{9.7}{4.63}\dl{9.6}{4.57}{9.7}{4.63}
\dl{9.6}{4.37}{9.5}{4.63}\dl{9.6}{4.44}{9.5}{4.63}\dl{9.6}{4.49}{9.5}{4.63}\dl{9.6}{4.53}{9.5}{4.63}\dl{9.6}{4.55}{9.5}{4.63}\dl{9.6}{4.57}{9.5}{4.63}
\dl{9.6}{2.87}{9.7}{3.13}\dl{9.6}{2.94}{9.7}{3.13}\dl{9.6}{2.99}{9.7}{3.13}\dl{9.6}{3.03}{9.7}{3.13}\dl{9.6}{3.05}{9.7}{3.13}\dl{9.6}{3.07}{9.7}{3.13}
\dl{9.6}{2.87}{9.5}{3.13}\dl{9.6}{2.94}{9.5}{3.13}\dl{9.6}{2.99}{9.5}{3.13}\dl{9.6}{3.03}{9.5}{3.13}\dl{9.6}{3.05}{9.5}{3.13}\dl{9.6}{3.07}{9.5}{3.13}
\dl{10.1}{3.63}{10.2}{3.37}\dl{10.1}{3.56}{10.2}{3.37}\dl{10.1}{3.51}{10.2}{3.37}\dl{10.1}{3.47}{10.2}{3.37}\dl{10.1}{3.45}{10.2}{3.37}\dl{10.1}{3.43}{10.2}{3.37}
\dl{10.1}{3.63}{10}{3.37}\dl{10.1}{3.56}{10}{3.37}\dl{10.1}{3.51}{10}{3.37}\dl{10.1}{3.47}{10}{3.37}\dl{10.1}{3.45}{10}{3.37}\dl{10.1}{3.43}{10}{3.37}
\dl{0}{4.37}{0.1}{4.63}\dl{0}{4.44}{0.1}{4.63}\dl{0}{4.49}{0.1}{4.63}\dl{0}{4.53}{0.1}{4.63}\dl{0}{4.55}{0.1}{4.63}\dl{0}{4.57}{0.1}{4.63}
\dl{0}{4.37}{-0.1}{4.63}\dl{0}{4.44}{-0.1}{4.63}\dl{0}{4.49}{-0.1}{4.63}\dl{0}{4.53}{-0.1}{4.63}\dl{0}{4.55}{-0.1}{4.63}\dl{0}{4.57}{-0.1}{4.63}
\dl{-7.5}{2.87}{-7.6}{3.13}\dl{-7.5}{2.94}{-7.6}{3.13}\dl{-7.5}{2.99}{-7.6}{3.13}\dl{-7.5}{3.03}{-7.6}{3.13}\dl{-7.5}{3.05}{-7.6}{3.13}\dl{-7.5}{3.07}{-7.6}{3.13}
\dl{-7.5}{2.87}{-7.4}{3.13}\dl{-7.5}{2.94}{-7.4}{3.13}\dl{-7.5}{2.99}{-7.4}{3.13}\dl{-7.5}{3.03}{-7.4}{3.13}\dl{-7.5}{3.05}{-7.4}{3.13}\dl{-7.5}{3.07}{-7.4}{3.13}
\dl{-4.4}{4.63}{-4.3}{4.37}\dl{-4.4}{4.56}{-4.3}{4.37}\dl{-4.4}{4.51}{-4.3}{4.37}\dl{-4.4}{4.47}{-4.3}{4.37}\dl{-4.4}{4.45}{-4.3}{4.37}\dl{-4.4}{4.43}{-4.3}{4.37}
\dl{-4.4}{4.63}{-4.5}{4.37}\dl{-4.4}{4.56}{-4.5}{4.37}\dl{-4.4}{4.51}{-4.5}{4.37}\dl{-4.4}{4.47}{-4.5}{4.37}\dl{-4.4}{4.45}{-4.5}{4.37}\dl{-4.4}{4.43}{-4.5}{4.37}
\dl{-7}{4.63}{-6.9}{4.37}\dl{-7}{4.56}{-6.9}{4.37}\dl{-7}{4.51}{-6.9}{4.37}\dl{-7}{4.47}{-6.9}{4.37}\dl{-7}{4.45}{-6.9}{4.37}\dl{-7}{4.43}{-6.9}{4.37}
\dl{-7}{4.63}{-7.1}{4.37}\dl{-7}{4.56}{-7.1}{4.37}\dl{-7}{4.51}{-7.1}{4.37}\dl{-7}{4.47}{-7.1}{4.37}\dl{-7}{4.45}{-7.1}{4.37}\dl{-7}{4.43}{-7.1}{4.37}
\dl{-7}{2.63}{-6.9}{2.37}\dl{-7}{2.56}{-6.9}{2.37}\dl{-7}{2.51}{-6.9}{2.37}\dl{-7}{2.47}{-6.9}{2.37}\dl{-7}{2.45}{-6.9}{2.37}\dl{-7}{2.43}{-6.9}{2.37}
\dl{-7}{2.63}{-7.1}{2.37}\dl{-7}{2.56}{-7.1}{2.37}\dl{-7}{2.51}{-7.1}{2.37}\dl{-7}{2.47}{-7.1}{2.37}\dl{-7}{2.45}{-7.1}{2.37}\dl{-7}{2.43}{-7.1}{2.37}
\dl{-5.57}{5}{-5.83}{5.1}\dl{-5.61}{5}{-5.83}{5.1}\dl{-5.65}{5}{-5.83}{5.1}\dl{-5.69}{5}{-5.83}{5.1}\dl{-5.73}{5}{-5.83}{5.1}\dl{-5.77}{5}{-5.83}{5.1}
\dl{-5.57}{5}{-5.83}{4.9}\dl{-5.61}{5}{-5.83}{4.9}\dl{-5.65}{5}{-5.83}{4.9}\dl{-5.69}{5}{-5.83}{4.9}\dl{-5.73}{5}{-5.83}{4.9}\dl{-5.77}{5}{-5.83}{4.9}
\dl{-5.83}{4}{-5.57}{3.9}\dl{-5.79}{4}{-5.57}{3.9}\dl{-5.75}{4}{-5.57}{3.9}\dl{-5.71}{4}{-5.57}{3.9}\dl{-5.67}{4}{-5.57}{3.9}\dl{-5.63}{4}{-5.57}{3.9}
\dl{-5.83}{4}{-5.57}{4.1}\dl{-5.79}{4}{-5.57}{4.1}\dl{-5.75}{4}{-5.57}{4.1}\dl{-5.71}{4}{-5.57}{4.1}\dl{-5.67}{4}{-5.57}{4.1}\dl{-5.63}{4}{-5.57}{4.1}
\dl{5.57}{4}{5.83}{3.9}\dl{5.61}{4}{5.83}{3.9}\dl{5.65}{4}{5.83}{3.9}\dl{5.69}{4}{5.83}{3.9}\dl{5.73}{4}{5.83}{3.9}\dl{5.77}{4}{5.83}{3.9}
\dl{5.57}{4}{5.83}{4.1}\dl{5.61}{4}{5.83}{4.1}\dl{5.65}{4}{5.83}{4.1}\dl{5.69}{4}{5.83}{4.1}\dl{5.73}{4}{5.83}{4.1}\dl{5.77}{4}{5.83}{4.1}
\dl{8.17}{5}{8.43}{4.9}\dl{8.21}{5}{8.43}{4.9}\dl{8.25}{5}{8.43}{4.9}\dl{8.29}{5}{8.43}{4.9}\dl{8.33}{5}{8.43}{4.9}\dl{8.37}{5}{8.43}{4.9}
\dl{8.17}{5}{8.43}{5.1}\dl{8.21}{5}{8.43}{5.1}\dl{8.25}{5}{8.43}{5.1}\dl{8.29}{5}{8.43}{5.1}\dl{8.33}{5}{8.43}{5.1}\dl{8.37}{5}{8.43}{5.1}
\dl{5.83}{5}{5.57}{5.1}\dl{5.79}{5}{5.57}{5.1}\dl{5.75}{5}{5.57}{5.1}\dl{5.71}{5}{5.57}{5.1}\dl{5.67}{5}{5.57}{5.1}\dl{5.63}{5}{5.57}{5.1}
\dl{5.83}{5}{5.57}{4.9}\dl{5.79}{5}{5.57}{4.9}\dl{5.75}{5}{5.57}{4.9}\dl{5.71}{5}{5.57}{4.9}\dl{5.67}{5}{5.57}{4.9}\dl{5.63}{5}{5.57}{4.9}
\dl{-8.2}{4.08}{-8.46}{4}\dl{-8.26}{4.04}{-8.46}{4}\dl{-8.3}{4.02}{-8.46}{4}\dl{-8.32}{4}{-8.46}{4}\dl{-8.34}{3.98}{-8.46}{4}\dl{-8.36}{3.96}{-8.46}{4}
\dl{-8.2}{4.08}{-8.34}{3.86}\dl{-8.26}{4.04}{-8.34}{3.86}\dl{-8.3}{4.02}{-8.34}{3.86}\dl{-8.32}{4}{-8.34}{3.86}\dl{-8.34}{3.98}{-8.34}{3.86}\dl{-8.36}{3.96}{-8.34}{3.86}
\dl{-8.2}{1.92}{-8.46}{2}\dl{-8.26}{1.96}{-8.46}{2}\dl{-8.3}{1.98}{-8.46}{2}\dl{-8.32}{2}{-8.46}{2}\dl{-8.34}{2.02}{-8.46}{2}\dl{-8.36}{2.04}{-8.46}{2}
\dl{-8.2}{1.92}{-8.34}{2.14}\dl{-8.26}{1.96}{-8.34}{2.14}\dl{-8.3}{1.98}{-8.34}{2.14}\dl{-8.32}{2}{-8.34}{2.14}\dl{-8.34}{2.02}{-8.34}{2.14}\dl{-8.36}{2.04}{-8.34}{2.14}
\dl{-8.18}{3.54}{-8.43}{3.53}\dl{-8.24}{3.52}{-8.43}{3.53}\dl{-8.28}{3.5}{-8.43}{3.53}\dl{-8.31}{3.49}{-8.43}{3.53}\dl{-8.34}{3.48}{-8.43}{3.53}\dl{-8.37}{3.47}{-8.43}{3.53}
\dl{-8.18}{3.54}{-8.37}{3.38}\dl{-8.24}{3.52}{-8.37}{3.38}\dl{-8.28}{3.5}{-8.37}{3.38}\dl{-8.31}{3.49}{-8.37}{3.38}\dl{-8.34}{3.48}{-8.37}{3.38}\dl{-8.37}{3.47}{-8.37}{3.38}
\dl{-5.82}{4.46}{-5.57}{4.47}\dl{-5.76}{4.48}{-5.57}{4.47}\dl{-5.72}{4.5}{-5.57}{4.47}\dl{-5.69}{4.51}{-5.57}{4.47}\dl{-5.66}{4.52}{-5.57}{4.47}\dl{-5.63}{4.53}{-5.57}{4.47}
\dl{-5.82}{4.46}{-5.63}{4.62}\dl{-5.76}{4.48}{-5.63}{4.62}\dl{-5.72}{4.5}{-5.63}{4.62}\dl{-5.69}{4.51}{-5.63}{4.62}\dl{-5.66}{4.52}{-5.63}{4.62}\dl{-5.63}{4.53}{-5.63}{4.62}
\dl{5.82}{4.46}{5.57}{4.47}\dl{5.76}{4.48}{5.57}{4.47}\dl{5.72}{4.5}{5.57}{4.47}\dl{5.69}{4.51}{5.57}{4.47}\dl{5.66}{4.52}{5.57}{4.47}\dl{5.63}{4.53}{5.57}{4.47}
\dl{5.82}{4.46}{5.63}{4.62}\dl{5.76}{4.48}{5.63}{4.62}\dl{5.72}{4.5}{5.63}{4.62}\dl{5.69}{4.51}{5.63}{4.62}\dl{5.66}{4.52}{5.63}{4.62}\dl{5.63}{4.53}{5.63}{4.62}
\dl{8.42}{4.46}{8.17}{4.47}\dl{8.36}{4.48}{8.17}{4.47}\dl{8.32}{4.5}{8.17}{4.47}\dl{8.29}{4.51}{8.17}{4.47}\dl{8.26}{4.52}{8.17}{4.47}\dl{8.23}{4.53}{8.17}{4.47}
\dl{8.42}{4.46}{8.23}{4.62}\dl{8.36}{4.48}{8.23}{4.62}\dl{8.32}{4.5}{8.23}{4.62}\dl{8.29}{4.51}{8.23}{4.62}\dl{8.26}{4.52}{8.23}{4.62}\dl{8.23}{4.53}{8.23}{4.62}
\dl{-5.62}{2.41}{-5.72}{2.64}\dl{-5.67}{2.47}{-5.72}{2.64}\dl{-5.7}{2.51}{-5.72}{2.64}\dl{-5.73}{2.54}{-5.72}{2.64}\dl{-5.75}{2.56}{-5.72}{2.64}
\dl{-5.62}{2.41}{-5.83}{2.53}\dl{-5.67}{2.47}{-5.83}{2.53}\dl{-5.7}{2.51}{-5.83}{2.53}\dl{-5.73}{2.54}{-5.83}{2.53}\dl{-5.75}{2.56}{-5.83}{2.53}
\dl{-3.2}{2.63}{-3}{2.81}\dl{-3.14}{2.65}{-3}{2.81}\dl{-3.09}{2.67}{-3}{2.81}\dl{-3.06}{2.69}{-3}{2.81}\dl{-3.04}{2.7}{-3}{2.81}
\dl{-3.02}{2.71}{-3}{2.81}
\dl{-3.2}{2.63}{-2.94}{2.65}\dl{-3.14}{2.65}{-2.94}{2.65}\dl{-3.09}{2.67}{-2.94}{2.65}\dl{-3.06}{2.69}{-2.94}{2.65}\dl{-3.04}{2.7}{-2.94}{2.65}\dl{-3.02}{2.71}{-2.94}{2.65}
\dl{-8.2}{6.81}{-7.95}{6.82}\dl{-8.14}{6.83}{-7.95}{6.82}\dl{-8.1}{6.85}{-7.95}{6.82}\dl{-8.07}{6.86}{-7.95}{6.82}\dl{-8.04}{6.87}{-7.95}{6.82}\dl{-8.01}{6.88}{-7.95}{6.82}
\dl{-8.2}{6.81}{-8.01}{6.97}\dl{-8.14}{6.83}{-8.01}{6.97}\dl{-8.1}{6.85}{-8.01}{6.97}\dl{-8.07}{6.86}{-8.01}{6.97}\dl{-8.04}{6.87}{-8.01}{6.97}\dl{-8.01}{6.88}{-8.01}{6.97}
%caption
\ptlu{0}{-3.5}{\textup{
\small
\begin{tabular}{c}
Figure 4. The decomposition and colouring of $T$. The dotted line\\ 
represents the $p-q$ geodesic $P$. Letters that label vertices are circled. 
\end{tabular}
}
}
}
\]\\[1ex]
\indent Now, we prove that $c$ is a total rainbow connected colouring for $T$, which implies the required upper bound $\overset{\rightarrow}{\smash{t}rc}(T)\le 2d+7$. The proof from here onwards is the same as in \cite{LLLS2017}, but we shall provide it for the sake of completeness. Let $x,y\in V(T)$. We show that there always exists a total-rainbow $x-y$ path. This is true if $x=a$. Let $x\not\in V_1\cup\{a\}$. If $y=a$, then $xy\in A(T)$. Otherwise, $y\in V_j$ for some $1\le j\le d$. Let $Q$ be an $a-y$ geodesic, which contains one vertex in each of $V_1,\dots,V_j$. If $x\in V(Q)$, then $Q$ contains a total-rainbow $x-y$ path. Otherwise, $x\not\in V(Q)$, and $xa\cup Q$ is a desired path.

Hence, suppose that $x\in V_1$. Note that by the choice of $p$, if $x\neq p$, then either $xp\in A(T)$, or there exists $w\in V_1$ such that $xwp$ is a path. Indeed, if $xp\not\in A(T)$ and no such $w$ exists in $V_1$, then in $T[V_1]$, the set of in-neighbours of $x$ contains $p$, and all in-neighbours of $p$. Thus, $x$ contradicts the choice of $p$. Hence in $T[V_1]$, there is an $x-p$ path $P'$ of length at most two, which is total-rainbow. Now, consider $y\not\in V_{d-1}\cup V_d$. If $y\in V(P'\cup P)$, then $P'\cup P$ contains a total-rainbow $x-y$ path. Otherwise, $P'\cup P\cup qy$ is a desired path. Next, if $y\in V_d$, then by a similar argument as for $T[V_1]$, we have a $q-y$ total-rainbow path $P''$ of length at most two in $T[V_d]$, by the choice of $q$. We have $P'\cup P\cup P''$ is a desired path. Finally, let $y\in V_{d-1}$. If $ry\in A(T)$ in case (i) or case (ii), or $sy\in A(T)$ in case (ii), then $P'\cup P\cup ry$ or $P'\cup P\cup sy$ contains a total-rainbow $x-y$ path. Otherwise, we can find an in-neighbour of $y$ in $V_{d-2}$, say $z$. Note that $z$ exists by considering an $a-y$ geodesic, and that $z\neq r$ in case (i), or $z\not\in\{s,r\}$ in case (ii). Then, $P'\cup P\cup qzy$ is a desired path.
\end{proof}

\section{Cactus digraphs}\label{cactussect}
In \cite{AM2017}, Alva-Samos and Montellano-Ballesteros studied the rainbow connection number of cactus digraphs. In
this section, we will consider both the rainbow vertex-connection and the total rainbow connection numbers of cactus digraphs.

Recall that a \emph{cactus} is a strongly connected oriented graph where every arc belongs to exactly one directed cycle. We begin by describing the structure of a cactus. Recall that for a digraph $D$, a \emph{block} is a maximal subdigraph without a cut-vertex. An \emph{end-block} is a block which is incident with at most one cut-vertex, and there must exist at least one end-block in $D$. Moreover, if $D$ itself is not a block, then there must exist at least two end-blocks in $D$. The \emph{block graph} of $D$, denoted $B(D)$, is the graph with $V(B(D)) = \{B_i:B_i$ is block of $D\}$ and $B_iB_j \in E(B(D))$ if $B_i$ and $B_j$ share a vertex in $D$. From the definition of cactus, it is not hard to obtain the following characterisation, as remarked in \cite{AM2017}.

\begin{lem}\label{cactuslem1}\textup{\cite{AM2017}}
Let $Q$ be a digraph with $n$ vertices and $m$ arcs. Then the following statements are equivalent.
\begin{enumerate}
\item[(i)] $Q$ is a cactus.
\item[(ii)] $Q$ is a strongly connected digraph in which every block is a directed cycle.
\item[(iii)] Let $q$ be the number of blocks in $Q$. Then $Q$ has a decomposition into directed cycles $H_1,\ldots,H_q$ such that, for each $k=2,\ldots,q$  we have 
\[
\bigg|V(H_k)\cap\bigg(\bigcup\limits_{i=1}^{k-1}V(H_i)\bigg)\bigg|=1
\]
and $q=m-n+1$.
\item[(iv)] There is exactly one directed path between each pair of vertices of $Q$.
\end{enumerate}
\end{lem}

Thus, (iv) shows that a cactus may be considered as a sort of analogue in a digraphs setting to trees in a simple graphs setting (although, the block graph of a cactus is not always a tree). Lemma \ref{cactuslem1} allows us to define some terms about a cactus $Q$. For $u,v\in V(Q)$, we denote by $uQv$ the unique $u - v$ directed path in $Q$. Let $K_{Q}$ denote the set formed by all the cut-vertices of $Q$. If $u\in V(Q)$ and $u$ belongs to the cycle $H$, then we write $u_{H+}$ and $u_{H-}$ for the out-neighbour and in-neighbour of $u$ in $H$. In particular, if $u\not\in K_Q$, then $u$ belongs to exactly one cycle, and we may simply write $u_+$ and $u_-$. We say that a cactus on $n$ vertices is an \emph{$(n,q)$-cactus} when it has a decomposition into $q$ cycles. Also, we always consider a cactus along with its cycle decomposition as given in (iii) of Lemma \ref{cactuslem1}. It is clear that such a decomposition is unique (up to the ordering of the cycles).

From the equation $q=m-n+1$ in (iii), we see that if $|V(H_i)|=n_i$ for $1\le i\le q$, then
\[
n_1+\dots+n_q=n+q-1.
\]
Since $n_i\ge 3$ for all $1\le i\le q$, we have $n\ge 2q+1$. It is also easy to obtain the following lemma.

\begin{lem}\label{cactuslem2}
Let $Q$ be a cactus with $q\ge 2$ blocks. Let $u\in V(H)$ and $v\in V(H')$, where $H,H'$ are two distinct cycles of $Q$. Then the unique $u-v$ and $v-u$ paths in $Q$ must be of the form $uQv=uHzQz'H'v$ and $vQu=vH'z'QzHu$, where $z\in V(H)\cap K_Q$ and $z'\in V(H')\cap K_Q$. Furthermore, no arc or internal vertex of $zQz'$ belongs to $H$ or $H'$, and likewise for $z'Qz$.
\end{lem}

\begin{proof}[Proof (sketch)]
By Lemma \ref{cactuslem1}(iii), in order to connect $u$ to $v$ in $Q$, we connect $H$ and $H'$ by a chain of intermediate cycles. We may think of $H,H'$, and these intermediate cycles as a chain of ``kissing cycles'', with $H$ and $H'$ at the two ends, and every two consecutive cycles meeting at a cut-vertex of $Q$. The vertices where two consecutive cycles meet are distinct. Now to connect $u$ to $v$, we start at $u$, traverse along $H$ until we reach the cut-vertex in $H$, say $z$. Then we traverse along the chain of intermediate cycles until we reach the cut-vertex in $H'$, say $z'$. Finally, we traverse along $H'$ until we reach $v$. Now, note that to connect $v$ to $u$, we traverse in the opposite direction, first from $v$, through $H'$ to $z'$, then the intermediate cycles in reverse order, and then through $z$, $H$, and finally to $u$. By Lemma \ref{cactuslem1}(iv), these are the unique $u-v$ and $v-u$ paths in $Q$, and they have the forms as described in the lemma.
\end{proof}

By Lemma \ref{cactuslem1}(iv), since the directed path between any pair of vertices in a cactus $Q$ is unique, we have
\[
\overset{\rightarrow}{rc}(Q)= \overset{\rightarrow}{src}(Q),\quad\overset{\rightarrow}{rvc}(Q)= \overset{\rightarrow}{srvc}(Q),\quad\textup{and}\quad\overset{\rightarrow}{\smash{t}rc}(Q)= \overset{\rightarrow}{s\smash{t}rc}(Q).
\]
Thus, it suffices to consider the parameters $\overset{\rightarrow}{rc}(Q)$, $\overset{\rightarrow}{rvc}(Q)$, and $\overset{\rightarrow}{\smash{t}rc}(Q)$. Moreover, note that again by Lemma \ref{cactuslem1}(iv), we have 
\begin{equation}
\overset{\rightarrow}{\smash{t}rc}(Q)\le\overset{\rightarrow}{rc}(Q)+\overset{\rightarrow}{rvc}(Q),\label{rc+rvceq}
\end{equation}
since a total rainbow connected colouring of $Q$ can be obtained by combining a rainbow connected colouring with $\overset{\rightarrow}{rc}(Q)$ colours, and a rainbow vertex-connected colouring with $\overset{\rightarrow}{rvc}(Q)$ additional colours.

If $Q$ is an $(n,1)$-cactus, then $Q=\rC_n$. We have $\overset{\rightarrow}{rc}(\rC_n)=n$ for $n\ge 3$. In \cite{LLLS2017}, it was proved that
\[
\overset{\rightarrow}{rvc}(\rC_n)=
\left\{
\begin{array}{ll}
n-2 & \textup{if $n=3,4$,}\\
n & \textup{if $n\ge 5$.}
\end{array}
\right.
\]
Finally, $\overset{\rightarrow}{\smash{t}rc}(\rC_n)$ is given by Theorem \ref{thm3}. From now on, we consider $(n,q)$-cactus digraphs where $q\ge 2$. In \cite{AM2017}, Alva-Samos and Montellano-Ballesteros proved the following result.

\begin{thm}\label{ASMBcactusthm1}\textup{\cite{AM2017}}
Let $Q$ be an $(n,q)$-cactus with $q\ge 2$. We have the following.
\begin{enumerate}
\item[(a)] $n-q+1\le\overset{\rightarrow}{rc}(Q)\le n-1$.
\item[(b)] $\overset{\rightarrow}{rc}(Q)=n-q+1$ if and only if $K_Q$ is independent.
\item[(c)] $\overset{\rightarrow}{rc}(Q)=n-1$ if and only if $B(Q)=P_q$ and $Q[K_Q]=\rP_{q-1}$.
\end{enumerate}
\end{thm}

They also showed that for every value $k$ in the range in (a), there exists an $(n,q)$-cactus whose rainbow connection number is equal to $k$.

Here, our aim is to prove similar results for the parameters $\overset{\rightarrow}{rvc}(Q)$ and $\overset{\rightarrow}{\smash{t}rc}(Q)$. To proceed, we first prove the following lemma. Given a digraph $D$ with a vertex-colouring $c$, we say that a vertex $u\in V(D)$ is \emph{singularly coloured} if no other vertex of $D$ has colour $c(u)$. Likewise, if $c$ is a total-colouring of $D$, we say that an element $x\in V(D)\cup A(D)$ is \emph{singularly coloured} if no other element of $V(D)\cup A(D)$ has colour $c(x)$.

\begin{lem}\label{cactuslem3}
Let $Q$ be a cactus with $q\ge 2$ blocks. Let $H$ be an end-block of $Q$, let $u$ be the unique cut-vertex of $Q$ in $H$, and $Q'=Q-V(H-u)$. Let $v$ and $w$ be the out-neighbour and in-neighbour of $u$ in $H$.
\begin{enumerate}
\item[(a)] If $c$ is a rainbow vertex-connected colouring on $Q$, then all vertices of $V(H-\{u,v,w\})$ must be singularly coloured.
\item[(b)] If $c$ is a total rainbow connected colouring on $Q$, then all elements of $V(H-\{u,v,w\})\cup A(H-\{wu,uv\})$ must be singularly coloured.
\item[(c)] We have 
\[
\overset{\rightarrow}{rvc}(Q) \ge\overset{\rightarrow}{rvc}(Q')+|V(H)|-3\textup{\emph{,\quad and}}\quad\overset{\rightarrow}{\smash{t}rc}(Q)\ge\overset{\rightarrow}{\smash{t}rc}(Q')+2|V(H)|-5.
\]
\end{enumerate}
\end{lem}

\begin{proof}
We first prove (b), from which we can easily deduce (a). We then prove (c).\\[1ex]
\indent(b) Let $c$ be a total rainbow connected colouring on $Q$. We will proceed by assuming that there are $a\in V(H-\{u,v,w\})\cup A(H-\{wu,uv\})$ and $b\in V(Q)\cup A(Q)$ with $c(a)=c(b)$. We then obtain a contradiction by finding a path $P$ whose set of arcs and internal vertices contains both $a$ and $b$, since the end-vertices of $P$ would then not be connected by a total-rainbow path in $Q$.

Firstly, let $x\in V(H-\{u,v,w\})$, and suppose that we have $r\in V(Q)$ with $c(r)=c(x)$. Suppose that $r\in V(H-u)$. Then note that either $H$ has length at least $5$, or $H$ has length $4$ and $r\in \{v,w\}$. In either case, we may let $P=x_-Hr_+$ or $r_-Hx_+$. Now, suppose that $r\in V(Q')$. Then $r$ belongs to some cycle $H'$ in $Q'$. By Lemma \ref{cactuslem2}, the unique $x-r$ and $r-x$ paths in $Q$ are $xQr=xHuQ'u'H'r$ and $rQx=rH'u'Q'uHx$, for some $u'\in V(H')\cap K_Q$. Thus, we may let $P=x_-HuQ'u'H'r_{H'+}$ or $P=r_{H'-}H'u'Q'uHx_+$. 

Next, suppose that we have $st\in A(Q)$ with $c(st)=c(x)$. If $st\in A(H)$, then since $H$ has length at least $4$, we let $P=x_-Ht$ or $P=sHx_+$. If $st\in A(Q')$, then let $H'$ be the unique cycle in $Q'$ containing $st$. By Lemma \ref{cactuslem2}, the unique $x-s$ and $s-x$ paths in $Q$ are $xQs=xHuQ'u'H's$ and $sQx=sH'u'Q'uHx$, for some $u'\in V(H')\cap K_Q$. The same condition holds for $t$ in place of $s$. Thus, we may let $P=x_-HuQ'u'H't$ or $P=sH'u'Q'uHx_+$. 

Secondly, let $xy\in A(H-\{wu,uv\})$, and suppose that we have $r\in V(Q)$ with $c(r)=c(xy)$. If $r\in V(H-u)$, then we may let $P=xHr_+$ or $P=r_-Hy$. Now, suppose that $r\in V(Q')$. Then $r$ belongs to some cycle $H'$ in $Q'$. By Lemma \ref{cactuslem2}, the unique $x-r$ and $r-x$ paths in $Q$ are $xQr=xHuQ'u'H'r$ and $rQx=rH'u'Q'uHx$, for some $u'\in V(H')\cap K_Q$. The same condition holds for $y$ in place of $x$. Thus, we may let $P=xHuQ'u'H'r_{H'+}$ or $P=r_{H'-}H'u'Q'uHy$. 

Finally, suppose that we have $st\in A(Q)$ with $c(st)=c(xy)$. If $st\in A(H)$, then we let $P=xHt$ or $P=sHy$. If $st\in A(Q')$, then let $H'$ be the unique cycle in $Q'$ containing $st$. By Lemma \ref{cactuslem2}, the unique $x-s$ and $s-x$ paths in $Q$ are $xQs=xHuQ'u'H's$ and $sQx=sH'u'Q'uHx$, for some $u'\in V(H')\cap K_Q$. The same conditions hold for 
when $x$ is replaced by $y$, or $s$ is replaced by $t$, or both. Thus, we may let $P=xHuQ'u'H't$ or $P=sH'u'Q'uHy$. 

In every case, we have found a desired path $P$, and thus we always have a contradiction.\\[1ex]
\indent(a) Let $c$ be a rainbow vertex-connected colouring on $Q$, and suppose that there exist vertices $x\in V(H-\{u,v,w\})$ and $r\in V(Q)$ with $c(r)=c(x)$. We can use exactly the same argument as in the first part of (b) to obtain a path $P$ which contains $x$ and $r$ as internal vertices. Thus $P$ is not vertex-rainbow, which is a contradiction.\\[1ex]
\indent(c) Suppose that $c$ is a rainbow vertex-connected colouring of $Q$. Then in $Q'$, we require at least $\overset{\rightarrow}{rvc}(Q')$ colours for $c$, otherwise there would exist two vertices in $Q'$ which are not connected by a vertex-rainbow path in $Q'$, and thus also in $Q$. Also by (a), the vertices of $V(H-u)$ must provide $|V(H)|-3$ additional colours for $c$. Therefore, $c$ uses at least $\overset{\rightarrow}{rvc}(Q')+|V(H)|-3$ colours, and $\overset{\rightarrow}{rvc}(Q) \ge\overset{\rightarrow}{rvc}(Q')+|V(H)|-3$. Likewise, if $c$ is a total rainbow connected colouring of $Q$, then in $Q'$, we require at least $\overset{\rightarrow}{\smash{t}rc}(Q')$ colours for $c$. By (b), the elements of $V(H-u)\cup A(H)$ must provide $2|V(H)|-5$ additional colours for $c$. Therefore, $\overset{\rightarrow}{\smash{t}rc}(Q) \ge\overset{\rightarrow}{\smash{t}rc}(Q')+2|V(H)|-5$.
\end{proof}

We are now ready to present the main results of this section. Firstly, we have the following result, which contains the rainbow vertex-connection and total rainbow connection analogues of Theorem \ref{ASMBcactusthm1}.

\begin{thm}\label{cactusthm1}
Let $Q$ be an $(n,q)$-cactus, with $q\ge 2$. We have the following.
\begin{enumerate}
\item[(a)] $n-2q+2\le\overset{\rightarrow}{rvc}(Q)\le n-2$.
\item[(b)] $2n-3q+3\le\overset{\rightarrow}{\smash{t}rc}(Q)\le 2n-3$.
\end{enumerate}
\end{thm}

\begin{proof}
(a) We first prove the lower bound by induction on $q$. For $q=2$, $Q$ consists of two cycles $H,H'$ meeting at a cut-vertex $u$. Then $u_{H+}HuH'u_{H'-}$ is a path in $Q$ with length $n-1$, and therefore $\overset{\rightarrow}{rvc}(Q)\ge n-2$. Now let $q\ge 3$, and suppose the lower bound holds for any cactus with $q-1$ blocks. Let $H$ be an end-block of $Q$, let $u$ be the unique cut-vertex of $Q$ in $H$, and let $Q'=Q-V(H-u)$. By Lemma \ref{cactuslem3}(c) and induction, we have
\begin{align*}
\overset{\rightarrow}{rvc}(Q) &\ge\overset{\rightarrow}{rvc}(Q')+|V(H)|-3\\
&\ge |V(Q')|-2(q-1)+2+|V(H)|-3\\
&=n-2q+2,
\end{align*}
since $|V(Q')|+|V(H)|=n+1$. This proves the lower bound of (a).

Now we prove the upper bound. Let $H,H'$ be two end-blocks of $Q$, and $u,u'$ be the cut-vertices of $Q$ in $H,H'$. Let $v,w$ be the out-neighbour and in-neighbour of $u$ in $H$, and similarly for $v',w'$ in $H'$. We define the vertex-colouring $c$ on $Q$ by setting $c(v)=c(v')=1$, $c(w)=c(w')=2$, and all remaining vertices are given the distinct colours $3,4,\dots,n-2$. It is easy to check that $c$ is a rainbow vertex-connected colouring for $Q$. Indeed, if $x,y\in V(Q)\setminus V(H-u)$ or $x,y\in V(Q)\setminus V(H'-u')$, then clearly $xQy$ is a vertex-rainbow path. Otherwise, we may assume that $x\in V(H-u)$ and $y\in V(H'-u')$. Then $xQy$ does not have $v$ and $w'$ as internal vertices, and thus is also vertex-rainbow. Therefore, $\overset{\rightarrow}{rvc}(Q)\le n-2$.\\[1ex]
\indent(b) The lower bound can be proved similarly by induction $q$. For $q=2$, since we have a path of length $n-1$ in $Q$, this implies $\overset{\rightarrow}{\smash{t}rc}(Q)\ge 2n-3$ by (\ref{eq}). Now let $q\ge 3$, let $H$ be an end-block of $Q$, and $Q'$ be as defined in (a). By Lemma \ref{cactuslem3}(c) and induction, we have
\begin{align*}
\overset{\rightarrow}{\smash{t}rc}(Q) &\ge\overset{\rightarrow}{\smash{t}rc}(Q')+2|V(H)|-5\\
&\ge 2|V(Q')|-3(q-1)+3+2|V(H)|-5\\
&=2n-3q+3,
\end{align*}
and the lower bound of (b) holds. 

The upper bound $\overset{\rightarrow}{\smash{t}rc}(Q)\le 2n-3$ easily follows from (\ref{rc+rvceq}), Theorem \ref{ASMBcactusthm1}(a) and part (a).
\end{proof}

We have the following corollary.
\begin{cor}\label{cactuscor1}
For every $(n,2)$-cactus $Q$, we have $\overset{\rightarrow}{rvc}(Q)=n-2$ and $\overset{\rightarrow}{\smash{t}rc}(Q)= 2n-3$.
\end{cor}

In the next result, we characterise the cactus digraphs that attain equality in each of the two lower bounds in Theorem \ref{cactusthm1}. It turns out that we have the same characterisation in both cases.

\begin{thm}\label{cactusthm2}
Let $Q$ be an $(n,q)$-cactus, with $q\ge 2$. Then the following are equivalent.
\begin{enumerate}
\item[(i)] $\overset{\rightarrow}{rvc}(Q)=n-2q+2$.
\item[(ii)] $\overset{\rightarrow}{\smash{t}rc}(Q)=2n-3q+3$.
\item[(iii)] For all $u,v\in K_Q$, we have $d(u,v)\ge 3$.
\end{enumerate}
\end{thm}

\begin{proof}
We first prove that (iii) $\Rightarrow$ (i). We use induction on $q$ to show that if $Q$ is an $(n,q)$-cactus and $d(u,v)\ge 3$ for all $u,v\in K_Q$, then we have $\overset{\rightarrow}{rvc}(Q)=n-2q+2$. For any such cactus $Q$, we may define a vertex-colouring $c$ as follows. Let $K_Q=\{u_1,\dots,u_p\}$ for some $1\le p\le q-1$. Note that for every $i\neq j$, we have $\Gamma[u_i]\cap\Gamma[u_j]=\emptyset$. Now, for $1\le i\le p$, set $c(w)=\alpha_i$ for all $w\in\Gamma^-(u_i)$, and $c(w)=\beta_i$ for all $w\in\Gamma^+(u_i)$, where the $\alpha_i$ and $\beta_i$ are some colours. Assign further distinct colours to all remaining vertices. Observe that every cycle of $Q$ is rainbow coloured.

We use induction on $q$ to show that $c$ is a rainbow vertex-connected colouring for $Q$ with $n-2q+2$ colours. The assertion holds for $q=2$, since this is included in the proof of the upper bound in Theorem \ref{cactusthm1}(a). Now suppose that $q\ge 3$ and the assertion holds for $q-1$. Let $c$ be the vertex-colouring of $Q$ as described. Let $H$ be an end-block of $Q$, let $u$ be the unique cut-vertex of $Q$ in $H$, and let $Q'=Q-V(H-u)$. Let $c'$ be the vertex-colouring of $Q'$ when $c$ is restricted to $Q'$. Then obviously we have $d_{Q'}(u,v)\ge 3$ for all $u,v\in K_{Q'}$, and it is easy to see that $c'$ is the vertex-colouring for $Q'$ of the type as described. Thus by induction, $c'$ is a rainbow vertex-connected colouring of $Q'$, using $|V(Q')|-2(q-1)+2$ colours. This means that the number of colours used by $c$ is $|V(Q')|-2(q-1)+2+|V(H)|-3=n-2q+2$, since $|V(Q')|+|V(H)|=n+1$. It remains to show that $c$ is a rainbow vertex-connected colouring for $Q$. Let $c(w)=\alpha$ for all $w\in\Gamma^-(u)$, and $c(w)=\beta$ for all $w\in\Gamma^+(u)$. Let $x,y\in V(Q)$. Then the unique $x-y$ path is vertex-rainbow if $x,y\in V(Q')$ (by induction), or $x,y\in V(H)$ (since $H$ is rainbow coloured by $c$). Now let $x\in V(Q'-u)$ and $y\in V(H-u)$. Then the unique $x-y$ path has the form $xQ'uHy$. By induction, the path $xQ'u$ is vertex-rainbow, and furthermore, by the definition of $c$, the vertices of $xQ'u-x$ have distinct colours and do not use the colour $\beta$. Since the path $uHy$ does not have an internal vertex with colour $\alpha$, it follows, again from the definition of $c$, that $xQ'uHy$ is a vertex-rainbow path. A very similar argument can be used for the case $x\in V(H-u)$ and $y\in V(Q'-u)$. Therefore by induction, $c$ is a rainbow vertex-connected colouring of $Q$ with $n-2q+2$ colours. Hence, $\overset{\rightarrow}{rvc}(Q)\le n-2q+2$, and $\overset{\rightarrow}{rvc}(Q)=n-2q+2$ by Theorem \ref{cactusthm1}(a). This completes the proof that (iii) $\Rightarrow$ (i).

Next, we prove that (i) $\Rightarrow$ (iii), and (ii) $\Rightarrow$ (iii), by using similar arguments for both implications. We use induction on $q$ to show that if $\overset{\rightarrow}{rvc}(Q)=n-2q+2$ (resp.~$\overset{\rightarrow}{\smash{t}rc}(Q)=2n-3q+3$), then $d(u,v)\ge 3$ for all $u,v\in K_Q$. The assertions trivially hold for $q=2$, since we have $\overset{\rightarrow}{rvc}(Q)=n-2$ (resp.~$\overset{\rightarrow}{\smash{t}rc}(Q)=2n-3$) by Corollary \ref{cactuscor1}, and $Q$ has only one cut-vertex. Next, we prove the assertion for $q=3$. Suppose that $Q$ has two cut-vertices $u,v$ with $d(u,v)\in\{1,2\}$. Then $Q$ consists of three cycles $H,H',H''$ where $V(H)\cap V(H')=\{u\}$, $V(H)\cap V(H'')=\{v\}$, $V(H')\cap V(H'')=\emptyset$, and $d_H(u,v)\in\{1,2\}$. The unique $v_{H''+}-u_{H'-}$ path in $Q$ is $v_{H''+}H''vHuH'u_{H'-}$ and has length at least $n-2$, so that $\overset{\rightarrow}{rvc}(Q)\ge n-3$ (resp.~$\overset{\rightarrow}{\smash{t}rc}(Q)\ge 2n-5$, by (\ref{eq})). Thus the assertions hold for $q=3$.

Now let $q\ge 4$, and suppose that the assertions hold for $q-1$. Suppose that there exist $u,v\in K_Q$ with $d(u,v)\in\{1,2\}$. Note first that there is a unique cycle $H$ containing both $u$ and $v$. Indeed, if $d(u,v)=1$, then $H$ is the unique cycle containing the arc $uv$. If $d(u,v)=2$, then we may assume that the unique $u-v$ path in $Q$ is $uwv$, where $w\not\in K_Q$ (if $w\in K_Q$, then we could consider $u,w$ instead of $u,v$). Then $H$ is the unique cycle containing the arc $uw$, and $H$ also contains $v$. Next, note that the set of the remaining cycles of $Q$ can be partitioned so that each part forms a ``branch of cycles attached to $H$''. More formally, for $z\in V(H)$, let $H(z)$ be the component of $Q-A(H)$ that contains $z$. Thus the subdigraphs $H(z)$ are the ``branches'', and every cycle other than $H$ belongs to exactly one of the $H(z)$. Now, we choose an end-block $H'$ of $Q$ as follows. If $H(u)$ has at least two cycles, then we let $H'$ be an end-block of $Q$ in $H(u)$, and likewise if $H(v)$ has at least two cycles. Otherwise, both $H(u)$ and $H(v)$ must have exactly one cycle, and since $q\ge 4$, there exists a vertex $w\in K_Q\setminus\{u,v\}$ in $H$. We let $H'$ be an end-block of $Q$ in $H(w)$. Now, let $x$ be the only cut-vertex of $Q$ in $H'$, and $Q'=Q-V(H'-x)$. By the choice of $H'$, we see that $u,v\in K_{Q'}$ and $d_{Q'}(u,v)\in\{1,2\}$.

For the rainbow vertex-connection, we have $\overset{\rightarrow}{rvc}(Q')\ge |V(Q')|-2(q-1)+3$ by induction. Therefore by Lemma \ref{cactuslem3}(c), we have
\begin{align*}
\overset{\rightarrow}{rvc}(Q) &\ge\overset{\rightarrow}{rvc}(Q')+|V(H')|-3\\
&\ge |V(Q')|-2(q-1)+3+|V(H')|-3\\
&=n-2q+3,
\end{align*}
since $|V(Q')|+|V(H')|=n+1$. By induction, we have proved that (i) $\Rightarrow$ (iii).

Similarly, for the total rainbow connection, we have $\overset{\rightarrow}{\smash{t}rc}(Q')\ge 2|V(Q')|-3(q-1)+4$ by induction. Therefore by Lemma \ref{cactuslem3}(c), we have
\begin{align*}
\overset{\rightarrow}{\smash{t}rc}(Q) &\ge\overset{\rightarrow}{\smash{t}rc}(Q')+2|V(H')|-5\\
&\ge 2|V(Q')|-3(q-1)+4+2|V(H')|-5\\
&=2n-3q+4,
\end{align*}
By induction, we have proved that (ii) $\Rightarrow$ (iii).

Finally, we can now easily show that (i) $\Rightarrow$ (ii), and this would mean that the three conditions are equivalent. Suppose that $\overset{\rightarrow}{rvc}(Q)=n-2q+2$. Since we know that (i) $\Rightarrow$ (iii), we have $d(u,v)\ge 3$ for all $u,v\in K_Q$. In particular, $K_Q$ is independent. By Theorem \ref{ASMBcactusthm1}(b), we have $\overset{\rightarrow}{rc}(Q)=n-q+1$. Thus by (\ref{rc+rvceq}), we have $\overset{\rightarrow}{\smash{t}rc}(Q) \le\overset{\rightarrow}{rc}(Q)+\overset{\rightarrow}{rvc}(Q)=2n-3q+3$. Therefore, $\overset{\rightarrow}{\smash{t}rc}(Q) =2n-3q+3$ by Theorem \ref{cactusthm1}(b), and (ii) holds.
\end{proof}

Similarly in the next result, we characterise the cactus digraphs that attain equality in each of the two upper bounds in Theorem \ref{cactusthm1}.

\begin{thm}\label{cactusthm3}
Let $Q$ be an $(n,q)$-cactus, with $q\ge 2$. Then the following are equivalent.
\begin{enumerate}
\item[(i)] $\overset{\rightarrow}{rvc}(Q)=n-2$.
\item[(ii)] $\overset{\rightarrow}{\smash{t}rc}(Q)=2n-3$.
\item[(iii)] $B(Q)=P_q$ and $Q[K_Q]=\rP_{q-1}$.
\end{enumerate}
\end{thm}

It is easy to see (as shown in \cite{AM2017}) that an equivalent way of saying $B(Q)=P_q$ and $Q[K_Q]=\rP_{q-1}$ is that $Q$ has the following structure. $Q$ consists of $q$ cycles $L_1,\dots,L_q$ and $q-1$ cut-vertices $u_1,\dots,u_{q-1}$ such that, $K_Q\cap V(L_1)=\{u_1\}$, $K_Q\cap V(L_q)=\{u_{q-1}\}$, and $K_Q\cap V(L_i)=\{u_{i-1},u_i\}$ for $2\le i\le q-1$. Moreover, $u_1u_2\cdots u_{q-1}$ is a directed path in $Q$. We shall identify $Q$ with this notation and say that such a cactus is a \emph{special path cactus}.

\begin{proof}[Proof of Theorem \ref{cactusthm3}]
(iii) $\Rightarrow$ (ii). Suppose that $B(Q)=P_q$ and $Q[K_Q]=\rP_{q-1}$. In other words, $Q$ is a special path cactus, with the notation as described above. Let $v$ be the out-neighbour of $u_{q-1}$ in $L_q$, and $w$ be the in-neighbour of $u_1$ in $L_1$. Then $vQw$ is a path of length $n-1$, and thus $\overset{\rightarrow}{\smash{t}rc}(Q)\ge 2n-3$ by (\ref{eq}). We have $\overset{\rightarrow}{\smash{t}rc}(Q)=2n-3$ by Theorem \ref{cactusthm1}(b).

(ii) $\Rightarrow$ (i). Let $\overset{\rightarrow}{\smash{t}rc}(Q)=2n-3$. By Theorem \ref{ASMBcactusthm1}(a), we have  $\overset{\rightarrow}{rc}(Q)\le n-1$. By (\ref{rc+rvceq}), we have $\overset{\rightarrow}{rvc}(Q)\ge n-2$. We have $\overset{\rightarrow}{rvc}(Q)=n-2$ by Theorem \ref{cactusthm1}(a).

(i) $\Rightarrow$ (iii). We use induction on $q$ to show that if $\overset{\rightarrow}{rvc}(Q)=n-2$, then $Q$ is a special path cactus. The case $q=2$ holds, since $\overset{\rightarrow}{rvc}(Q)=n-2$ by Corollary \ref{cactuscor1} and $Q$ is always a special path cactus. Next, we consider the case $q=3$. Then $Q$ consists of three cycles $H,H',H''$. Suppose that $Q$ is not a special path cactus. Then either $Q$ has one cut-vertex, or two cut-vertices $u,v$ with $d(u,v)\neq 1$ and $d(v,u)\neq 1$. If we have the former, or the latter with $d(u,v)\ge 3$ and $d(v,u)\ge 3$, then by Theorem \ref{cactusthm2}, we have $\overset{\rightarrow}{rvc}(Q)=n-4$. Otherwise, we may assume that $V(H)\cap V(H')=\{u\}$, $V(H)\cap V(H'')=\{v\}$, $V(H')\cap V(H'')=\emptyset$, $d(u,v)=2$ and $d(v,u)\ge 2$. Define the vertex-colouring $c$ on $Q$ where $c(w)=1$ if $w\in\Gamma^+(u)$, $c(w)=2$ if $w\in\Gamma^+(v)$, $c(w)=3$ if $w\in\{u_{H'-},v_{H''-}\}$, and all remaining vertices are given further distinct colours. It is easy to check that $c$ is a rainbow vertex-connected colouring for $Q$ with $n-3$ colours, and thus $\overset{\rightarrow}{rvc}(Q)\le n-3$. In fact, we have $\overset{\rightarrow}{rvc}(Q)=n-3$ by Theorems \ref{cactusthm1}(a) and \ref{cactusthm2}.

Now, let $q\ge 4$, and suppose that the assertion holds for $q-1$. Let $H$ be an end-block of $Q$, let $u$ be the only cut-vertex of $Q$ in $H$, and $Q'=Q-V(H-u)$. If $Q'$ is not a special path cactus, then by induction and Theorem \ref{cactusthm1}(a), we have $\overset{\rightarrow}{rvc}(Q')\le |V(Q')|-3$. We have a rainbow vertex-connected colouring for $Q$, using $\overset{\rightarrow}{rvc}(Q')+|V(H)|-1\le |V(Q')|-3+|V(H)|-1=n-3$ colours, since $|V(Q')|+|V(H)|=n+1$. Therefore, $Q'$ must be a special path cactus with at least three cycles. Using the prescribed notation, the end-blocks of $Q'$ are $L_1$ and $L_{q-1}$. Suppose that $V(H)\cap V(L_1)=\emptyset$. Then we can apply the same argument with $L_1$ instead of $H$, and conclude that $Q-V(L_1-u_1)$ is a special path cactus, and we have $u$ is the out-neighbour of $u_{q-2}$ in $L_{q-1}$. Thus $Q$ is a special path cactus. A similar argument holds if $V(H)\cap V(L_{q-1})=\emptyset$, where we would apply the same argument with $L_{q-1}$ instead of $H$.
\end{proof}

Finally in the following result, we see that, as in the case for $\overset{\rightarrow}{rc}(Q)$, every value in the range in Theorem \ref{cactusthm1}(a) can be attained by $\overset{\rightarrow}{rvc}(Q)$. However, the same is not quite true for $\overset{\rightarrow}{\smash{t}rc}(Q)$ with respect to the range in Theorem \ref{cactusthm1}(b). We will see that $\overset{\rightarrow}{\smash{t}rc}(Q)$ can attain every value in the range, except for the value of $2n-4$.

\begin{thm}\label{cactusthm4}
Let $q\ge 2$.
\begin{enumerate}
\item[(a)] Let $2\le k\le 2q-2$. For every $n$ where
\[
n\ge
\left\{
\begin{array}{ll}
2q+1 & \textup{\emph{if} $k$ \emph{is even,}}\\
2q+2 & \textup{\emph{if} $k$ \emph{is odd,}}
\end{array}
\right.
\]
there is an $(n,q)$-cactus $Q$ with $\overset{\rightarrow}{rvc}(Q)=n-2q+k$.
\item[(b)] Let $3\le k\le 3q-3$ with $k\neq 3q-4$. For every $n$ where
\[
n\ge
\left\{
\begin{array}{ll}
2q+1 & \textup{\emph{if} $k\equiv 0$ (mod $3$)\emph{,}}\\
2q+2 & \textup{\emph{if} $k\equiv 1$ (mod $3$)\emph{,}}\\
2q+3 & \textup{\emph{if} $k\equiv 2$ (mod $3$)\emph{,}}
\end{array}
\right.
\]
there is an $(n,q)$-cactus $Q$ with $\overset{\rightarrow}{\smash{t}rc}(Q)=2n-3q+k$.
\item[(c)] For every $(n,q)$-cactus $Q$, we have $\overset{\rightarrow}{\smash{t}rc}(Q)\neq 2n-4$.
\end{enumerate}
\end{thm}

\begin{proof}
For (a) and (b), we construct an $(n,q)$-cactus $Q_{n,q,\ell}$ as follows. For some $1\le \ell\le q-1$, we take $\ell$ copies of $\rC_3$, say $H_1,\dots,H_\ell$, and amalgamate them linearly. That is, we may let $H_1=uv_1w_1u$ and $H_i=v_{i-1}v_iw_iv_{i-1}$ for $2\le i\le \ell$. Then, we attach one copy of $\rC_{n-2q+2}$, say $H_{\ell+1}$, and $q-\ell-1$ further copies of $\rC_3$, say $H_{\ell+2},\dots,H_q$, at the same vertex $u$. That is, we identify one vertex from each of $H_{\ell+1},\dots,H_q$ with $u$. Let $Q_{n,q,\ell}$ be the resulting $(n,q)$-cactus, and note that $Q_{n,q,\ell}$ can be constructed for $n\ge 2q+1$. We define the path
\begin{equation}
P=v_\ell w_\ell v_{\ell-1}w_{\ell-1}\cdots v_1w_1uH_{\ell+1}u_-\label{cactusthm4eq1}
\end{equation} 

(a) Suppose first that $k=2\ell$ is even, and note that $1\le\ell\le q-1$. Let $Q=Q_{n,q,\ell}$. The path $P$ as defined in (\ref{cactusthm4eq1}) has length $n-2q+k+1$. Thus $\overset{\rightarrow}{rvc}(Q)\ge n-2q+k$. Now, consider the vertex-colouring $c$ of $Q$ where $c(z)=1$ if $z=v_\ell$ or $z\in\Gamma^+(u)\setminus\{v_1\}$; $c(z)=2$ if $z=w_\ell$ or $z\in\Gamma^-(u)\setminus\{w_1\}$; and the remaining vertices are given further distinct colours. It is easy to check that $c$ is a rainbow vertex-connected colouring for $Q$ using $n-2q+k$ colours, since no path can have two internal vertices with both having colour $1$, or both having colour $2$. Thus, $\overset{\rightarrow}{rvc}(Q)\le n-2q+k$, and $\overset{\rightarrow}{rvc}(Q)=n-2q+k$ as required.

Now, let $k=2\ell-1$ be odd, and note that $2\le\ell\le q-1$. We take $Q_{n,q,\ell}$ and replace the arc $uv_1$ with the path $uxv_1$, and $H_{\ell+1}$ with a copy of $\rC_{n-2q+1}$. Let $Q$ be the resulting cactus, and note that $Q$ can be constructed for $n\ge 2q+2$. Now the path $P$ as in (\ref{cactusthm4eq1}) has length $n-2q+k+1$, so that $\overset{\rightarrow}{rvc}(Q)\ge n-2q+k$. We consider the vertex-colouring $c$ of $Q$ where $c(z)=1$ if $z\in\Gamma^+(u)$; $c(z)=2$ if $z=w_\ell$ or $z\in\Gamma^-(u)\setminus\{w_1\}$; $c(v_\ell)=c(w_{\ell-1})=3$; and the remaining vertices are given further distinct colours. Again, it is easy to check that $c$ is a rainbow vertex-connected colouring for $Q$ using $n-2q+k$ colours. Thus, $\overset{\rightarrow}{rvc}(Q)\le n-2q+k$, and $\overset{\rightarrow}{rvc}(Q)=n-2q+k$ as required.\\[1ex]
\indent(b) For the case $k\equiv 0$ (mod $3$), let $k=3\ell$ and $Q=Q_{n,k,\ell}$. We have $1\le\ell\le q-1$. Let $Q'=\bigcup_{i=1}^{\ell+1} H_i$. Then 
\begin{equation}
\overset{\rightarrow}{\smash{t}rc}(Q)\ge\overset{\rightarrow}{\smash{t}rc}(Q')+q-\ell-1.\label{cactusthm4eq2}
\end{equation}
Indeed, (\ref{cactusthm4eq2}) can be seen with a similar argument as in the proof of Lemma \ref{cactuslem3}(c). If we have a total rainbow connected colouring $c$ for $Q$, then $Q'$ must use at least $\overset{\rightarrow}{\smash{t}rc}(Q')$ colours. Furthermore, by Lemma \ref{cactuslem3}(b), the $q-\ell-1$ arcs in $H_{\ell+2},\dots, H_q$ not incident with $u$ must be singularly coloured by $c$, and thus they provide an additional $q-\ell-1$ colours for $c$. Now, the path $P$ in (\ref{cactusthm4eq1}) has length $n-2q+2\ell+1$ in $Q'$, so that by (\ref{eq}), we have $\overset{\rightarrow}{\smash{t}rc}(Q')\ge 2(n-2q+2\ell+1)-1$, and thus
\[
\overset{\rightarrow}{\smash{t}rc}(Q)\ge [2(n-2q+2\ell+1)-1]+q-\ell-1=2n-3q+k.
\]

Now we define a total-colouring $c$ of $Q$ as follows. Let $c(v_iw_i)=c(v_iv_{i+1})=i$ for $1\le i\le\ell-1$; $c(uz)=\ell$ if $z\in\Gamma^+(u)$; $c(zu)=\ell+1$ if $z\in\Gamma^-(u)$; $c(z)=c(v_\ell)=\ell+2$ if $z\in\Gamma^+(u)\setminus\{v_1\}$; and $c(z)=c(w_\ell)=\ell+3$ if $z\in\Gamma^-(u)\setminus\{w_1\}$. All remaining arcs and vertices are given further distinct colours. Then, since the colours $1,\dots,\ell-1$ are each used twice, and the colours $\ell,\ell+1,\ell+2,\ell+3$ are each used $q-\ell+1$ times, and $|V(Q)\cup A(Q)|=2n+q-1$, we have that $c$ uses $2n+q-1-(\ell-1)-4(q-\ell)=2n-3q+k$ colours. Moreover, it is easy to check that $c$ is a total rainbow connected colouring for $Q$, since no path can have two arcs or internal vertices with the same colour. Thus, $\overset{\rightarrow}{\smash{t}rc}(Q)\le 2n-3q+k$, and $\overset{\rightarrow}{\smash{t}rc}(Q)=2n-3q+k$ as required.

Next, for the case $k\equiv 1$ (mod $3$), let $k=3\ell-2$, so that $2\le\ell\le q-1$. With such an $\ell$, we let $Q$ be the $(n,q)$-cactus as in the second part of (a). Let $Q'=\bigcup_{i=1}^{\ell+1} H_i$. The path $P$ in (\ref{cactusthm4eq1}) has length $n-2q+2\ell$ in $Q'$, so that by (\ref{eq}), we have $\overset{\rightarrow}{\smash{t}rc}(Q')\ge 2(n-2q+2\ell)-1$. Again (\ref{cactusthm4eq2}) holds, and thus
\[
\overset{\rightarrow}{\smash{t}rc}(Q)\ge [2(n-2q+2\ell)-1]+q-\ell-1=2n-3q+k.
\]

Now we define a total-colouring $c$ of $Q$ as follows. Let $c(v_iw_i)=c(v_iv_{i+1})=i$ for $1\le i\le\ell-1$; $c(xv_1)=c(w_2v_1)=\ell$; $c(w_1)=c(v_\ell)=\ell+1$; $c(uz)=\ell+2$ if $z\in\Gamma^+(u)$; $c(zu)=\ell+3$ if $z\in\Gamma^-(u)$; $c(z)=\ell+4$ if $z\in\Gamma^+(u)$; and $c(z)=c(w_\ell)=\ell+5$ if $z\in\Gamma^-(u)\setminus\{w_1\}$. All remaining arcs and vertices are given further distinct colours. Then, since the colours $1,\dots,\ell+1$ are each used twice, the colours $\ell+2,\ell+3,\ell+4,\ell+5$ are each used $q-\ell+1$ times, and $|V(Q)\cup A(Q)|=2n+q-1$, we have that $c$ uses $2n+q-1-(\ell+1)-4(q-\ell)=2n-3q+k$ colours. Again, it is easy to check that $c$ is a total rainbow connected colouring for $Q$. Thus, $\overset{\rightarrow}{\smash{t}rc}(Q)\le 2n-3q+k$, and $\overset{\rightarrow}{\smash{t}rc}(Q)=2n-3q+k$ as required.

Finally, we consider the case $k\equiv 2$ (mod $3$) with $k\neq 3q-4$. Let $k=3\ell-4$, and note that $3\le\ell\le q-1$. With such an $\ell$, we take $Q_{n,q,\ell}$ and replace the arcs $uv_1$ and $v_1v_2$ with the paths $ux_1v_1$ and $v_1x_2v_2$, and $H_{\ell+1}$ with a copy of $\rC_{n-2q}$. Let $Q$ be the resulting $(n,q)$-cactus, and note that $Q$ can be constructed for $n\ge 2q+3$. Now the path $P$ in (\ref{cactusthm4eq1}) has length $n-2q+2\ell-1$, so that $\overset{\rightarrow}{\smash{t}rc}(Q')\ge 2(n-2q+2\ell-1)-1$ by (\ref{eq}). Again (\ref{cactusthm4eq2}) holds, and thus
\[
\overset{\rightarrow}{\smash{t}rc}(Q)\ge [2(n-2q+2\ell-1)-1]+q-\ell-1=2n-3q+k.
\]

Now we define a total-colouring $c$ of $Q$ as follows. Let $c(x_2v_2)=c(w_3v_2)=1$; $c(v_iw_i)=c(v_iv_{i+1})=i$ for $2\le i\le\ell-1$; $c(x_1v_1)=c(w_2v_1)=\ell$; $c(v_1w_1)=c(v_\ell)=\ell+1$; $c(x_1)=c(w_2)=\ell+2$; $c(x_2)=c(w_\ell)=\ell+3$; $c(uz)=\ell+4$ if $z\in\Gamma^+(u)$; $c(zu)=\ell+5$ if $z\in\Gamma^-(u)$; $c(z)=c(v_1x_2)=\ell+6$ if $z\in\Gamma^+(u)\setminus\{x_1\}$; and $c(z)=\ell+7$ if $z\in\Gamma^-(u)$. All remaining arcs and vertices are given further distinct colours. Then, since the colours $1,\dots,\ell+3$ are each used twice, the colours $\ell+4,\ell+5,\ell+6,\ell+7$ are each used $q-\ell+1$ times, and $|V(Q)\cup A(Q)|=2n+q-1$, we have that $c$ uses $2n+q-1-(\ell+3)-4(q-\ell)=2n-3q+k$ colours. Again, it is easy to check that $c$ is a total rainbow connected colouring for $Q$. Thus, $\overset{\rightarrow}{\smash{t}rc}(Q)\le 2n-3q+k$, and $\overset{\rightarrow}{\smash{t}rc}(Q)=2n-3q+k$ as required.\\[1ex]
\indent(c) If $Q$ is a special path cactus, then by Theorem \ref{cactusthm3}, we have $\overset{\rightarrow}{\smash{t}rc}(Q)=2n-3$. Otherwise, by Theorems \ref{ASMBcactusthm1}, \ref{cactusthm1}, and \ref{cactusthm3}, we have $\overset{\rightarrow}{rc}(Q)\le n-2$ and $\overset{\rightarrow}{rvc}(Q)\le n-3$. Thus by (\ref{rc+rvceq}), we have $\overset{\rightarrow}{\smash{t}rc}(Q)\le (n-2)+(n-3)=2n-5$.
\end{proof}

\section*{Acknowledgements}

Hui Lei and Yongtang Shi are partially supported by National Natural Science Foundation of\, China\, (Nos.~11371021,\, 11771221),\, and\, Natural\, Science\, Foundation\, of\, Tianjin\, (No. 17JCQNJC00300). Henry Liu is partially supported by the Startup Fund of One Hundred Talent Program of SYSU. Henry Liu would also like to thank the Chern Institute of Mathematics, Nankai University, for their generous hospitality. He was able to carry out part of this research during his visit there.

The authors thank the anonymous referees for the careful reading of the manuscript.

\end{document}